\documentclass[12pt,oneside]{amsbook}

\usepackage{graphicx,wrapfig,amsthm,fancyhdr}

\swapnumbers

\usepackage[left=1.5625in, right=1in, top=1.50in, bottom=1in]{geometry}

\newtheorem{lemma}{Lemma}[chapter]
\newtheorem{proposition}[lemma]{Proposition}
\newtheorem{theorem}{Theorem}
\newtheorem{definition}[lemma]{Definition}

\newtheorem{corollary}[lemma]{Corollary}
\newcommand{\skipline}{\vspace{12pt}}

\usepackage{hyperref}

\begin{document}

    \pagenumbering{roman} \pagestyle{plain}


\begin{center}

\begin{Large}Global Solutions to the Ultra-Relativistic Euler
Equations\end{Large}\\
\skipline
By\\
\skipline
BRIAN DAVID WISSMAN\\
B.S. (University of California, Davis) 2002\\
M.A. (University of California, Davis) 2004\\ \skipline
DISSERTATION\\
\skipline
Submitted in partial satisfaction of the requirements for the degree of\\
\skipline
DOCTOR OF PHILOSOPHY\\
\skipline
in\\
\skipline
MATHEMATICS\\
\skipline
in the\\
\skipline
OFFICE OF GRADUATE STUDIES\\
\skipline
of the\\
\skipline
UNIVERSITY OF CALIFORNIA,\\
\skipline
DAVIS\\
\vspace{1in}
Approved:\\
\skipline
\rule{2.5in}{1pt}\\
\skipline
\rule{2.5in}{1pt}\\
\skipline
\rule{2.5in}{1pt}\\
\skipline
Committee in Charge\\
\skipline
2007\\

\end{center}


\renewcommand{\baselinestretch}{2}\small\normalsize


\newpage
\tableofcontents

\newpage
\section*{Abstract}

    We prove a global existence theorem for the $3\times 3$ system of
relativistic Euler equations in one spacial dimension.  It is shown
that in the ultra-relativistic limit, there is a family of equations
of state that satisfy the second law of thermodynamics for which
solutions exist globally. With this limit and equation of state,
which includes equations of state for both an ideal gas and one
dominated by radiation, the relativistic Euler equations can be
analyzed by a Nishida-type method leading to a large data existence
theorem, including the entropy and particle number evolution, using
a Glimm scheme. Our analysis uses the fact that the equations of
state are of the form $p=p(n,S)$, but whose form simplifies to
$p=a^{2}\rho$ when viewed as a function of $\rho$ alone.


\newpage
\section*{Acknowledgments}

I would like to thank my mentor, advisor and colleague Blake Temple
for his generous support throughout my time as a graduate student. I
owe much of my success to his constant advice and encouragement. I
would also like to thank my parents and my wife Carri for all their
love and support.  I could not have accomplished much without you.\\
\begin{center}
    THANK YOU!
\end{center}
\newpage

\renewcommand{\topmargin}{-.5in}
\fancypagestyle{headings}{
  \lhead{\slshape \S \thechapter.\rightmark}
  \fancyhfoffset[r]{.5in}
  \rhead{\thepage \skipline \skipline \skipline}
  \cfoot{}
  \renewcommand{\headrulewidth}{0in}
} \fancypagestyle{firstpage}{
  \lhead{}
  \fancyhfoffset[r]{.5in}
  \rhead{\thepage \skipline \skipline \skipline}
  \cfoot{}
  \renewcommand{\headrulewidth}{0in}
} \fancypagestyle{plain}{
  \lhead{}
  \fancyhfoffset[r]{.5in}
  \rhead{\thepage \skipline \skipline \skipline}
  \cfoot{}
  \renewcommand{\headrulewidth}{0in}
}

\pagestyle{headings} \pagenumbering{arabic}

\newpage
\pagestyle{headings} \pagenumbering{arabic}

    \chapter{Introduction}
    \thispagestyle{firstpage}
    \section{The Compressible Euler Equations}
The compressible Euler equations form a nonlinear system of first
order partial differential equations that models a gas as a
continuous medium. Nearly seventy years after Newton wrote down the
laws of motion in his Principia for a system of discrete particles,
$\textbf{F}=m\textbf{a}$, Euler and d'Alembert produced a linear,
continuum theory of sound waves.  These sound waves obeyed the
linear wave equation,
    \begin{displaymath}
      \square u=u_{tt}-c^{2}Div(u)=0,
    \end{displaymath}
where $c>0$ is the sound speed. Several years later, Euler wrote
down the evolution equations for the nonlinear theory of sound
waves.  Today these equations are written as
    \begin{eqnarray}\label{Classical-Euler-Multi-D}
      \rho_{t}+Div\left[\rho u\right]=0,\nonumber\\
      (\rho u)_{t}+Div\left[\rho u \otimes u +pI\right]=0,\nonumber\\
      E_{t}+Div\left[(E+p)u\right]=0,
    \end{eqnarray}
where subscripts in the independent variables denotes partial
differentiation and $Div=\partial/\partial x+\partial/\partial
y+\partial/\partial z$.  In three spacial dimensions, the
compressible Euler equations \eqref{Classical-Euler-Multi-D}, also
called Euler's equations, form a system of five equations with six
unknowns, $\rho$, $\epsilon$, $u^{i}$, and $p$, which closes when an
equation of state, $p=p(\rho,S)$, is prescribed. In the following we
will focus our study on the case of one spacial dimension.  Under
this assumption, the Euler equations reduce to a system of three
equations:
    \begin{eqnarray}\label{Classical-Euler}
      \rho_{t}+\left[\rho u\right]_{x}=0,\nonumber\\
      (\rho u)_{t}+\left[\rho u^{2} +p\right]_{x}=0,\nonumber\\
      E_{t}+\left[(E+p)u\right]_{x}=0.
    \end{eqnarray}
It is well known that even for smooth initial data, discontinuities
form in the fluid variables in the solution to the Cauchy problem in
finite time, \cite{Cour-Friderichs}. Qualitatively, the
nonlinearities in the equations cause waves to propagate at
different speeds leading to the ``breaking" of waves. See Figure
\ref{Shock-Formation}. This loss of regularity corresponds to the
emergence of shock waves.

\begin{figure}
\begin{center}
  \includegraphics[width=5in]{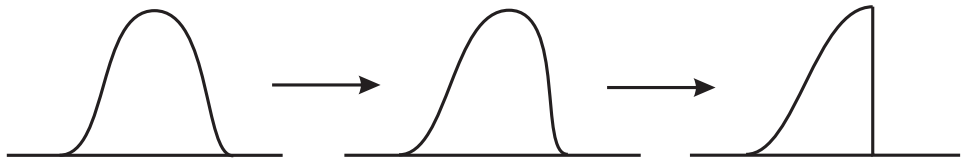}
  \caption{The ``breaking" of a wave front which produces a shock wave.}\label{Shock-Formation}
\end{center}
\end{figure}

The Euler equations are a particular example of a system of
conservation laws.  A system of conservation laws in one spacial
dimension is a first order quasi-linear system of partial
differential equations of the form
    \begin{equation}\label{ConservationLaw}
      U_{t}+F(U)_{x}=0,
    \end{equation}
where $U=(U_{1},\ldots,U_{n})$ are the conserved quantities and
$F(U)=(F_{1}(U),\ldots,F_{n}(U))$ the fluxes.  Much of the early
work on the general structure of systems of conservation laws was
set out by Lax, \cite{Lax}.  Lax's results provided the foundation
necessary for Glimm to give the first general existence theorem in
$1965$, \cite{Glimm}. Glimm's fundamental result provided a new way
to analyze shock wave interactions.  In the $1990$'s, Bressan, Liu
and Yang headed a push for the well posedness of the general
$n\times n$ Cauchy problem, \cite{Bressan-Book}.

A Nishida system is a specific class of conservation laws, which in
certain cases includes the the Euler and Relativistic Euler
equations, that allows one to prove global existence of solutions.
In particular, the shock-rarefaction curves in a Nishida system
behave nicely in the large. Nishida and Smoller were first to gave a
global, large initial data, existence proof for the compressible
Euler equations with a particular equation of state,
\cite{Nishida-Smoller}. Shortly after this, Temple extended Nishida
and Smoller's global existence result by including the entropy
evolution of the gas, \cite{Temple-Large}. More recently, Smoller
and Temple proved that under certain conditions the Relativistic
Euler equations also form a Nishida system,
\cite{Smoller-Temple-Global-Solutions-Rel-Euler}.

It should be noted that the existence theorem for a general system
of conservation laws comes at a cost; we require the initial data to
be of sufficiently small total variation. The smallness requirement
is needed because the structure of the shock-rarefaction curves can
exhibit complicated nonlinear phenomenon in the large. When
sufficiently small data is considered, the analysis can be confined
within a small region in state space in which the shock-rarefaction
curves have a canonical structure that can be exploited when
analyzing solutions.

\bigskip

    \section{The Relativistic Euler Equations}

    In $1905$, Einstein introduced the special theory of
    relativity.  Within this framework one can generalize the
    classical Euler equations to obtain equations that fit within the
    theory of relativity.

    The \textit{relativistic compressible Euler equations} in one spatial dimension form a system
    of three equations,
    \begin{eqnarray}\label{System-Divergence}
      (u^{\alpha}n),_{\alpha} & = & 0,\nonumber\\
      T^{\alpha\beta},_{\alpha} & = & 0,
      \phantom{44}\beta=0,1,
    \end{eqnarray}
    where $T^{\alpha\beta}$ is the stress energy tensor for a
    perfect fluid,
    \begin{displaymath}
      T^{\alpha\beta}=(\rho+p)u^{\alpha}u^{\beta}+p\eta^{\alpha\beta},
    \end{displaymath}
and the subscript ``$,\alpha$" denotes partial differentiation with
respect to the coordinate $x^{\alpha}$.  We will use Einstein's
summation convention where repeated up-down indices are summed and
adopt the following notation:
\begin{eqnarray}
  u^{\alpha} & \phantom{44444444444} & \textrm{Components of the
  }2\textrm{-Velocity}\nonumber\\
  \rho & & \textrm{Proper Rest Energy Density}\nonumber\\
  p & & \textrm{Pressure}\nonumber\\
  \epsilon & & \textrm{Specific Internal Energy}\nonumber\\
  n & & \textrm{Baryon Number}\nonumber\\
  S & & \textrm{Specific Entropy}\nonumber\\
  T & & \textrm{Temperature}\nonumber
\end{eqnarray} The components of the Minkowski metric $\eta^{\alpha\beta}$ are
given by
\begin{displaymath}
    \eta^{\alpha\beta}=\left(
                          \begin{array}{rc}
                            -1 & 0 \\
                            0 & 1 \\
                          \end{array}
                        \right).
\end{displaymath}
For convenience, we will also use units where the speed of light is
unit, $c=1$. The proper energy density, $\rho$, is related to the
particle number density and the internal energy by
$\rho=n(1+\epsilon)$, \cite{Weinberg}. This equation is the sum of
the rest mass energy $nc^{2}=n$ and the internal energy $n\epsilon$.
Furthermore, thermodynamics provides a functional relationship
between the quantities, $\epsilon$, $T$, $S$, $p$ and $n$.  This
relationship is given by the second law of thermodynamics,
\cite{Cour-Friderichs}:
\begin{equation}\label{Thermodynamics}
    d\epsilon=TdS+\frac{p}{n^{2}}dn.
\end{equation}

The relativistic Euler equations \eqref{System-Divergence} can be
written as a system of conservation laws by choosing a particular
Lorentz frame and writing the instantaneous worldline trajectory of
the fluid, $u^{\alpha}$, in terms of the classical velocity $v$. The
components of $(u^{0},u^{1})$ are proportional to the vector $(1,v)$
and is of unit length according to the inner-product defined by the
metric $\eta$.  From this we find the components $u^{\alpha}$ are
related to $v$ by
\begin{displaymath}
  \left(u^{0},u^{1}\right)=\left(\frac{1}{\sqrt{1-v^{2}}},\frac{v}{\sqrt{1-v^{2}}}\right).
\end{displaymath}
Using this, the first equation is equivalent to
\begin{displaymath}
  \frac{\partial}{\partial{t}}\left(\frac{n}{\sqrt{1-v^{2}}}\right)+\frac{\partial}{\partial
  x}\left(\frac{nv}{\sqrt{1-v^{2}}}\right)=0.
\end{displaymath}
The second and third equations in \eqref{System-Divergence} can also
be rewritten. With $\beta=0$ we find
$T^{0\alpha}_{\phantom{00},\alpha}=0$ gives
\begin{displaymath}
  \frac{\partial}{\partial
  t}\left((\rho+p)\frac{1}{1-v^{2}}-p\right)+\frac{\partial}{\partial
  x}\left((\rho+p)\frac{v}{1-v^{2}}\right)=0
\end{displaymath}
and with $\beta=1$, $T^{1\alpha}_{\phantom{00},\alpha}=0$ gives
\begin{displaymath}
  \frac{\partial}{\partial
  t}\left((\rho+p)\frac{v}{1-v^{2}}\right)+\frac{\partial}{\partial
  x}\left((\rho+p)\frac{v^{2}}{1-v^{2}}+p\right)=0.
\end{displaymath}
Simplifying the terms inside, we can write the system
\eqref{System-Divergence} as the system of conservation laws,
\begin{equation}\label{System-Conservation}
  U_{t}+F(U)_{x}=0,
\end{equation}
where,
    \begin{equation}\label{Conserved-Variables}
      U=\left(\frac{n}{\sqrt{1-v^{2}}}\phantom{3}\!,\left(\rho+p\right)\frac{v}{1-v^{2}}\phantom{3}\!,\left(\rho+p\right)\frac{v^{2}}{1-v^{2}}+\rho\right)
    \end{equation}
    and
    \begin{equation}\label{Flux-Function}
      F(U)=\left(\frac{nv}{\sqrt{1-v^{2}}}\phantom{3}\!,\left(\rho+p\right)\frac{v^{2}}{1-v^{2}}+p\phantom{3}\!,\left(\rho+p\right)\frac{v}{1-v^{2}}\right).
    \end{equation}

It is interesting to note that the relativistic Euler equations are
indeed a generalization of the classical Newtonian equations of
hydrodynamics \eqref{Classical-Euler}.  To see this we view
\eqref{System-Conservation} under the assumptions of a classical
fluid; fluid velocities are small compared to the speed of light and
the pressure is dominated by the rest mass.  More specifically, we
assume $|v|\ll 1$ and $p/\rho\ll 1$.  Under these assumptions the
equations $T^{\alpha\beta}_{\phantom{33},\alpha}=0$ become the
equations of motion of a classical gas:
\begin{displaymath}
  \frac{\partial}{\partial t}\rho+\frac{\partial}{\partial x}(\rho
  v)=0
\end{displaymath}
and
\begin{displaymath}
    \frac{\partial}{\partial t}(\rho v)+\frac{\partial}{\partial
    x}(\rho v^{2}+p)=0.
\end{displaymath}
Notice that the density of the fluid in the classical Euler
equations is now replaced by the proper mass-energy density.  The
new variable $n$ is used for conservation of particle number.

Nearly all terrestrial phenomenon falls into the classical,
Newtonian case.  In a hurricane, for example, wind speeds may reach
speeds of $90m/s$.  However, this velocity is insignificant when
compared to the speed of light,
\begin{displaymath}
  |v|\sim 90m/s \sim 3\times10^{-7}c=10^{-7}\ll 1.
\end{displaymath}
Furthermore, the pressure to mass density ratio, $p/\rho$, can be
shown to be of the order of $10^{-12}$, \cite{Gravitiation}. In this
situation the classical Euler equations would certainly suffice.

It is clear from the last example that even in seemingly extreme
situations on Earth, they are far from relativistic events.  We must
look to the cosmos to find examples where a gas has a high enough
pressure to make $p/\rho$ non-negligible and sufficiently high
velocity to make the relativistic correction terms such as
$1/\sqrt{1-v^{2}}$ important to the gas' evolution. These situations
arise in astrophysical events such as gamma-ray bursts, solar flares
and in remnants of supernovas.   The relativistic Euler equations
are also used in modeling the early universe, \cite{Wald-GR}.

Like the classical Euler equations, the relativistic Euler equations
are not closed; an equation of state relating thermodynamic
variables is needed to close the system.  This choice of equation of
state changes the characteristics of the evolution of the gas and
has a significant effect on the complexity of its analysis. A
natural equation of state for a gas is one satisfying the ideal gas
law and whose internal energy is proportional to its temperature.
Using the second law of thermodynamics, one finds the relation
\begin{equation}\label{Polytropic-EOS}
  \epsilon(n,S)=e^{\frac{\gamma-1}{R}S}n^{\gamma-1},
\end{equation}
which for some constant $\gamma>1$ is called a polytropic equation
of state. A polytropic equation of state is typically used to model
air in the classical sense with $\gamma\approx 1.4$. It is known
that using this equation of state vacuums may form in a solution to
\eqref{Classical-Euler} and \eqref{System-Divergence} when
velocities and densities are sufficiently large to completely void a
region of matter. Vacuums pose problems in the standard estimating
techniques and at this point prevents one from obtaining large data
existence theorems, \cite{Smoller}.

A class of equations of state one typically encounters which still
include most desirable dynamics are called barotropic, given by
$p=p(\rho)$. The class of equations of state,
$p=a^{2}\rho^{\gamma}$, for $1< \gamma <2$ are barotropic and are
used in astrophysical modeling.  In this case $0<a$ is constant,
\cite{Anile-RelFluids}.

If one considers \eqref{p-a2rho}, the limiting case of the
barotropic equation of state $p=a^{2}\rho^{\gamma}$ when $\gamma=1$,
the system \eqref{System-Conservation} contains special properties;
in this limit one can prove global solutions exist for initial data
with arbitrarily large, but finite, total variation,
\cite{Smoller-Temple-Global-Solutions-Rel-Euler}. Moreover, vacuums
do not form in the solution.
\begin{equation}\label{p-a2rho}
  p=a^{2}\rho
\end{equation}

In this thesis we will extend these results to prove large data
existence theorem for an ultra-relativistic gas with an equation of
states of the form
    \begin{equation}\label{EOS-Family}
        \epsilon(n,S)=A(S)n^{\gamma-1},
    \end{equation}
where the function $A$ satisfies the following conditions:
\begin{equation}\label{A1}
  A:\mathbb{R}^{+}\rightarrow\mathbb{R}^{+},\tag{A1}
\end{equation}
\begin{equation}\label{A2}
  A\in\mathcal{C}^{1}(\mathbb{R}^{+}),\tag{A2}
\end{equation}
\begin{equation}\label{A3}
  A'(S)>0\phantom{3}\mathrm{for}\phantom{3}S>0.\tag{A3}
\end{equation}
The family \eqref{EOS-Family} includes equations of state for a
polytropic gas \eqref{Polytropic-EOS} and one dominated by radiation
satisfying
\begin{equation}\label{Radiation-EOS}
        \epsilon(n,S)=\frac{a_{R}T^{\frac{\gamma}{\gamma-1}}}{n}.
\end{equation}

Using the relation $\rho=n(1+\epsilon)$, the equations of state
\eqref{EOS-Family} do not reduce to \eqref{p-a2rho}. However, they
do in the ultra-relativistic limit. For the ultra-relativistic
limit, we assume the internal energy dominates the rest mass energy;
in other words, $\rho=n\epsilon$. Under this assumption, an equation
of state of the form \eqref{EOS-Family} reduces to an equation of
state of the form \eqref{p-a2rho} with $a^{2}=(\gamma-1)$. We take
advantage of the fact that in this limit the pressure is still a
function of $n$ and $S$, but whose form reduces to \eqref{p-a2rho}
when viewed as a function of $\rho$ alone. This model now allows one
to find the entropy and particle number density evolution of the gas
and still take advantage of the simplifying effects of an equation
of state of the form \eqref{p-a2rho}.

The particular equation of state \eqref{Radiation-EOS} is also used
to model massless thermal radiation.  In this case the
ultra-relativistic assumption is not needed since the mass-energy in
$\rho$ drops out, leaving only the internal energy. In particular
for $\gamma=4/3$, the radiation dominated equation of state is used
to model the early universe, because this radiation has been
predicted to make the dominant energy contribution, \cite{Wald-GR}.
In either situation, massless particles or in the ultra-relativistic
limit, we still have an equation of state of the form
\eqref{p-a2rho}.

It is interesting that for the classical Euler equations there is
only one way to assign an entropy profile to a gas with an equation
of state of the form \eqref{p-a2rho}.  This equation of state is
given by
\begin{equation}\label{Classical-Limit-Ideal}
  \epsilon(\rho,S)=a^{2}\ln\left(\rho \right)+\frac{a^{2}S}{R}+C,
\end{equation}
for constants $a>0$, $C>0$.  A global existence theorem for the
classical Euler equations with this equation of state was given by
Temple in \cite{Temple-Large}.

\bigskip

\section{Statement of Main Theorem}

    The goal of this paper is to prove the following:
    \begin{theorem}\label{Main-Theorem}
        Let $\rho_{0}(x)$, $v_{0}(x)$ and $S_{0}(x)$ be arbitrary initial data
        satisfying, $\rho_{0}(x)>0$, $-1<v_{0}(x)<1$ and
        $S_{0}(x)>0$.  Let $\Sigma=\ln\left[A(S)\right]$ for $\epsilon(n,S)=A(S)n^{\gamma-1}$, $1<\gamma <2$, and $A$ satisfying
        \eqref{A1}, \eqref{A2} and \eqref{A3}.  Suppose further that
        \begin{equation}\label{Initial-Bound-Entropy}
          Var\{\Sigma_{0}(\cdot)\}<\infty,
        \end{equation}
        \begin{equation}\label{Initial-Bound-Density}
          Var\{\ln(\rho_{0}(\cdot))\}<\infty,
        \end{equation}
        and
        \begin{equation}\label{Initial-Bound-Velocity}
          Var\Bigg\{\ln\left(\frac{1+v_{0}(\cdot)}{1-v_{0}(\cdot)}\right)\Bigg\}<\infty.
        \end{equation}
        Then there exists a bounded weak solution $\{\rho(x,t),v(x,t),S(x,t)\}$ to
        \eqref{System-Conservation} in the Ultra-Relativistic limit,
        satisfying
        \begin{equation}\label{Bound-Enrtopy}
          Var\{\Sigma(\cdot,t)\}< N,
        \end{equation}
        \begin{equation}\label{Bound-Density}
          Var\{\ln(\rho(\cdot,t))\}< N,
        \end{equation}
        and
        \begin{equation}\label{Bound-Velocity}
            Var\Bigg\{\ln\left(\frac{1+v(\cdot,t)}{1-v(\cdot,t)}\right)\Bigg\}<N,
        \end{equation}
        where $N$ is a constant depending only on the initial variation
        bounds in \eqref{Initial-Bound-Entropy}, \eqref{Initial-Bound-Density}, and
        \eqref{Initial-Bound-Velocity}.
    \end{theorem}

It should be noted that Theorem \ref{Main-Theorem} is a
generalization of the work by Smoller and Temple in
\cite{Smoller-Temple-Global-Solutions-Rel-Euler} that includes the
entropy evolution.  In other words, in this model we are able to
prove global solutions exist including a physically relevant entropy
and particle number density profile. Smoller and Temple found that
the relativistic Euler equations with equation of state
\eqref{p-a2rho} possessed the property that after each elementary
wave interaction in a Glimm scheme, $Var\{\ln(\rho)\}$ is
non-increasing. This functional, introduced by Liu, is used as a
replacement for the quadratic potential in Glimm's original
analysis, which can be used to show that
\eqref{Initial-Bound-Density} and \eqref{Initial-Bound-Velocity}
implies \eqref{Bound-Density} and \eqref{Bound-Velocity}.
Considering the ultra-relativistic limit, the solutions of Riemann
problems are independent of the value of $S$, enabling one to solve
for the intermediate state in the projected state space and place a
corresponding entropy wave between them.

In \cite{Smoller-Temple-Global-Solutions-Rel-Euler} it is shown that
for an equation of state of the form \eqref{p-a2rho}, the shock
curves are translationally invariant in the plane of Riemann
invariants. In our case, this property continues to hold under
certain coordinate changes in the three dimensional non-projected
state space for an equation of state of the form \eqref{EOS-Family}.
This can be viewed as the relativistic analogue of the large data
existence result in \cite{Temple-Large} with a family of distinct
entropy profiles.

The main part of the analysis is showing that $Var\{S\}$ is bounded
in our approximate solutions.  We extend the analysis by Smoller and
Temple for the ultra-relativistic regime with equation of state
given by \eqref{EOS-Family}, by utilizing the geometry of the shock
curves in the space of Riemann invariants. If we only considered the
variation of $S$ across shock waves, we find that $Var\{S\}$ is
uniformly bounded by $Var\{\ln(\rho)\}$ for a polytropic equation of
state. However, across the linearly degenerate entropy waves, there
is no change in pressure, and hence no jump in proper energy density
by \eqref{p-a2rho}.  Thus, another method must be employed to
estimate the strengths of these jumps. For a gas dominated by
radiation or for a general equation of state of the form
\eqref{EOS-Family}, the situation seems more dire as the change in
entropy across a shock depends on the initial entropy value.  It is
not known \textit{a priori} that this dependence does not lead to
blow-up in the variation in $S$.

Furthermore, in certain elementary wave interactions, $Var\{S\}$ may
actually increase while $Var\{\ln(\rho)\}$ remains invariant.
Complicating matters, using $\Delta\ln(\rho)$ as the definition of
wave strengths increases the technicality of the entropy wave
estimates. For example, after the interaction of two shocks of the
same family, the entropy change across the new shock may be less
than the sum of the jumps across the two preceding shock waves. This
happens because the new shock wave has strength less than the sum of
the two previous. In other words when two shock waves combine, the
strengths are not simply additive, but the new wave strength is
strictly less than the simple sum of the incoming shock strengths.
It follows that under certain circumstances the change in entropy
across the new single shock may be less than the sum of the entropy
jumps across the approaching shocks.

To alleviate these technicalities, we propose a more classical
approach by using the change of Riemann invariants as a measure of
wave strength. More specifically, the strength of a $1(3)-$shock is
determined by the change in the first$($third$)$-Riemann invariant
and the contact discontinuity by the change in entropy or a specific
change in a function of the entropy. Using the change in Riemann
invariants as a measure of wave strength for a Nishida system was
used to prove existence of solutions in \cite{Liu},
\cite{Nishida-Smoller} and \cite{Temple-Large}.  Under this regime,
wave strengths are now additive and the sum of all the strengths of
shock waves is shown to be non-increasing in time. Moreover, the
wave interaction estimates can be analyzed as in the classical case.
In conclusion, using $\Delta\ln(\rho)$ as a measure of wave strength
dramatically simplifies the interaction estimates for the nonlinear
waves, but complicates the problem dealing with entropy.

In summary, we show there exists a family of equations of state,
which include the case of a polytropic and radiation dominated gas,
that one can use and obtain a global existence theorem. These
equations of state allow one to also calculate the entropy and
particle density associated with the gas. This is in contrast with
the classical case where there is one equation of state with the
same properties corresponding to very heavy molecules.

\skipline

\noindent The rest of this paper is outlined as follows:

In Chapter \ref{Chapter-Relativistic-Gas-Dynamics}, we give a
detailed analysis of the structure of simple wave solutions of
\eqref{System-Conservation}.  Using these properties, we prove
global existence of solutions to Riemann problems.  Furthermore, we
obtain \textit{a priori} wave interaction estimates which will be
used to produce estimates on approximate solutions constructed using
a Glimm scheme in Chapter \ref{Glimm-Difference-Scheme}.

In Chapter \ref{Glimm-Difference-Scheme} we give an overview of the
Glimm difference scheme and prove estimates on the approximate
solutions obtained for system \eqref{System-Conservation}.  Chapter
\ref{Existence} contains the proof of our main theorem.

    \chapter{Relativistic Gas
    Dynamics}\label{Chapter-Relativistic-Gas-Dynamics}
    \thispagestyle{firstpage}
    \section{Gas Dynamics}
    We consider a gas where the proper energy density and pressure
satisfy the relationship \eqref{p-a2rho}. Causality restricts the
sound speed $c_{s}=\sqrt{dp / d\rho}=a$ to be less than unity. Under
assumption \eqref{p-a2rho}, the system \eqref{System-Conservation}
decouples so that we may solve for two variables first, then solve
for the third afterward. In this section, we will show in the domain
$\rho>0$, $-1<v<1$, and $S>0$, Riemann problems are globally
solvable and their general structure consists of two waves separated
by a jump in entropy traveling with the fluid. We then discuss wave
interaction estimates which will allow us to prove global existence
of solutions using a Glimm scheme in Chapter \ref{Existence}. Our
analysis uses the special geometry of the shock and rarefaction
curves in the space of Riemann invariants.

To begin, we will compute the eigenvalues and eigenvectors
associated with the relativistic Euler equations
\eqref{System-Conservation}. In order to simplify this process, we
will exchange the first equation, conservation of particle number,
with the equivalent equation that says that entropy is constant
along flow lines.  We note that this equation holds for continuous
solutions, but fails when shock waves form since entropy increases
across shocks, \cite{Cour-Friderichs}.


\begin{proposition}
  For smooth solutions of \eqref{System-Conservation}, the following supplemental equation
  holds:
  \begin{equation}\label{Entropy-Div0}
    u^{\alpha}S_{,\alpha}=0.
  \end{equation}
  More specifically after choosing a particular Lorentz frame,
  \begin{equation}\label{Entropy-Conserved}
    S_{t}+vS_{x}=0.
  \end{equation}
\end{proposition}
\begin{proof}
We show that conservation of energy and momentum is equivalent to
continuous flow being adiabatic, i.e. \eqref{Entropy-Conserved}.  We
take the stress-energy tensor of a perfect fluid,

    \begin{eqnarray}
      T^{\alpha\beta} & = & n\left(1+\epsilon+\frac{p}{n}\right)u^{\alpha}u^{\beta}+p\eta^{\alpha\beta}\nonumber,\\
                      & = & n\omega u^{\alpha}u^{\beta}+p\eta^{\alpha\beta},\nonumber
    \end{eqnarray}
    with $\omega=\left(1+\epsilon+\frac{p}{n}\right)$ for convenience.  Then conservation of energy-momentum equation,
    $T^{\alpha\beta}_{\phantom{\alpha\beta},\beta}=0$, is given by

    \begin{eqnarray}\label{con-energy-momentum}
      0 & = & T^{\alpha\beta}_{\phantom{\alpha\beta},\beta},\nonumber\\
        & = & \left(n\omega u^{\alpha}u^{\beta}\right)_{,\beta}+p_{,\beta}\eta^{\alpha\beta},\nonumber\\
        & = & n u^{\beta}\left(\omega u^{\alpha}\right)_{,\beta}+p_{,\beta}\eta^{\alpha\beta}.
    \end{eqnarray}
    where conservation of particle number, $\left(n u^{\beta}\right)_{,\beta}=0$, is used in the last step.
    Multiplying \eqref{con-energy-momentum} by $-u_{\alpha}$ and summing we find,
    \begin{equation}\label{con-energy-momentum2}
      0=-u_{\alpha}T^{\alpha\beta}_{\phantom{\alpha\beta},\beta}=-n u^{\beta}\left(\omega
      u^{\alpha}\right)_{,\beta}u_{\alpha}-p_{,\beta}u^{\beta}.
    \end{equation}
    To simplify this expression, we claim,
    \begin{displaymath}
      u_{\alpha}\left(\omega u^{\alpha}\right)_{,\beta}=-\omega_{,\beta}.
    \end{displaymath}
    Indeed,
    \begin{displaymath}
      u_{\alpha}\left(\omega
      u^{\alpha}\right)_{,\beta}=u_{\alpha}u^{\alpha}\omega_{,\beta}+\left(u_{\alpha}u^{\alpha}_{\phantom{\alpha},\beta}\right)\omega=-\omega_{,\beta},
    \end{displaymath}
    where the second term, $\left(u_{\alpha}u^{\alpha}_{\phantom{\alpha},\beta}\right)\omega$, vanishes because
    \begin{displaymath}
      0=\left(u_{\alpha}u^{\alpha}\right)_{,\beta}=2u_{\alpha}u^{\alpha}_{\phantom{\alpha},\beta}.
    \end{displaymath}
    Thus, \eqref{con-energy-momentum2} now reads,
    \begin{eqnarray}
      0 & = & n\omega_{,\beta}u^{\beta}-p_{,\beta}u^{\beta},\nonumber\\
        & = & \left(\epsilon_{,\beta}u^{\beta}+\left(\frac{p}{n}\right)_{\!\!,\beta}u^{\beta}\right)-p_{,\beta}u^{\beta},\nonumber\\
        & = & n\left(\epsilon_{,\beta}u^{\beta}+\left(\frac{1}{n}\right)p_{,\beta}u^{\beta}+p\left(\frac{1}{n}\right)_{\!\!,\beta}u^{\beta}\right)-p_{,\beta}u^{\beta},\nonumber\\
        & = & n\left(\epsilon_{,\beta}+p\left(\frac{1}{n}\right)_{\!\!,\beta}\right)u^{\beta},\nonumber\\
        & = & n T S_{,\beta}u^{\beta}.\nonumber
    \end{eqnarray}
    The last step follows from the second law of thermodynamics.
    Since $n,T\neq 0$ we conclude, $u^{\beta}S_{,\beta}=0$.
    Furthermore, after choosing a particular frame of reference and
    replacing the worldline trajectory with
    \begin{displaymath}
      u=\left(\frac{1}{\sqrt{1-v^{2}}},\frac{v}{\sqrt{1-v^{2}}}\right),
    \end{displaymath}
    we get
    \begin{displaymath}
      \frac{1}{\sqrt{1-v^{2}}}S_{t}+\frac{v}{\sqrt{1-v^{2}}}S_{x}=0.
    \end{displaymath}
    In particular, since $\frac{1}{\sqrt{1-v^{2}}}\neq 0$,
    \eqref{Entropy-Conserved} holds.
\end{proof}

It is interesting to note that \eqref{Entropy-Div0} continues to
hold in curved spacetimes within general relativity.  For this case,
differentiation is replaced by covariant differentiation.

In order to solve the Riemann problem by a series of simple waves,
we need to know that the corresponding wave speeds are distinct.  If
this is the case the system is strictly hyperbolic.

\begin{definition}  We call a system of conservation laws
\eqref{ConservationLaw} \textbf{Strictly Hyperbolic} in an open
connected subset $U\subseteq \mathbb{R}^{n}$ if at each point $u\in
U$, $dF$ has $n$ real distinct eigenvalues,
$\{\lambda_{i}(u)\}_{i=1}^{n}$, such that
\begin{displaymath}
  \lambda_{1}(u)<\ldots <\lambda_{n}(u).
\end{displaymath}
\end{definition}

Since a strictly hyperbolic system has $n$ distinct eigenvalues, the
corresponding eigenvectors form a basis at every point in $U$. Along
with strict hyperbolicity, we require one more assumption on the
eigenvector-eigenvalue pairs;  the corresponding eigenvalues are
either constant or monotonically increasing or decreasing along the
integral curves determined by the eigenvectors.

\begin{definition}
Let $\left\{(\lambda_{i}(u),R_{i}(u))\right\}_{i=1}^{n}$ be the
eigenvalue-eigenvector pairs associated with $dF$ for a strictly
hyperbolic conservation law in an open connected subset $U\subseteq
\mathbb{R}^{n}$ with $\lambda_{1}(u)<\ldots <\lambda_{n}(u)$. We
call the $\textrm{i}^{\textrm{th}}$ characteristic field
\textbf{Genuinely Non-Linear} in $U$ if for all $u\in U$,
\begin{displaymath}
  R_{i}(u)\cdot\nabla\lambda_{i}(u)\neq 0,
\end{displaymath}
and \textbf{Linearly Degenerate} if for all $u\in U$,
\begin{displaymath}
  R_{i}(u)\cdot\nabla\lambda_{i}(u)= 0.
\end{displaymath}
\end{definition}
In the following proposition we characterize the three eigenclasses
of the system \eqref{System-Conservation}.

\begin{proposition}  Let $p=a^{2}\rho$ with $0<a<1$.  Then the system
$\eqref{System-Conservation}$ is strictly hyperbolic at $(\rho,v,S)$
for $\rho>0$, $-1<v<1$ and $S>0$.  Furthermore, the first and third
characteristic fields are genuinely non-linear and the second
linearly degenerate.
\end{proposition}
\begin{proof}

  Equivalent systems of equations possess the same eigenvalues, so we will
  replace the conservation of particle number equation with the
  equivalent equation \eqref{Entropy-Conserved}.  Since the flux
  functions \eqref{Flux-Function} are complicated implicit functions of
  the conserved variables \eqref{Conserved-Variables}, our plan
  is to rewrite the conservation laws \eqref{System-Conservation}
  as
  \begin{displaymath}
    \omega_{t}+G(\omega)\omega_{x}=0,
  \end{displaymath}
  where $\omega=(\rho,v,S)^{T}$, then calculate the eigenvalues and
  eigenvectors in terms of these variables.  To do this we rewrite
  \eqref{System-Conservation} using the chain rule as
  \begin{displaymath}
    A(\omega)\omega_{t}+B(\omega)\omega_{x}=0,
  \end{displaymath}
  then find $G(\omega)$ by multiplying on the left by $A^{-1}$ to get
  \begin{displaymath}
    \omega_{t}+\left[A^{-1}B\right](\omega)\omega_{x}=0.
  \end{displaymath}
  By the chain rule,
  \begin{displaymath}
    A(\omega)=\left[
         \begin{array}{ccc}
           0 & 0 & 1 \\
           (a^{2}+1)\frac{v}{1-v^{2}} & (a^{2}+1)\rho\frac{1+v^{2}}{(1-v^{2})^{2}} & 0 \\
           (a^{2}+1)\frac{v^{2}}{1-v^{2}}+1 & (a^{2}+1)\rho\frac{2v}{(1-v^{2})^{2}} & 0 \\
         \end{array}
       \right]
  \end{displaymath}
  and
  \begin{displaymath}
    B(\omega)=\left[
         \begin{array}{ccc}
           0 & 0 & v \\
           (a^{2}+1)\frac{v^{2}}{1-v^{2}+a^{2}} & (a^{2}+1)\rho\frac{2v}{(1-v^{2})^{2}} & 0 \\
           (a^{2}+1)\frac{v}{1-v^{2}} & (a^{2}+1)\rho\frac{1+v^{2}}{(1-v^{2})^{2}} & 0 \\
         \end{array}
       \right].
  \end{displaymath}
  Note that $A(\omega)$ is invertible because for $-1<v<1$ and $\rho>0$,
  \begin{displaymath}
    Det[A(\omega)]=\frac{(1+a^{2})(a^{2}v^{2}-1)}{(1-v^{2})^{2}}\rho\neq
    0.
  \end{displaymath}
  After some work we get
  \begin{displaymath}
    A^{-1}(\omega)=\left[
         \begin{array}{ccc}
           0 & \frac{2v}{a^{2}v^{2}-1} & \frac{1+v^{2}}{1-a^{2}v^{2}}\\
           0 & \frac{(1-v^{2})(1+a^{2}v^{2}}{(a^{2}+1)\rho(1-a^{2}v^{2})} & \frac{v^{3}-v}{\rho(1-a^{2}v^{2})}\\
           1 & 0 & 0 \\
         \end{array}
       \right].
  \end{displaymath}
  Therefore,
  \begin{displaymath}
    G(\omega)=\left[A^{-1}B\right](\omega)=\left[
         \begin{array}{ccc}
           \frac{(a^{2}-1)v}{a^{2}v^{2}-1} & \frac{(a^{2}+1)\rho}{1-a^{2}v^{2}} & 0 \\
           \frac{a^{2}(1-v^{2})^{2}}{(a^{2}+1)\rho(a^{2}v^{2}-1)} & \frac{(a^{2}-1)v}{a^{2}v^{2}-1} & 0 \\
           0 & 0 & v \\
         \end{array}
       \right].
  \end{displaymath}
  We look for the roots of the characteristic polynomial,
  \begin{displaymath}
    0=\textrm{Det}\left[G(\omega)-\lambda
    I\right]=\frac{(v-\lambda)(\lambda (-1+av)-a+v)(-a-v+\lambda
    (1+av))}{(av-1)(1+av)}.
  \end{displaymath}
  There are three values of $\lambda$ that make the numerator zero,
  \begin{equation}\label{Characteristic-Speeds}
    \lambda_{1}=\frac{v-a}{1-va}, \phantom{3333} \lambda_{2}=v,
    \phantom{3333} \lambda_{3}=\frac{v+a}{1+va}.
  \end{equation}

  We show for $0<a<1$ and $-1<v<1$,
  \begin{displaymath}
    \lambda_{1}<\lambda_{2}<\lambda_{3}.
  \end{displaymath}
  Indeed,
  \begin{displaymath}
    v^{2}<1 \Longleftrightarrow -av^{2}>-a \Longleftrightarrow
    v-av^{2}>v-a.
  \end{displaymath}
  By the restrictions on $v$ and $a$, $(1-av)>0$ and thus,
  \begin{displaymath}
    v(1-av)>v-a \Longleftrightarrow v>\frac{v-a}{1-av}.
  \end{displaymath}
  Showing $v<(v+a)/(1+va)$ is similar, we omit the details.  We
  conclude that for $\rho>0$, $-1<v<1$, and $S>0$ the system
  \eqref{System-Conservation} is strictly hyperbolic.

  Now, we show that the first and third
  characteristic fields are genuinely nonlinear and the second is linearly
  degenerate.  To do this we need to find the eigenvectors of $G(\omega)$.
  For $\lambda_{2}$ we simply find
  \begin{displaymath}
   R_{2}(\rho,v,S)=(0,0,1)^{T},
  \end{displaymath}
  and after some work,
  \begin{displaymath}
    R_{1}(\rho,v,S)=\left(-\frac{(a^{2}+1)\rho}{a(1-v^{2})},1,0\right)^{T}
  \end{displaymath}
  and
  \begin{displaymath}
    R_{3}(\rho,v,S)=\left(\frac{(a^{2}+1)\rho}{a(1-v^{2})},1,0\right)^{T}.
  \end{displaymath}
  Computing the gradients of the eigen-fields with respect to $\omega=(\rho,v,S)$,
  \begin{eqnarray}
    \nabla\lambda_{1} & = & \left(0,\frac{1-a^{2}}{(1-av)^{2}},0\right),\nonumber\\
    \nabla\lambda_{2} & = & \left(0,1,0\right),\nonumber\\
    \nabla\lambda_{3} & = & \left(0,\frac{1-a^{2}}{(1+av)^{2}},0\right).\nonumber
  \end{eqnarray}
  Thus,
  \begin{eqnarray}
    R_{1}\cdot\nabla\lambda_{1} & = & \frac{1-a^{2}}{(1-av)^{2}}\neq 0,\nonumber\\
    R_{2}\cdot\nabla\lambda_{2} & = & 0,\nonumber\\
    R_{3}\cdot\nabla\lambda_{3} & = & \frac{1-a^{2}}{(1+av)^{2}}\neq 0.\nonumber
  \end{eqnarray}
  The first and third are non-zero and bounded by the restrictions
  on $a$ and $v$.
\end{proof}

It is interesting to note that the eigenvalues
\eqref{Characteristic-Speeds} are the relativistic analog of the sum
of the local sound speed and fluid velocity in the classical Euler
equations. In the classical case, the first and third characteristic
fields have eigenvalues $\lambda_{1}=u-c$ and $\lambda_{3}=u+c$,
which are the sum and differences of the fluid speed and the local
speed of sound respectively. The eigenvalues
\eqref{Characteristic-Speeds} are exactly the relativistic sum of
two velocities within the frame work of relativity.

\bigskip

\section{Riemann Invariants}

    The Riemann invariants for the system \eqref{System-Conservation}
can be found from the eigenvectors. An
\textbf{$i^{\textrm{th}}-$Riemann invariant} is a function $\psi$
such that
\begin{displaymath}
  R_{i}\cdot\nabla\psi=0.
\end{displaymath}
In other words, the level curves of $\psi$ are the integral curves
of the $i^{\textrm{th}}$ characteristic field.  We will perform our
interaction estimate analysis in the coordinate system of Riemann
invariants because the rarefaction curves have particularly simple
structure; straight lines parallel to the coordinate axes.  From the
eigenvector $R_{1}$ we see that along $1-$rarefaction curves,
\begin{displaymath}
  \frac{d\rho}{dv}=-\frac{a^{2}+1}{a}\frac{\rho}{1-v^{2}},
\end{displaymath}
which we can explicitly solve to find that along the first integral
curve,
\begin{displaymath}
  \frac{a}{a^{2}+1}\ln(\rho)+\frac{1}{2}\ln\left(\frac{1+v}{1-v}\right)=\textrm{const.}
\end{displaymath}
This can be done for the third integral curve in a similar fashion.
We therefore define:
\begin{eqnarray}\label{Riemann-Invariant-rs}
  r=\frac{1}{2}\ln\left(\frac{1+v}{1-v}\right)-\frac{a}{1+a^{2}}\ln(\rho),\nonumber\\
  s=\frac{1}{2}\ln\left(\frac{1+v}{1-v}\right)+\frac{a}{1+a^{2}}\ln(\rho).
\end{eqnarray}

The function $r=r(\rho,v)$ is constant across $3-$rarefaction waves
and $s=s(\rho,v)$ is constant across $1-$rarefaction waves. From the
supplemental equation \eqref{Entropy-Conserved}, we see that the
entropy, $S$, is a third Riemann invariant constant across $1$ and
$3-$rarefaction waves. In our analysis, we will view state space in
the coordinates of the Riemann invariants rather than the conserved
variables. However, using $S$ is not sufficient because the shock
curves in $(r,s,S)$ space are, in general, not translationally
invariant.  Instead we will use $\Sigma=\ln(A(S))$ as our third
coordinate. It will be shown in Section
\ref{Equations-of-State-Section} that in $(r,s,\Sigma)$ space, the
shock-rarefaction curves are indeed independent of base point. Since
$S$ is a Riemann invariant, $\ln(A(S))$ must be one too. Indeed,
suppose that $\psi$ is a $i^{\textrm{th}}-$Riemann invariant and let
$f\in\mathcal{C}^{1}(\mathbb{R},\mathbb{R})$. Then $f(\psi)$ is an
$i^{\textrm{th}}-$Riemann invariant as well since,
\begin{displaymath}
  R_{i}\cdot\nabla f(\psi)=f'(\psi)R_{i}\cdot\nabla\psi=0.
\end{displaymath}
We now change our variables from the conserved quantities
$(U_{1},U_{2},U_{3})$ to $(\rho,v,S)$.
\begin{proposition}\label{Change-Of-Variables}
In the region, $\rho>0$, $-1<v<1$, $S>0$, the mapping
$(\rho,v,S)\to(U_{1},U_{2},U_{3})$ is one-to-one, and the Jacobian
determinant of the map is both continuous and non-zero.
\end{proposition}
\begin{proof}
We will show first that the map $(\rho,v)\rightarrow(U_{2},U_{3})$
is one-to-one for $\rho>0$ and $-1<v<1$.  Assume the contrary.
Suppose we have $(\rho_{1},v_{1})$ and $(\rho_{2},v_{2})$ such that
$U_{2}(\rho_{1},v_{1})=U_{2}(\rho_{2},v_{2})$ and
$U_{3}(\rho_{1},v_{1})=U_{3}(\rho_{2},v_{2})$.   To begin we show
that if $v_{1}=v_{2}=v$ then $\rho_{1}=\rho_{2}$. From the equality
$U_{3}(\rho_{1},v)=U_{3}(\rho_{2},v)$ we have
  \begin{displaymath}
    \rho_{1}\left((a^{2}+1)\frac{v^{2}}{1-v^{2}}+1\right)=\rho_{2}\left((a^{2}+1)\frac{v^{2}}{1-v^{2}}+1\right).
  \end{displaymath}
  Since the term
  \begin{displaymath}
    \left((a^{2}+1)\frac{v^{2}}{1-v^{2}}+1\right)\neq 0
  \end{displaymath}
  for any $-1<v<1$ we must have $\rho_{1}=\rho_{2}$.
  We now show that if the images of $U_{2}$ and $U_{3}$ are equal,
  then we must have $v_{1}=v_{2}$
  and, by the previous argument, $\rho_{1}=\rho_{2}$.

  From $U_{3}(\rho_{1},v_{1})=U_{3}(\rho_{2},v_{2})$ and $U_{2}(\rho_{1},v_{1})=U_{2}(\rho_{2},v_{2})$ we have
  \begin{equation}\label{Coordinate-One-To-One}
    \frac{\rho_{1}}{\rho_{2}}\left((a^{2}+1)\frac{v_{1}^{2}}{1-v_{1}^{2}}+1\right)=\left((a^{2}+1)\frac{v_{2}^{2}}{1-v_{2}^{2}}+1\right)
  \end{equation}
  and
  \begin{displaymath}
    \frac{\rho_{1}}{\rho_{2}}=\left(\frac{v_{1}}{1-v_{1}^{2}}\right)^{-1}\left(\frac{v_{2}}{1-v_{2}^{2}}\right).
  \end{displaymath}
  Note that if $v_{1}/(1-v_{1}^{2})=0$ we must also have $v_{1}=0$ and
  by \eqref{Coordinate-One-To-One}, $v_{2}=0$ since $\rho_{1}/\rho_{2}\neq
  0$. Assume that $v_{1}\neq 0$.

  Replacing $\rho_{1}/\rho_{2}$ in \eqref{Coordinate-One-To-One} and simplifying,
  \begin{displaymath}
    \frac{a^{2}v_{1}^{2}+1}{v_{1}}=\frac{a^{2}v_{2}^{2}+1}{v_{2}},
  \end{displaymath}
  which further reduces to
  \begin{displaymath}
    (v_{1}-v_{2})(a^{2}v_{1}v_{2}-1)=0.
  \end{displaymath}
Since $|a|,|v_{1}|,|v_{2}|<1$, the second term,
$(a^{2}v_{1}v_{2}-1)\neq 0$, so it must be $v_{1}=v_{2}$. Therefore,
the mapping $(\rho,v)\longleftrightarrow (U_{2},U_{3})$ is
one-to-one.

  Now we show that the mapping $(\rho,v,S)\rightarrow (U_{1},U_{2},U_{3})$ is
  one-to-one.  Proceed again by contradiction by supposing
  $(\rho_{1},v_{1},S_{1})$ and $(\rho_{2},v_{2},S_{2})$ have the
  same image.  Since $U_{2}$ and $U_{3}$ only depend on $\rho$ and
  $v$, the previous argument shows that $\rho_{1}=\rho_{2}$ and
  $v_{1}=v_{2}$.  We now show that $S_{1}=S_{2}$.  Since
  $n=n(\rho,S)$ the equality
  $U_{1}(\rho_{1},v_{1},S_{1})=U_{1}(\rho_{2},v_{2},S_{2})$ reduces
  to
  \begin{displaymath}
    n(\rho,S_{1})=n(\rho,S_{2}).
  \end{displaymath}
  Therefore, we are done if $\partial n/\partial S\neq 0$.  We use the
  fact that $\rho=n\epsilon$ to rewrite the second law of
  thermodynamics \eqref{Thermodynamics} as
  \begin{displaymath}
    nd\rho=n^{2}TdS+(a^{2}+1)\rho dn.
  \end{displaymath}
  Therefore,
  \begin{displaymath}
    \frac{\partial n}{\partial S}=-\frac{n^{2}T}{(a^{2}+1)\rho}\neq
    0,
  \end{displaymath}
  and the mapping $(\rho,v,S)\rightarrow(U_{1},U_{2}U_{3})$ is one-to-one.

  The jacobian matrix of the map is given by
  \begin{displaymath}
  J=\left(
           \begin{array}{ccc}
             \frac{1}{(a^{2}+1)\sqrt{1-v^{2}}\epsilon}  & (a^{2}+1)\frac{v}{1-v^{2}} &  (a^{2}+1)\frac{v^{2}}{1-v^{2}}+1\\
             \frac{nv}{(1-v^{2})^{3/2}} & (a^{2}+1)\frac{\rho(1+v^{2})}{(1-v^{2})^{2}} & (a^{2}+1)\frac{2\rho v}{(1-v^{2})^{2}} \\
             \frac{-n^{2}T}{(1+a^{2})\rho} & 0 & 0 \\
           \end{array}
         \right),
  \end{displaymath}
  whose determinant is
  \begin{displaymath}
    \textrm{det}(J)=\frac{n^{2}T(1-a^{2}v^{2})}{(1-v^{2})^{2}}>0,
  \end{displaymath}
  which is continuous on $\rho>0$, $-1<v<1$ and $S>0$.

\end{proof}

\bigskip

\section{Jump Conditions}

    Systems of conservation laws, or more specifically the relativistic
Euler equations \eqref{System-Conservation}, encode the required
information to calculate the evolution of discontinuities, i.e.
shock waves, in one or more of the conserved variables.  One must
use care however, because systems of equations equivalent to
\eqref{System-Conservation} for smooth solutions can, and typically
do not, give the same relations for discontinuous solutions.  A
prime example of this is specific entropy is constant along flow
lines for continuous solutions of \eqref{System-Conservation} from
\eqref{Entropy-Conserved}, but entropy is not conserved and
increases across a shock front.

For systems of conservation laws, the relations defining the
dynamics of shock waves are the Rankine-Hugoniot jump conditions.
These relations state for a shock wave traveling at speed $s$, the
change in the conserved quantities $U$ across the shock and the
change in $F(U)$ across the shock, denoted $[[U]]$ and $[[F(U)]]$
respectively, satisfy,
\begin{equation}\label{Rankine-Hugoniot}
  s [[U]]=[[F(U)]].
\end{equation}
For a given state $U_{L}$, the Rankine-Hugoniot relations, for each
$i=1,\ldots,n,$ define a $1-$parameter family of states that can be
connected on the right by a shock wave in the $i^{th}$
characteristic family.  Moreover, this curve has second order
contact with the curve defining all the states that connect to
$U_{L}$ on the right by an $i^{th}$ rarefaction wave given by the
$i^{th}$ integral curve. These facts were first proven by Lax in
$1957$ for a general system of strictly hyperbolic conservation laws
with genuinely nonlinear or linearly degenerate characteristic
fields, \cite{Lax}.

We call $\mathcal{R}_{i}(U)$ the integral curve of the $i^{th}$
characteristic field that passes through the state $U$ and
$\mathcal{S}_{i}(U)$ the one parameter family of states defined by
\eqref{Rankine-Hugoniot} that defines states that connect to $U$ by
a shock wave in the $i^{th}$ family.  Only half of each of these
curves will be physically relevant.  For a genuinely non-linear
characteristic field we take the portion of $\mathcal{R}_{i}(U)$
extending from $U$ that satisfies $\lambda_{i}(U)<\lambda_{i}(U')$.
Call this portion $\mathcal{R}^{+}_{i}(U)$.  On the other hand, take
the portion of the shock curve $\mathcal{S}_{i}(U)$ that satisfies
the Lax entropy condition,
\begin{displaymath}
  \lambda_{i}(U')<s<\lambda_{i}(U).
\end{displaymath}
Call this portion $\mathcal{S}^{-}_{i}(U)$.  Finally, define
$\mathcal{T}_{i}(U)=\mathcal{R}^{+}_{i}(U)\cup
\mathcal{S}^{-}_{i}(U)$.

For our system given by \eqref{System-Conservation}, we have that
the tangent to the worldline of the shock front is proportional to
$(1,s)$. Define $l^{\alpha}$ by
\begin{displaymath}
  (l^{0},l^{1})=(1,s).
\end{displaymath}

The jump conditions \eqref{Rankine-Hugoniot} for the system
\eqref{System-Divergence} is then given by
\begin{eqnarray}\label{Jump-Conditions}
  \left[\left[nu^{\alpha}\right]\right]l_{\alpha}=0,\phantom{3333333}\nonumber \\
  \left[\left[T^{\alpha \beta}\right]\right] l_{\alpha}=0, \phantom{11}
  \beta=0,1.
\end{eqnarray}
Recall that $l_{\alpha}$ is found by contracting $l^{\alpha}$ with
the metric $\eta$:
\begin{displaymath}
  l_{\alpha}=l^{\beta}\eta_{\alpha\beta}
\end{displaymath}
From the first equation in \eqref{Jump-Conditions} we have for some
constant $m$,
\begin{equation}\label{Shock-Wave-Mass}
  m=nu^{\alpha}l_{\alpha}=n_{L}u_{L}^{\alpha}l_{\alpha}.
\end{equation}
For the case $m=0$, we have for $n,n_{L}>0$,
\begin{displaymath}
  u^{\alpha}l_{\alpha}=u^{\alpha}_{L}l_{\alpha}.
\end{displaymath}
Since the components $u^{\alpha}$ are in a one-to-one relation with
the fluid velocity $v$, we have $v=v_{L}$.  Furthermore, the second
equation reduces to $p=p_{L}$.  This case, $m=0$, corresponds to an
entropy wave rather than a compressive shock.  Shock waves will
correspond to $m\neq 0$.  The thermodynamic relationships across a
shock wave in a solution to the relativistic Euler equations was
first given by Taub, \cite{Taub-Rel_RankHugo}.
\begin{proposition}[Taub, $1948$]\label{Hugoniot}
 Let $U=(\rho,v,n)$ and $U_{L}=(\rho_{L},v_{L},n_{L})$ be two states separated by a shock wave.  Then the
 following relation holds:
 \begin{equation}\label{Taub-Adiabat}
   \frac{\rho+p}{n^{2}}\left(\rho+p_{L}\right)=\frac{\rho_{L}+p_{L}}{n_{L}^{2}}\left(\rho_{L}+p\right).
 \end{equation}
\end{proposition}
\begin{proof}
  We will show that across a shock wave the following condition
  holds on the two separating states,
  \begin{equation}\label{Taub-Adiabat-Alt}
    \left(\frac{p+\rho}{n}\right)^{2}-\left(\frac{p_{L}+\rho_{L}}{n_{L}}\right)^{2}+(p_{L}-p)\left(\frac{p+\rho}{n}+\frac{p_{L}+\rho_{L}}{n_{L}}\right)=0.
  \end{equation}

Assuming this holds, we multiply out, cancel and collect terms with
$n$ and $n_{L}$ in the denominator on the left and right
respectively to get
\begin{displaymath}
  \frac{\rho^{2}+p\rho+\rho
  p_{L}+pp_{L}}{n^{2}}=\frac{\rho_{L}^{2}+p_{L}\rho_{L}+pp_{L}+p_{L}\rho}{n_{L}^{2}}.
\end{displaymath}
Equation \eqref{Taub-Adiabat} follows directly.

For convenience define
\begin{displaymath}
    g=\frac{p+\rho}{n} \phantom{444} \textrm{and} \phantom{444}
    g_{L}=\frac{p_{L}+\rho_{L}}{n_{L}}.
\end{displaymath}
If $m=0$ we have a jump discontinuity.  Since we are concerned about
the shock waves, assume $m\neq 0$.  In this case the second equation
in \eqref{Jump-Conditions} gives
\begin{displaymath}
  ngu^{\alpha}u^{\beta}l_{\alpha}+p\eta^{\alpha\beta}l_{\alpha}=n_{L}g_{L}u_{L}^{\alpha}u_{L}^{\beta}l_{\alpha}+p_{L}\eta^{\alpha\beta}l_{\alpha},
\end{displaymath}
that, in light of \eqref{Shock-Wave-Mass}, reduces to
\begin{equation}\label{Jump-Condition-T-Mass}
  mgu^{\beta}+pl^{\beta}=mg_{L}u_{L}^{\beta}+p_{L}l^{\beta}.
\end{equation}
Contracting equation \eqref{Jump-Condition-T-Mass} with $u_{\beta}$
and $u_{L_{\beta}}$ then using \eqref{Shock-Wave-Mass} we find
\begin{equation}\label{Taub-1}
  -g+\frac{p}{n}=g_{L}u^{\beta}_{L}u_{\beta}+\frac{p_{L}}{n}
\end{equation}
and
\begin{equation}\label{Taub-2}
  gu^{\beta}u_{L_{\beta}}+\frac{p}{n_{L}}=-g_{L}+\frac{p_{L}}{n_{L}}.
\end{equation}
We use \eqref{Taub-1} to solve for $u_{L}^{\beta}u_{\beta}$:
\begin{equation}\label{Taub-ulu}
    u_{L}^{\beta}u_{\beta}=\frac{1}{g_{L}}\left(-g+\frac{p}{n}-\frac{p_{L}}{n}\right).
\end{equation}
Plugging \eqref{Taub-ulu} into \eqref{Taub-2}, combining and using
the definition of $g$ and $g_{L}$, we obtain
\eqref{Taub-Adiabat-Alt}.
\end{proof}

In particular, with $p=a^{2}\rho$, \eqref{Taub-Adiabat} reduces to
\begin{equation}\label{Taub-pa2rho}
  \frac{n^{2}}{n_{L}^{2}}=\frac{\rho^{2}}{\rho_{L}^{2}}\frac{\left(1+a^{2}\frac{\rho_{L}}{\rho}\right)}{\left(1+a^{2}\frac{\rho}{\rho_{L}}\right)}.
\end{equation}

The global structure of the solutions of the shock relations
\eqref{Rankine-Hugoniot} for the relativistic Euler equations in the
space of Riemann invariants was first done by Smoller and Temple for
an equation of state of the form \eqref{p-a2rho},
\cite{Smoller-Temple-Global-Solutions-Rel-Euler}. We summarize their
results in the following lemma:

\begin{lemma}[Smoller, Temple, $1993$]\label{Lemma-Smoller-Temple}
  Let $p=a^{2}\rho$ with $0<a<1$.  The projection of the $i$-shock
  curves for $i=1,3$ onto the plane of Riemann invariants $(r,s)$ at any entropy level satisfy the
  following:
  \begin{enumerate}
    \item The shock speed $s$ is monotonically increasing or decreasing along the
    shock curve $\mathcal{S}_{i}$ and for each state
    $(\rho_{L},v_{L})\neq (\rho_{R},v_{R})$ on $\mathbf{S}_{i}$ the
    Lax entropy condition holds:
    \begin{displaymath}
      \lambda_{i}(\rho_{R},v_{R})<s_{i}<\lambda_{i}(\rho_{L},v_{L}).
    \end{displaymath}

    \item The shock curves, when parameterized by $\Delta\ln(\rho)$, are translationally invariant.
    Furthermore the $1$ and $3-$shock curves based at a common point $(\overline{r},\overline{s})$
    have mirror symmetry across the line $r=s$ through the point
    $(\overline{r},\overline{s})$.

    \item The $i-$shock curves are convex and
    \begin{displaymath}
      0\leq \frac{ds}{dr}\leq \frac{\sqrt{2K}-1}{-\sqrt{2K}-1}<1
    \end{displaymath}
    for $i=1$ and
    \begin{displaymath}
      0\leq \frac{dr}{ds}\leq \frac{\sqrt{2K}-1}{-\sqrt{2K}-1}<1
    \end{displaymath}
    for $i=3$ where $K=2a^{2}/(1+a^{2})^{2}$.
  \end{enumerate}
\end{lemma}

In light of Lemma \ref{Lemma-Smoller-Temple}, we see that we can
globally define the shock curves $\mathcal{S}_{i}(U)$ in the
$rs-$plane and we know that everywhere on this curve the Lax entropy
conditions hold. We now extend the analysis of Smoller and Temple
and show that the entropy change along the shock waves also possess
the translationally invariant property and are convex in a
particular coordinate system. After this we will show that the
Riemann problem is globally solvable with equation of state
\eqref{EOS-Family}, in the ultra-relativistic limit.

\bigskip

\section{Equations of State}\label{Equations-of-State-Section}
In this section, we will show certain properties hold for our family
of equations of state. Namely, we will need that as a function of
wave strength, the change in a certain function of entropy is
independent of base point. Moreover, we will find that the change of
this function of entropy and its derivative are monotone increasing.
We will use these facts in our estimates on the entropy waves in
Section \ref{Section-InteractionEstimates}.

    For an equation of state of the form
\begin{displaymath}
  \epsilon(n,S)=A(S)n^{\gamma-1},
\end{displaymath}
with $A$ satisfying \eqref{A1}, \eqref{A2} and \eqref{A3}, the
second law of thermodynamics says,
\begin{displaymath}
  p(n,S)=n^{2}\frac{\partial\epsilon}{\partial n}=(\gamma-1)A(S)n^{\gamma}=(\gamma-1)\epsilon n.
\end{displaymath}
In the ultra-relativistic limit this further reduces to
\begin{displaymath}
    p(n,S)=(\gamma-1)\rho,
\end{displaymath}
an equation of state of the form \eqref{p-a2rho} with
$a=\sqrt{\gamma-1}$.

Now we will show that a certain function of entropy across a shock
wave is independent of base state $(\rho_{L},v_{L},S_{L})$ by using
Proposition \ref{Hugoniot}. Choose $\Sigma$ by
\begin{equation}\label{Sigma}
  \Sigma(S)=\ln\left(A(S)\right).
\end{equation}
Our goal is to show that across a shock wave, the difference
$[\Sigma-\Sigma_{L}]$ is a function of the change of the
corresponding Riemann invariants alone. Then the difference
$[\Sigma-\Sigma_{L}]$ along the shock curve is independent of base
point. Finally, we will show that the difference
$[\Sigma-\Sigma_{L}]$ and its derivative, as a function of the
change of Riemann invariants, are monotone increasing. Later we will
measure the strength of $1-$shocks as the change in $r$ and by the
change in $s$ for $3-$shocks. It is sufficient to show that the
change $[\Sigma-\Sigma_{L}]$ and its derivative are monotone
increasing as viewed as a function of $\ln(\rho/\rho_{L})$, because
they satisfy the relationship as parameters,
\begin{displaymath}
  \Delta r=\frac{2a}{a^{2}+1}\Delta\ln(\rho).
\end{displaymath}
For $3-$Shocks we replace $\Delta r$ with $\Delta s$. Thus,
\begin{displaymath}
  \frac{d[S-S_{L}]}{d(r-r_{L})}=\frac{d[S-S_{L}]}{d\ln(\rho/\rho_{L})}\cdot\left|\frac{d\ln(\rho/\rho_{L})}{d(r-r_{L})}\right|=\frac{a^{2}+1}{2a}\cdot\frac{d[S-S_{L}]}{d\ln(\rho/\rho_{L})}.
\end{displaymath}

Using \eqref{Taub-pa2rho}, we calculate $[\Sigma-\Sigma_{L}]$,
\begin{eqnarray}
      \Sigma-\Sigma_{L} & = & \ln\left(A(S)\right)-\ln\left(A(S_{L})\right),\\\nonumber\\
       & = & \ln\left(\frac{\rho}{n^{\gamma}}\frac{n_{L}^{\gamma}}{\rho_{L}}\right),\nonumber\\\nonumber\\
       & = &
       (1-\gamma)\ln\left(\frac{\rho}{\rho_{L}}\right)-\frac{\gamma}{2}\ln\left(\frac{1+\left(\gamma-1\right)\frac{\rho}{\rho_{L}}}{1+\left(\gamma-1\right)\frac{\rho_{L}}{\rho}}\right).\nonumber
\end{eqnarray}
Thus for $\sigma=\ln(\rho/\rho_{L})$,
\begin{equation}\label{Sigma-Change}
  [\Sigma-\Sigma_{L}](\sigma)=(1-\gamma)\sigma+\frac{\gamma}{2}\ln\left(\frac{1+(\gamma-1)e^{\sigma}}{1+(\gamma-1)e^{-\sigma}}\right).
\end{equation}
After differentiating, we have
\begin{equation}
  \frac{d[\Sigma-\Sigma_{L}]}{d\sigma}=\frac{(e^{\sigma}-1)^{2}(2-\gamma)(\gamma-1)}{2(1+e^{\sigma}(\gamma-1))(e^{\sigma}+(\gamma-1))},
\end{equation}
which is non-negative in the domain $1<\gamma<2$ and $\sigma\geq 0$.
Furthermore, the derivative is zero only when $\sigma=0$. Thus,
$[S-S_{L}](\sigma)$ is a monotone increasing function.

Next we show that $d[\Sigma-\Sigma_{L}]/d\sigma$ is also monotone
increasing. We take another derivative and find
\begin{displaymath}
  \frac{d^{2}[\Sigma-\Sigma_{L}]}{d\sigma^{2}}=\frac{\gamma^{2}(2-\gamma)(\gamma-1)(e^{3\sigma}-e^{\sigma})}{2(1+e^{\sigma}(\gamma-1))^{2}(e^{\sigma}+(\gamma-1))^{2}}.
\end{displaymath}
The denominator is always positive and the numerator is positive
because $1<\gamma<2$ and $e^{3\sigma}\geq e^{\sigma}$ for
$\sigma\geq 0$.

We have proven the following proposition:

\begin{proposition}\label{EOS-Independent}
Consider the ultra-relativistic Euler equations with the equation of
state $\epsilon(n,S)=A(S)n^{\gamma-1}$ and $A$ satisfying
\eqref{A1}, \eqref{A2} and \eqref{A3}. Then the change in
$\Sigma=\ln(A(S))$, when regarded as a function of the change in the
corresponding Riemann invariant, is independent of base state.
Geometrically, the shock curves, as viewed in $(r,s,\Sigma)-$space,
are translationally invariant.
\end{proposition}

An interesting fact is that the change in $\Sigma$ becomes nearly
linear for strong shock waves.  We state this as a corollary.
\begin{corollary}
  Under the assumptions of Proposition
  \ref{EOS-Independent}, the change in $\Sigma$ becomes nearly
  linear for strong shocks.
\end{corollary}
\begin{proof}
This follows immediately when considering the following limit:
\begin{displaymath}
  \lim_{\sigma\rightarrow\infty}\frac{d[\Sigma-\Sigma_{L}]}{d\sigma}=\lim_{\sigma\rightarrow\infty}\frac{(e^{\sigma}-1)^{2}(2-\gamma)(\gamma-1)}{2(1+e^{\sigma}(\gamma-1))(e^{\sigma}+(\gamma-1))}=\frac{(2-\gamma)}{2}.
\end{displaymath}
\end{proof}

In the following sections we will show that both an ideal gas and
one dominated by radiation fall into the family of equations of
state given by \eqref{EOS-Family}.

    \subsection{Ideal Gas}
        \skipline An ideal gas satisfies the ideal gas law,
\begin{displaymath}
  \frac{p}{n}=RT,
\end{displaymath}
\cite{Cour-Friderichs}. Furthermore, if we assume that the internal
energy is proportional to the temperature,
\begin{displaymath}
  \epsilon=\frac{R}{\gamma-1}T,
\end{displaymath}
we can use the second law of thermodynamics to determine
$\epsilon(n,S)$. From
\begin{displaymath}
  \frac{\partial\epsilon}{\partial S}=T=\frac{\gamma-1}{R}\epsilon,
\end{displaymath}
we get the for some function $\varphi$,
\begin{displaymath}
    \epsilon(n,S)=e^{\varphi(n)}e^{\frac{\gamma-1}{R}S}.
\end{displaymath}
To find $\varphi(n)$ we again use the second law of thermodynamics
to get the relation
\begin{displaymath}
  \frac{\partial\epsilon}{\partial n}=\frac{\gamma-1}{n}\epsilon,
\end{displaymath}
which reduces to
\begin{displaymath}
  \frac{d\varphi}{dn}=\frac{\gamma-1}{n}.
\end{displaymath}
Solving for $\varphi$, we get
\begin{displaymath}
  \varphi(n)=(\gamma-1)\ln{(n)}=\ln{(n^{\gamma-1})}.
\end{displaymath}
Therefore,
\begin{displaymath}
  \epsilon(n,S)=e^{\frac{\gamma-1}{R}S}n^{\gamma-1},
\end{displaymath}
for some $1<\gamma<2$. Thus, in the case of a polytropic gas, we
have an equation of state of the form
$\epsilon(n,S)=A(S)n^{\gamma-1}$, where
\begin{displaymath}
  A(S)=e^{\frac{\gamma-1}{R}S},
\end{displaymath}
satisfying \eqref{A1}, \eqref{A2} and \eqref{A3}. We also see that
$\Sigma$ is proportional to $S$:
\begin{displaymath}
  \Sigma(S)=\ln(e^{\frac{\gamma-1}{R}S})=\frac{\gamma-1}{R}S.
\end{displaymath}
Notice that in the case of an ultra-relativistic polytropic gas, it
would have been sufficient to consider $(r,s,S)-$space since the
shock curves would still be translationally invariant.

    \subsection{Radiation Dominated Gas}
        \skipline A gas in local thermodynamical equilibrium with radiation
when only the internal energy and pressure are dominated by
radiation is characterized by
    \begin{equation}\label{Internal-Pressure-Radiation-4-3}
      \epsilon=\frac{a_{R}T^{4}}{n}\phantom{33} \textrm{and}
      \phantom{44} p=\frac{1}{3}a_{R}T^{4},
    \end{equation}
where $a_{R}=7.56\times 10^{-15}$ is the Stefan-Boltzmann constant,
\cite{Anile-RelFluids}. We can generalize the equation of state
\eqref{Internal-Pressure-Radiation-4-3} to the continuum of
equations of state,
    \begin{equation}\label{Internal-Pressure-Radiation-Gamma}
      \epsilon=\frac{a_{R}T^{\frac{\gamma}{\gamma-1}}}{n}\phantom{33} \textrm{and}
      \phantom{44} p=(\gamma-1)a_{R}T^{\frac{\gamma}{\gamma-1}},
    \end{equation}
for $1<\gamma<2$.  Notice \eqref{Internal-Pressure-Radiation-Gamma}
reduces to \eqref{Internal-Pressure-Radiation-4-3} when $\gamma$ is
chosen to be $4/3$.

In order to find the entropy profile associated with this equation
of state we again use the second law of thermodynamics. From
$d\epsilon/dn=p/n^{2}$ we find
\begin{displaymath}
  \frac{\left(\frac{\gamma}{\gamma-1}\right)a_{R}nT^{\left(\frac{1}{\gamma-1}\right)}\frac{dT}{dn}-a_{R}T^{\left(\frac{\gamma}{\gamma-1}\right)}}{n^{2}}=\frac{p}{n^{2}},
\end{displaymath}
which in light of \eqref{Internal-Pressure-Radiation-Gamma} and
after some algebra, reduces to
\begin{equation}\label{Radiation-dt-dn}
  \frac{dT}{dn}=(\gamma-1)\frac{T}{n}.
\end{equation}
Similarly, from $d\epsilon/dS=T$, we find
\begin{displaymath}
  \frac{\left(\frac{\gamma}{\gamma-1}\right)a_{R}nT^{\left(\frac{1}{\gamma-1}\right)}\frac{dT}{dS}}{n^{2}}=T,
\end{displaymath}
which can be simplified and integrated to find that for some
function $f(n)$,
\begin{equation}\label{Radiation-T-S-fn}
  \gamma a_{R}T^{\left(\frac{1}{\gamma-1}\right)}=nS+f(n).
\end{equation}
Differentiating \eqref{Radiation-T-S-fn} with respect to $n$ and
using \eqref{Radiation-dt-dn}, we find the following relation on
$f$,
\begin{displaymath}
  nf'(n)=f(n).
\end{displaymath}
Thus, for some constant $c$, $f(n)=cn$, and we can incorporate $c$
into the entropy level $S$ giving,
\begin{displaymath}
  \gamma a_{R}T^{\left(\frac{1}{\gamma-1}\right)}=nS,
\end{displaymath}
or equivalently,
\begin{displaymath}
  S=\frac{\gamma a_{R}T^{\left(\frac{1}{\gamma-1}\right)}}{n}.
\end{displaymath}
Therefore, in the ultra-relativistic regime or for a massless gas,
\begin{displaymath}
  \rho=a_{R}\left(\frac{S}{\gamma a_{R}}\right)^{\gamma}n^{\gamma}
\end{displaymath}
and
\begin{equation}\label{Radiation-Internal}
  \epsilon(n,S)=a_{R}\left(\frac{S}{\gamma a_{R}}\right)^{\gamma}n^{\gamma-1}.
\end{equation}
The equation of state for thermal radiation
\eqref{Radiation-Internal} is of the form \eqref{EOS-Family} with
\begin{displaymath}
  A(S)=a_{R}\left(\frac{S}{\gamma a_{R}}\right)^{\gamma}.
\end{displaymath}

Unlike the case for a polytropic gas, where it would have been
sufficient to consider just the change in $S$, the change in entropy
across a shock wave is no longer a function of $\ln(\rho/\rho_{L})$
alone; it is also dependent on the starting entropy level.  More
specifically, in this case we have the ratio of entropy values being
independent of base state, rather than the difference, leading to
\begin{displaymath}
  S-S_{L}=S_{L}\left(\frac{S}{S_{L}}-1\right)=S_{L}\left(\left[\frac{S}{S_{L}}\right](\sigma)-1\right).
\end{displaymath}
Choosing the new coordinate $\Sigma=\ln(A(S))$ is necessary to keep
the change independent of base point.  In the case of a radiation
dominated gas,
\begin{displaymath}
  \Sigma=\gamma\ln\left(\frac{S}{\gamma
  a_{R}^{\frac{\gamma-1}{\gamma}}}\right).
\end{displaymath}

\bigskip

\section{The Riemann Problem}

    Riemann problems are used as the building blocks of finite volume
method solution schemes for systems of conservation laws. The
Riemann problem is a particular class of Cauchy problems with
initial data of the form,
\begin{displaymath}
  U_{0}(x)=\left\{\begin{array}{ll}
    U_{L}\phantom{44444} & x<0, \\
    U_{R} & x>0. \nonumber\\
  \end{array}\right.
\end{displaymath}
We will show that for any to initial states in the region $\rho>0$,
$-1<v<1$ and $S>0$, there exists a solution of the Riemann problem
for the system \eqref{System-Conservation} with equation of state
\eqref{EOS-Family}.

\begin{theorem}
Consider left and right states $U_{L}=(\rho_{L},v_{L},S_{L})$ and
$U_{R}=(\rho_{R},v_{R},S_{R})$, such that $\rho_{L},\rho_{R}>0$,
$-1<v_{L},v_{R}<1$, and $S_{L},S_{R}>0$.  With the equation of state
\eqref{EOS-Family} satisfying $1<\gamma<2$, \eqref{A1}, \eqref{A2}
and \eqref{A3}, there exists a weak solution to the Riemann problem
$<U_{L},U_{R}>$ for system \eqref{System-Conservation} in the
ultra-relativistic limit. This solution is unique in the class of
solutions with constant states separated by centered rarefaction,
shock and contact waves.
\end{theorem}

\begin{proof}
For any entropy level, the projection of the shock-rarefaction
curves onto the $rs-$plane is translationally invariant by Lemma
\ref{Lemma-Smoller-Temple}. We will show first that for any two
states, $\overline{U}_{L}=(\rho_{L},v_{L})$ and
$\overline{U}_{R}=(\rho_{R},v_{R})$ in the $rs-$plane, there exists
a intermediate state $\overline{U}_{M}=(\rho_{M},v_{M})$ such that
$\overline{U}_{M}$ is on the shock-rarefaction curve based at
$\overline{U}_{L}$ and $\overline{U}_{R}$ is on the
shock-rarefaction curve based at $\overline{U}_{M}$.  For
convenience, let $\mathcal{T}_{i}(\overline{U})$ denote the
projection of the $i^{th}-$shock-rarefaction curve based at
$\overline{U}$ at any value of $S$ onto the $rs-$plane.  Given a
state $\overline{U}_{L}$, partition the $rs-$plane into four
regions: $I$, consisting of all states above
$\mathcal{T}_{1}(\overline{U}_{L})$ and to the right of
$\mathcal{T}_{3}(\overline{U}_{L})$; $II$, states above
$\mathcal{T}_{1}(\overline{U}_{L})$ and to the left of
$\mathcal{T}_{3}(\overline{U}_{L})$; $III$, states below
$\mathcal{T}_{1}(\overline{U}_{L})$ and above
$\mathcal{T}_{3}(\overline{U}_{L})$; and $IV$, states below
$\mathcal{T}_{1}(\overline{U}_{L})$ and to the left of
$\mathcal{T}_{3}(\overline{U}_{L})$. See Figure
\ref{Riemann-Problem-Sections}.

\begin{figure}
\begin{center}
  \includegraphics[height=185pt]{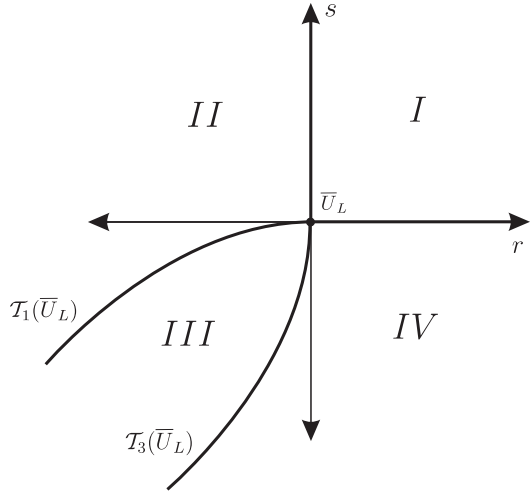}
  \caption{A partition of the $rs-$plane into four sections: $I$, $II$, $III$ and $IV$.}\label{Riemann-Problem-Sections}
\end{center}
\end{figure}

Consider $\mathcal{T}_{1}(\overline{U}_{L})$. For each
$\overline{U}_{M}\in\mathcal{T}_{1}(\overline{U}_{L})$ the 3-wave
rarefaction curves based at $\overline{U}_{M}$ extend vertically
upwards, parallel to the $s-$axis. Therefore, for all states
$\overline{U}_{R}$ in region $I$ or $II$, there is a unique state
$\overline{U}_{M}\in\mathcal{T}_{1}(\overline{U}_{L})$ that connects
$\overline{U}_{L}$ to $\overline{U}_{R}$ by a $1-$shock or a
$1-$rarefaction wave followed by a $3-$rarefaction wave.

We now turn our attention to the portion below
$\mathcal{T}_{1}(\overline{U}_{L})$ in the $rs-$plane.  For region
$IV$, we notice for any state
$\overline{U}_{1}\in\mathcal{R}^{+}_{1}(\overline{U}_{L})$ the shock
curve $\mathcal{S}^{-}_{3}(\overline{U}_{1})$ is a horizontal
translation of $\mathcal{S}^{-}_{3}(\overline{U}_{L})$. Thus, all
the $3-$shock curves extending from
$\mathcal{R}^{+}_{1}(\overline{U}_{L})$ cover region $IV$. For
region $III$ it is clear that the $3-$shock curves extending
downward from $\mathcal{S}^{-}_{1}(\overline{U}_{L})$ must cover all
states in the region. But, we must show that if we take two states
$\overline{U}_{1}$ and $\overline{U}_{2}$ on the shock curve of
$\overline{U}_{L}$ they will never intersect. Suppose two shock
curves intersect at a third state $\overline{U}_{3}$.  See Figure
\ref{Shock-Intesection}. We know by Lemma \ref{Lemma-Smoller-Temple}
that

\begin{figure}
\begin{center}
  \includegraphics[height=190pt]{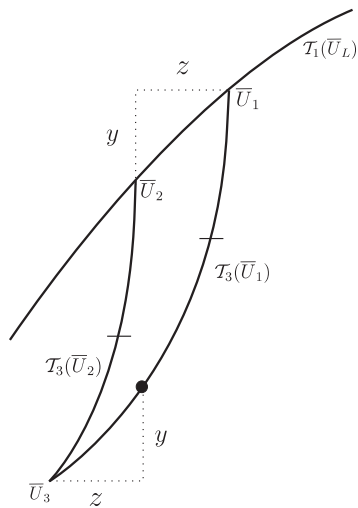}
  \caption{Possible Intersection of two $3-$Shock Curves.}\label{Shock-Intesection}
\end{center}
\end{figure}

\begin{displaymath}
  \frac{z}{y}\leq\frac{\sqrt{2K}-1}{-\sqrt{2K}-1}<1.
\end{displaymath}
However, if $\overline{U}_{1}$ and $\overline{U}_{2}$ are on the
same shock curve,
\begin{displaymath}
  \frac{y}{z}\leq\frac{\sqrt{2K}-1}{-\sqrt{2K}-1}<1.
\end{displaymath}
It must be that the curves never intersect.  Thus, we can solve the
Riemann problem $<\overline{U}_{L},\overline{U}_{2}>$ in the
$rs-$plane.

Now, we use this result to find a solution to the Riemann problem
with $U_{L}=(\rho_{L},v_{L},S_{L})$ and
$U_{R}=(\rho_{R},v_{R},S_{R})$. By the previous argument, find a
middle state $(\rho_{M},v_{M})$ that solves the Riemann problem,
$<(\rho_{L},v_{L}),(\rho_{L},v_{L})>$, in the $rs$-plane.  We only
need to find the two values of $S$ on either side of the contact
discontinuity. This can be accomplished by determining the change in
entropy, across the $1$ and $3-$waves then adapting these changes to
the left and right values of S. For example, the left middle state
$U_{M}$ would have entropy value $S_{L}$ if we had a $1-$rarefaction
wave, and would have entropy value $S_{M}$, where $S_{M}-S_{L}$
equals the corresponding increase in $S$ across the shock wave. We
can find the change in entropy by looking at the equation
$\rho/n^{\gamma}=A(S)$ and solving for $S$.  This is possible since
$A$ is strictly monotone increasing away from zero. Similar methods
determine the value of $S_{M}'$ and the value of entropy in the
right middle state. Since the entropy values of the middle states
satisfy, $S_{L}\leq S_{M}$ and $S_{R}\leq S_{M}'$, we have
$S_{M},S_{M}'>0$. The position of the entropy jump is determined by
the particle path emanating from the initial discontinuity with
speed $v_{M}$.

This construction determines the two unique states
$U_{M}=(\rho_{M},v_{M},S_{M})$ and $U_{M}'=(\rho_{M},v_{M},S_{M}')$
that solves the Riemann problem in the region $\rho>0$, $-1<v<1$ and
$S>0$. Figure \ref{Figure-Riemann-Problem}.
\end{proof}

\begin{figure}
\begin{center}
  \includegraphics[width=6in]{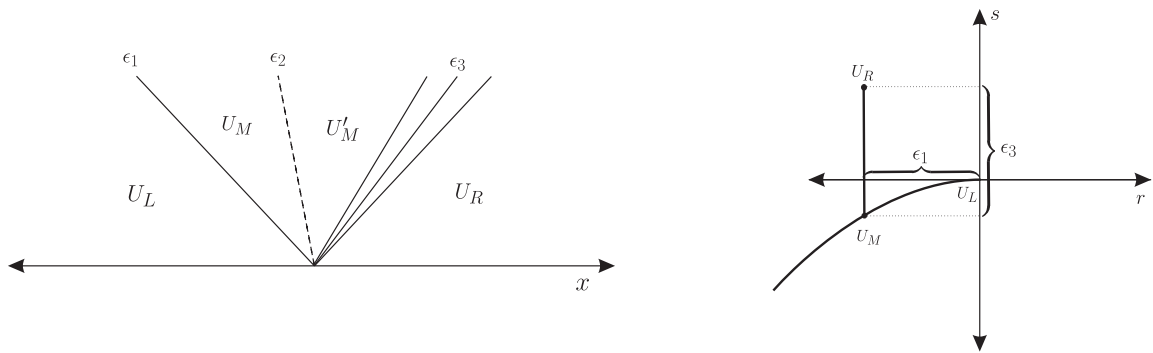}
  \caption{Solution to the Riemann Problem
$<U_{L},U_{R}>$. The states $U_{M}$ and $U_{M}'$ differ only in
$S$.\label{Figure-Riemann-Problem}}
\end{center}
\end{figure}


We parameterize the $1-($resp. $3)$shock/rarefaction curve by the
change in $r($resp. $s)$ and define the strength of a shock or
rarefaction wave as the difference in the values of either $r$ for a
$1-$shock-rarefaction wave, or $s$ for a $2-$shock-rarefaction wave.
We choose the orientation on our parametrization so that we have a
positive parameter along the rarefaction curve and negative
parameter along the shock curve. Therefore, the solution of the
Riemann problem can be given as a sequence of three coordinates,
$(\epsilon_{1},\epsilon_{2},\epsilon_{3})$ where, $\epsilon_{1}$
denotes the change in the Riemann invariant $r$ from $U_{L}$ to
$U_{M}$, $\epsilon_{2}$ the change in $S$ from $U_{M}$ to $U_{M}'$
and $\epsilon_{3}$ the change in the Riemann invariant $s$ from
$U_{M}'$ to $U_{R}$. In summary, for $i=1,3$ we have a shock wave of
strength $\epsilon_{i}$ when $\epsilon_{i}<0$ and a rarefaction wave
of strength $\epsilon_{i}$ when $\epsilon_{i}>0$.

We adopt the following notation:
\begin{eqnarray}
  \alpha & \phantom{444444} & \textrm{Strength of } 1-\textrm{Shock Wave}\nonumber\\
  \beta  & \phantom{444444} & \textrm{Strength of } 3-\textrm{Shock Wave}\nonumber\\
  \mu    & \phantom{444444} & \textrm{Strength of } 1-\textrm{Rarefaction Wave}\nonumber\\
  \eta    & \phantom{444444} & \textrm{Strength of } 3-\textrm{Rarefaction Wave}\nonumber\\
  \delta & \phantom{444444} & \textrm{Strength of Entropy}\phantom{3}\Sigma\textrm{-Wave}\nonumber
\end{eqnarray}
If $(\epsilon_{1},\epsilon_{2},\epsilon_{3})$ is the solution to the
Riemann problem with states $U_{L}$, $U_{R}$, we would have:
\begin{displaymath}
  \alpha=\left\{\begin{array}{ll}
    -\epsilon_{1}\phantom{44444} & \epsilon_{1}\leq 0 \\
    0 & \textrm{Otherwise} \nonumber\\
  \end{array}\right.\nonumber
  \phantom{44444}
  \beta=\left\{\begin{array}{ll}
    -\epsilon_{3}\phantom{44444} & \epsilon_{3}\leq 0 \\
    0 & \textrm{Otherwise} \nonumber\\
  \end{array}\right.\nonumber
\end{displaymath}
\begin{displaymath}
  \mu=\left\{\begin{array}{ll}
    \epsilon_{1}\phantom{44444} & \epsilon_{1}\geq 0 \\
    0 & \textrm{Otherwise} \nonumber\\
  \end{array}\right.\nonumber
  \phantom{44444}
  \eta=\left\{\begin{array}{ll}
    \epsilon_{3}\phantom{44444} & \epsilon_{3}\geq 0 \\
    0 & \textrm{Otherwise} \nonumber\\
  \end{array}\right.\nonumber
\end{displaymath}

\bigskip

We define $\delta=\Sigma_{R}-\Sigma_{L}$ where $\Sigma=\ln(A(S))$.
The value of $S$ may be recovered by recalling this definition and
since $\Sigma$ is a strictly increasing function of $S$ by
\eqref{A3}. Also, we will denote $\delta_{\omega}$ as the absolute
change of $\Sigma$ across a shock wave of strength $\omega$.  More
specifically, if two states were separated by a shock of strength
$\omega$ the absolute change in $\Sigma$ across the shock would be
$\delta_{\omega}$ for either a $1$ or $3-$shock. Since we have shown
that the change in $\Sigma$ is independent on the base state and
dependent only on the strength of the wave, $\delta_{\omega}$ is
well defined.

\bigskip

\section{Interaction Estimates}\label{Section-InteractionEstimates}

    In this section we prove estimates for elementary wave interactions
with a method that follows the work by Nishida and Smoller, and
Temple in \cite{Nishida-Smoller} and \cite{Temple-Large}.  This
method is employed in order to simplify the estimates on the
variation in the entropy. The alternative approach, useing the wave
interaction potential $Var\left\{ln(\rho)\right\}$ introduced by Liu
and used in \cite{Smoller-Temple-Global-Solutions-Rel-Euler},
simplifies the estimates dealing with the first and third,
nonlinear, characteristic classes, but complicates the estimates
dealing with the entropy.

Consider the following three states,
$U_{L}=(\rho_{L},v_{L},\Sigma_{L})$,
$U_{M}=(\rho_{M},v_{M},\Sigma_{M})$, and
$U_{R}=(\rho_{R},v_{R},\Sigma_{R})$. We wish to estimate the
difference in the solutions of the three Riemann problems
$<U_{L},U_{M}>$, $<U_{M},U_{R}>$, and $<U_{L},U_{R}>$ with solutions
denoted by a $1$ subscript, $2$ subscript and $'$ respectively.
\begin{proposition}\label{Interaction-Estimates}
Let $\Omega$ be a simply connected compact set in $rs-$space.  Then
there exists a constant $C_{0}$, $1/2<C_{0}<1,$ such that for any
interaction $<U_{L},U_{M}>+<U_{M},U_{R}>\rightarrow<U_{L},U_{R}>$ in
$\Omega$ at any value of $\Sigma$, one of the following holds:
  \begin{eqnarray}
    i.)\phantom{4} & A=-\xi\leq 0, \phantom{333} 0\leq B \leq C_{0}\xi, \phantom{33} \nonumber\\
                     & \textrm{or}\nonumber\\
                     & B=-\xi\leq 0, \phantom{333} 0\leq A \leq C_{0}\xi, \phantom{33} \nonumber\\
    ii.) \phantom{4}            & A\leq 0, \textrm{ and } B\leq 0.\nonumber \phantom{4444444444444}
  \end{eqnarray}
  Where $A=\alpha'-\alpha_{1}-\alpha_{2}$ and
  $B=\beta'-\beta_{1}-\beta_{2}$ are change in the strengths of the $1$
  and $3$ shock waves in the solutions.
\end{proposition}
Here we note that after an interaction, the shock wave strength in
one family may increase, but this increase is uniformly bounded by a
corresponding decrease in shock strength for the opposite family.
\begin{proof}
These estimates are proven in Chapter
\ref{Interaction-Estimates-Chapter} by a systematic look at all
possible wave interactions.  Because the interactions are
independent of entropy level, we only consider interactions within
the first and third characteristic classes.  There are sixteen
unique incoming wave configurations and between one and four
possible outgoing wave configurations.  The main idea is that
\textit{after an interaction, there cannot be an overall increase in
the strengths of the shock waves}. This fact follows since as the
solution progresses forward in time, cancelations and merging of
shock and rarefaction waves of the same class lead to a decrease in
shock strength. For example, when a shock wave is weakened by a
rarefaction wave, a reflected shock wave is created in the opposite
family. This interaction may increase the total strength of the
shock waves in the opposite family, but the total gain in shock
strength is uniformly bounded by the loss in the weakened or
annihilated shock.

We choose the constant $C_{0}$ to be the maximum slope of the
largest shock wave curve that lies within the compact set $\Omega$
or $1/2$ in order to bound the constant below.  More specifically,
let $\overline{\omega}$ be the strongest largest shock wave possible
in $\Omega$.  Then we take $C_{0}$ to be
\begin{equation}\label{Definition-C_0}
  C_{0}=\max\left\{\frac{1}{2},\left.\frac{dr}{ds}\right|_{\overline{\omega}},\left.\frac{ds}{dr}\right|_{\overline{\omega}}\right\}.
\end{equation}
Finally, by Lemma \ref{Lemma-Smoller-Temple}, the slopes of the
shock wave curves in a compact set in the $rs-$plane are strictly
bounded away by $1$. Therefore, we conclude $C_{0}<1$.
\end{proof}

\begin{figure}
\begin{center}
  \includegraphics[width=5.5in]{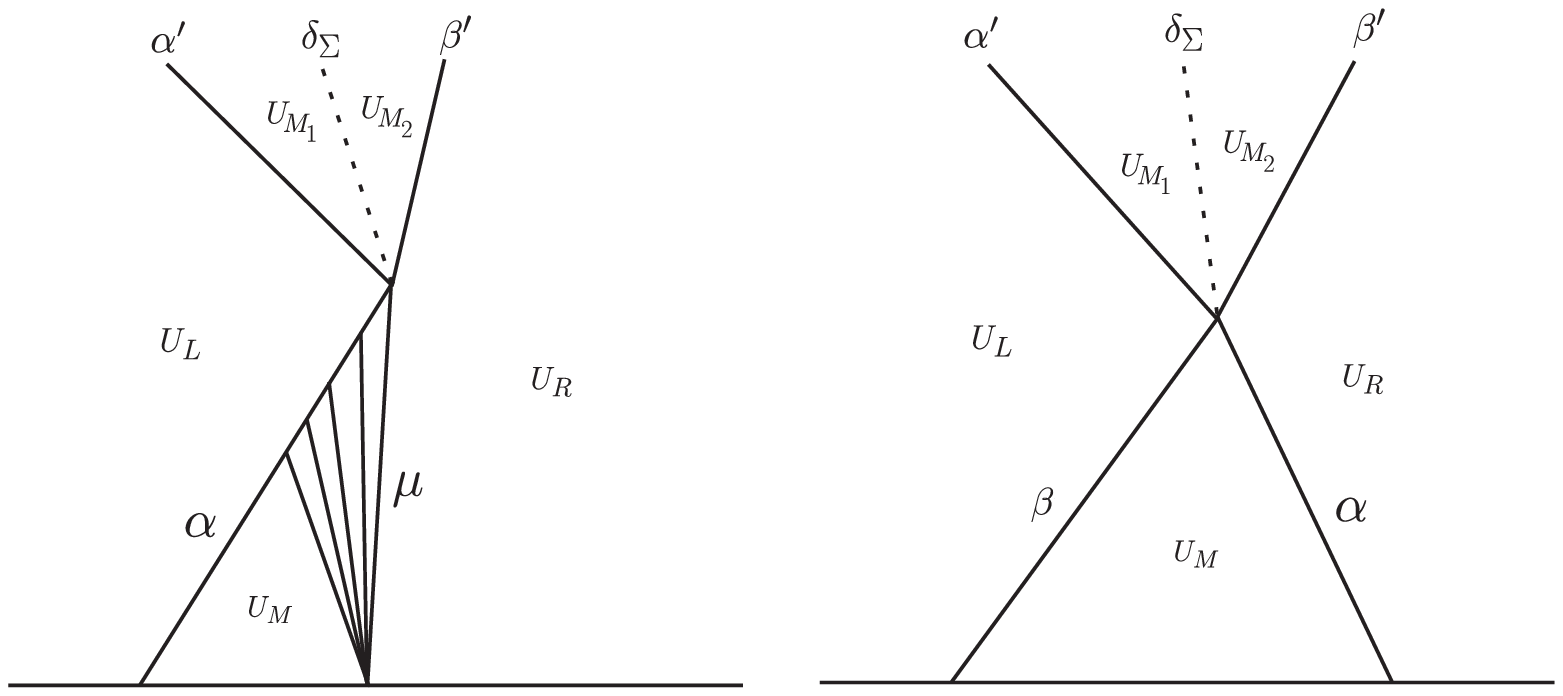}
  \caption{The creation of an entropy wave after elementary waves interact.}\label{Shock-Rarefaction-Interaction}
\end{center}
\end{figure}

For interactions in a compact set, the variation in $\Sigma$ across
a shock wave is uniformly bounded by a constant times the strength
of the shock. But, the variation in $\Sigma$ may increase after an
interaction because of the likely creation of an entropy wave.
Typically, across these waves the pressure is invariant and there is
a jump in density; however, under the assumption \eqref{p-a2rho},
there must be no jump in energy density. Thus, we cannot use
$\ln\left(\rho/\rho_{L}\right)$ or the change in the Riemann
invariants $r$ or $s$ as a measure of wave strength. It should be
noted that under certain interactions, such as an $i-$shock being
weakened by an incoming $i-$rarefaction wave, an entropy wave is
created with strength such that $S_{M_{2}}-S_{M_{1}}$ is equal to
the loss in entropy change across the shock, plus the change in the
entropy across the new shock wave in the opposite family.  We need a
way to bound the variation in the entropy waves, and it turns out
that this increase is bounded by a corresponding decrease in the
shock strengths.
\begin{proposition}\label{Entropy-Interaction}
  For every simply connected compact set $\Omega$ in $rs-$space, there exists a
  constant $M>0$ such that after every interaction in $\Omega$, at
  any value $\Sigma$ for the system \eqref{System-Conservation} with \eqref{EOS-Family} in the ultra-relativistic limit, the following holds:
  \begin{displaymath}
    |\delta'|-|\delta_{1}|-|\delta_{2}|+(\delta_{\alpha_{1}}+\delta_{\alpha_{2}}-\delta_{\alpha'})+(\delta_{\beta_{1}}+\delta_{\beta_{2}}-\delta_{\beta'})\leq
    -M(A+B).
  \end{displaymath}
\end{proposition}
\begin{proof}
  Choose $C_{0}$ so that Proposition \ref{Interaction-Estimates}
  holds.  Since $\Omega$ is a compact set, let
  \begin{displaymath}
    \overline{\omega}=\sup\left\{\|(r_{1},s_{1})-(r_{2},s_{2})\|:(r_{1},s_{1}),(r_{2},s_{2})\in\Omega\right\}.
  \end{displaymath}
  Then the strength of the largest shock wave in $\Omega$ is bounded by
  $\overline{\omega}$.  Furthermore, let $M=\left(1-C_{0}\right)^{-1}\overline{M}$,
  where
  \begin{equation}\label{Overline-M}
    \overline{M}=2\frac{d [\Sigma-\Sigma_{L}]}{d\omega}(\overline{\omega}),
  \end{equation}
  which is twice the largest rate of change of $\Sigma$ for all
  shocks contained in $\Omega$.  Also, since $[\Sigma-\Sigma_{L}](\omega)$ is positive and convex up, we have for
  strengths, $\omega'\geq\omega_{1}+\omega_{2}$,
  $\delta_{\omega'}\geq\delta_{\omega_{1}}+\delta_{\omega_{2}}$.

  The proof will be split into two cases, one for each of the two cases from
  Proposition \ref{Interaction-Estimates}.  First let us assume that
  $A\leq 0$ and $B\leq 0$. i.e.
  \begin{displaymath}
    \alpha'-\alpha_{1}-\alpha_{2}=-\xi_{\alpha}\leq 0 \phantom{33}
    \textrm{and} \phantom{33}
    \beta'-\beta_{1}-\beta_{2}=-\xi_{\beta}\leq 0.
  \end{displaymath}
  We have, $\alpha_{1}+\alpha_{2}-\xi_{\alpha}=\alpha'$ and hence,
  $\delta_{(\alpha_{1}+\alpha_{2}-\xi_{\alpha})}=\delta_{\alpha'}$.
  It follows that
  \begin{displaymath}
    \delta_{\alpha_{1}}+\delta_{\alpha_{2}}-\frac{1}{2}\overline{M}\xi_{\alpha}\leq\delta_{\alpha_{1}+\alpha_{2}}-\frac{1}{2}\overline{M}\xi_{\alpha}\leq\delta_{\alpha'}.
  \end{displaymath}
  Rearranging,
  \begin{equation}\label{Delta-Alpha-Estimate}
    \delta_{\alpha_{1}}+\delta_{\alpha_{2}}-\delta_{\alpha'}\leq\frac{1}{2}\overline{M}\xi_{\alpha}\leq-\frac{1}{2}MA,
  \end{equation}
  and similarly,
  \begin{equation}\label{Delta-Beta-Estimate}
    \delta_{\beta_{1}}+\delta_{\beta_{2}}-\delta_{\beta'}\leq\frac{1}{2}\overline{M}\xi_{\beta}\leq-\frac{1}{2}MB.
  \end{equation}
  The right hand inequalities follow from the fact that
  $\overline{M}<M$.
  Also, the change in entropy across the two Riemann problems before
  and the resulting one are equal:
  \begin{equation}\label{Net-Entropy-Change-Conserved}
    \delta_{\alpha'}+\delta'-\delta_{\beta'}=\delta_{\alpha_{1}}+\delta_{1}-\delta_{\beta_{1}}+\delta_{\alpha_{2}}+\delta_{2}-\delta_{\beta_{2}}.
  \end{equation}
  Rearranging \eqref{Net-Entropy-Change-Conserved} and using the previous estimates \eqref{Delta-Alpha-Estimate} and \eqref{Delta-Beta-Estimate}, we find
  \begin{equation}\label{Estimate-Delta-MA}
    \left(\delta'-\delta_{1}-\delta_{2}\right)+\left(\delta_{\beta_{1}}+\delta_{\beta_{2}}-\delta_{\beta'}\right)=\left(\delta_{\alpha_{1}}+\delta_{\alpha_{2}}-\delta_{\alpha'}\right)\leq
    -\frac{1}{2}MA
  \end{equation}
  and
  \begin{equation}\label{Estimate-Delta-MB}
    \left(\delta'-\delta_{1}-\delta_{2}\right)+\left(\delta_{\alpha'}-\delta_{\alpha_{1}}-\delta_{\alpha_{2}}\right)=\left(\delta_{\beta'}-\delta_{\beta_{1}}-\delta_{\beta_{2}}\right)\geq
    \frac{1}{2}MB.
  \end{equation}
  Adding the inequality \eqref{Delta-Alpha-Estimate} to
  \eqref{Estimate-Delta-MA},
  \begin{displaymath}
    \left(\delta'-\delta_{1}-\delta_{2}\right)+\left(\delta_{\alpha_{1}}+\delta_{\alpha_{2}}-\delta_{\alpha'}\right)+\left(\delta_{\beta_{1}}+\delta_{\beta_{2}}-\delta_{\beta'}\right)\leq-\frac{1}{2}MA-\frac{1}{2}MA=-MA,
  \end{displaymath}
  and adding $-1$ times the inequality \eqref{Delta-Beta-Estimate}
  to  \eqref{Estimate-Delta-MB},
  \begin{displaymath}
    \left(\delta'-\delta_{1}-\delta_{2}\right)+\left(\delta_{\alpha'}-\delta_{\alpha_{1}}+\delta_{\alpha_{2}}\right)+\left(\delta_{\beta'}-\delta_{\beta_{1}}-\delta_{\beta_{2}}\right)\geq\frac{1}{2}MB+\frac{1}{2}MB=MB.
  \end{displaymath}
  Therefore, after multiplying the entire inequality by $-1$ we
  have
  \begin{displaymath}
    -\left(\delta'-\delta_{1}-\delta_{2}\right)+\left(\delta_{\alpha_{1}}+\delta_{\alpha_{2}}-\delta_{\alpha'}\right)+\left(\delta_{\beta_{1}}+\delta_{\beta_{2}}-\delta_{\beta'}\right)\leq
    -MB.
  \end{displaymath}
  Since $0\leq -MA$ and $0\leq -MB$ by assumption, it follows that
  \begin{displaymath}
    |\delta'-\delta_{1}-\delta_{2}|+\left(\delta_{\alpha_{1}}+\delta_{\alpha_{2}}-\delta_{\alpha'}\right)+\left(\delta_{\beta_{1}}+\delta_{\beta_{2}}-\delta_{\beta'}\right)\leq-M(A+B).
  \end{displaymath}
  Furthermore, since $|\delta'-\delta_{1}-\delta_{2}|\geq
  |\delta'|-|\delta_{1}|-|\delta_{2}|$, we deduce,
  \begin{displaymath}
    |\delta'|-|\delta_{1}|-|\delta_{2}|+\left(\delta_{\alpha_{1}}+\delta_{\alpha_{2}}-\delta_{\alpha'}\right)+\left(\delta_{\beta_{1}}+\delta_{\beta_{2}}-\delta_{\beta'}\right)\leq-M(A+B).
  \end{displaymath}
  This concludes the proof of the first case.

  Now, without loss of generality assume $A=-\xi\leq 0$ and $0\leq
  B\leq C_{0}\xi$.  The mirror case when $0\leq A$ is similar.  As
  before, we can obtain the estimates,
  \begin{displaymath}
    \delta_{\alpha_{1}}+\delta_{\alpha_{2}}-\delta_{\alpha'}\leq\frac{1}{2}\overline{M}\xi
  \end{displaymath}
  and
  \begin{equation}\label{Delta-Alpha-Estimate-2}
    \left(\delta'-\delta_{1}-\delta_{2}\right)+\left(\delta_{\alpha_{1}}+\delta_{\alpha_{2}}-\delta_{\alpha'}\right)+\left(\delta_{\beta_{1}}+\delta_{\beta_{2}}-\delta_{\beta'}\right)\leq\overline{M}\xi.
  \end{equation}
From \eqref{Net-Entropy-Change-Conserved} we have,
    \begin{displaymath}
      -\left(\delta'-\delta_{1}-\delta_{2}\right)+\left(\delta_{\alpha_{1}}+\delta_{\alpha_{2}}-\delta_{\alpha'}\right)+\left(\delta_{\beta'}-\delta_{\beta_{1}}+\delta_{\beta_{2}}\right)=0,
    \end{displaymath}
    and since $\beta'\geq\beta_{1}+\beta_{2}$, we have
    $\delta_{\beta_{1}}+\delta_{\beta_{2}}-\delta_{\beta'}\leq 0$,
    and so by adding this inequality twice,
    \begin{equation}\label{Delta-Beta-Estimate-2}
      -\left(\delta'-\delta_{1}-\delta_{2}\right)+\left(\delta_{\alpha_{1}}+\delta_{\alpha_{2}}-\delta_{\alpha'}\right)+\left(\delta_{\beta_{1}}+\delta_{\beta_{2}}-\delta_{\beta'}\right)\leq
      0.
    \end{equation}
Therefore, from \eqref{Delta-Alpha-Estimate-2} and
\eqref{Delta-Beta-Estimate-2},
    \begin{displaymath}
       |\delta'-\delta_{1}-\delta_{2}|+\left(\delta_{\alpha_{1}}+\delta_{\alpha_{2}}-\delta_{\alpha'}\right)+\left(\delta_{\beta_{1}}+\delta_{\beta_{2}}-\delta_{\beta'}\right)\leq \overline{M}\xi,
    \end{displaymath}
which as before, reduces to,
  \begin{displaymath}
    |\delta'|-|\delta_{1}|-|\delta_{2}|+\left(\delta_{\alpha_{1}}+\delta_{\alpha_{2}}-\delta_{\alpha'}\right)+\left(\delta_{\beta_{1}}+\delta_{\beta_{2}}-\delta_{\beta'}\right)\leq\overline{M}\xi.
  \end{displaymath}
 But,
 \begin{displaymath}
   \overline{M}\xi=M(1-C_{0})\xi=M(\xi-C_{0}\xi)\leq
   M(-A-B)=-M(A+B),
 \end{displaymath}
 where we used the fact $-C_{0}\xi\leq -B$ following from the
 assumption that $0\leq B\leq C_{0}\xi$ and $A=-\xi$.
\end{proof}

    \chapter{The Glimm Difference Scheme}\label{Glimm-Difference-Scheme}
    \thispagestyle{firstpage}
    \section{Introduction}
\indent \indent In $1965$, Glimm proved existence of solutions to
general systems of strictly hyperbolic conservation laws with
genuinely non-linear or linearly degenerate characteristic fields,
\cite{Glimm}. To obtain existence, Glimm needed to restrict to
initial data with sufficiently small total variation to avoid having
to rule out the possibility that the complicated, global nonlinear
structure of the conservation law might create finite time blow up
of the solution or approximation scheme. His method takes a
piecewise constant approximate solution at one time step and uses
numerous solutions to Riemann problems, defined at each point of
discontinuity, to evolve the solution to a later time. After the
approximate solution is brought forward in time, the solution is
randomly sampled and a new piecewise constant approximate solution
is obtained. A fascinating consequence is that one cannot choose any
sequence of sample points to choose the states used for the new
piecewise constant function at each time step, but rather must
sample outside a set of measure zero in the space of all possible
choices.

One way this scheme may break down for general systems of hyperbolic
conservation laws is that Riemann problems may not have solutions if
the initial left and right states are sufficiently far apart.  A
canonical example of this phenomenon occurs in the $p-$system which
models a classical isentropic gas in Lagrangian coordinates. For
this system, if the difference in velocity of the two initial states
is sufficiently large, all the gas will be pulled from the region in
between the two states forming a vacuum, \cite{Smoller}.  This
possible complication and issues with large scale non-linearities,
led Glimm to prove existence for initial data with small variation.
He showed that in this case, the total possible increase in
variation in the approximate solution is bounded by a corresponding
decrease in a quadratic functional. Thus, having a bound on the
total variation in the solution showed that the Riemann problems
used in evolving the approximate solution can be defined for all
time.

In our case we prove a large data existence theorem; there is no
restriction on the ``smallness" of the initial conditions. In our
existence proof, we will not need to use a quadratic functional to
bound the total variation, because the geometric structure of the
shock and rarefaction curves in $rs-$space do not allow the
approximate solution to behave badly in the large. In this section,
we will introduce the Glimm scheme and use it to construct solutions
to \eqref{System-Conservation}.

\bigskip

\section{Glimm Difference Scheme}
We say $U(x,t)$ is a weak solution of \eqref{System-Conservation}
with initial data $U_{0}(x)$, if for all
$\varphi\in\mathcal{C}_{0}^{1}\left[\mathbb{R}^{+},\mathbb{R}\right]$
the following holds:
\begin{displaymath}
    \int_{0}^{\infty}\!\!\int_{-\infty}^{\infty}U\varphi_{t}+F(U)\varphi_{x}dxdt+\int_{-\infty}^{\infty}U(x,0)\varphi(x,0)dx=0.
\end{displaymath}

We begin by partitioning space into intervals of length $\Delta x$
and time into intervals of length $\Delta t$. In order to keep
neighboring Riemann problems from colliding, we impose the following
CFL condition:
\begin{displaymath}
  \frac{\Delta x}{\Delta t}>1>|\lambda_{i}|, \phantom{333} i=1,2,3.
\end{displaymath}
Note for $1<\gamma<2$ this condition is satisfied since the
characteristic speeds \eqref{Characteristic-Speeds} are bounded
above and below by $1$ and $-1$.

We inductively define our approximate solution.  To begin suppose
that we have an approximate solution at time $t=n\Delta t$,
$U(x,n\Delta t)$, which is constant on the intervals, $\left(k\Delta
x,(k+2)\Delta x\right)$, where $k+n$ is odd.  At each point
$x=k\Delta x$ a Riemann problem is defined.  Solve each Riemann
problem for time $t=\Delta t$.  This evolves our approximate
solution from $t=n\Delta t$ to $t=(n+1)\Delta t$.  To finish, we
must construct a new piecewise constant function at time
$t=(n+1)\Delta t$.  Choose $a\in\left[-1,1\right]$ and define,
$U(x,(n+1)\Delta t)=U((k+1+a)\Delta x, (n+1)\Delta t-)$ for
$x\in(k\Delta x,(k+2)\Delta x)$ and $k+n+1$ odd.  The term $\Delta
t-$ denotes the lower limit.

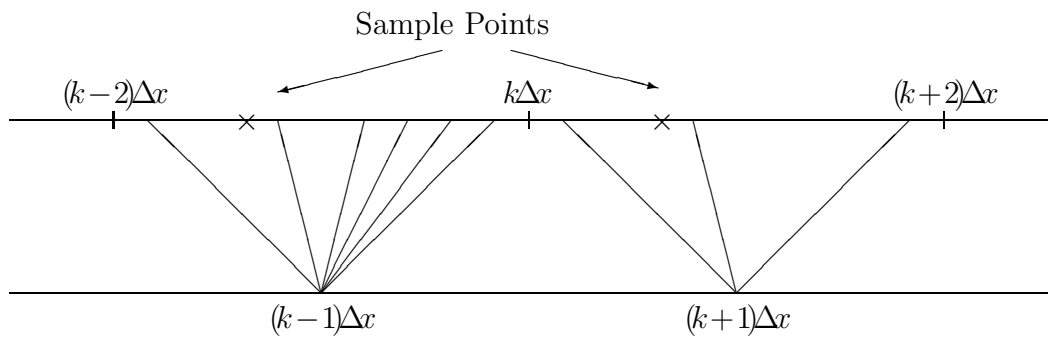
\begin{figure}
  \begin{center}
    \setlength{\unitlength}{2.3mm}
    \begin{picture}(0,20)
      \put(-30,13){\line(1,0){60}}
      \put(-30,3){\line(1,0){60}}
      \put(-12,3){\line(-1,1){10}}
      \put(-12,3){\line(-1,4){2.5}}
      \put(-12,3){\line(1,4){2.5}}
      \put(-12,3){\line(1,2){5}}
      \put(-12,3){\line(3,4){7.5}}
      \put(-12,3){\line(1,1){10}}
      \put(12,3){\line(-1,1){10}}
      \put(12,3){\line(-1,4){2.5}}
      \put(12,3){\line(1,1){10}}
      \put(-15,1){$(\!k\!-\!1\!)\!\Delta\! x$}
      \put(9,1){$(\!k\!+\!1\!)\!\Delta\! x$}
      \put(-27,14){$(\!k\!-\!2\!)\!\Delta\! x$}
      \put(21,14){$(\!k\!+\!2\!)\!\Delta\! x$}
      \put(-1.5,14){$k\!\Delta\! x$}
      \put(-24,12.5){\line(0,1){1}}
      \put(0,12.5){\line(0,1){1}}
      \put(24,12.5){\line(0,1){1}}
      \put(-17,12.4){$\times$}
      \put(7,12.4){$\times$}
      \put(-10,18){Sample Points}
      \put(-5,17){\vector(-4,-1){9.5}}
      \put(-1,17){\vector(4,-1){8.5}}
    \end{picture}
    \caption{Construction of Piecewise Constant Function\label{Approx-Piecewise}}
  \end{center}
\end{figure}

To begin this process at $t=0$, obtain a piecewise constant function
from the initial data $U_{0}(x)$ by again choosing
$a\in\left[-1,1\right]$ and defining, $U(x,0)=U_{0}((k+a)\Delta x)$
for $k$ odd.

Consider, $\theta\in\prod_{i=0}^{\infty}\left[-1,1\right]$.  In
other words,
$\theta=\left(\theta_{0},\theta_{1},\ldots,\theta_{n},\ldots\right)$
with $\theta_{i}\in\left[-1,1\right]$.  Then, for initial data, we
say $U_{\theta,\Delta x}(x,t)$ is the approximate solution given by
a mesh size of $\Delta x$ with sampling points at the
$n^{\textrm{th}}$ time step given by $\theta_{n}$.

\begin{figure}
  \begin{center}
    \setlength{\unitlength}{2.3mm}
    \begin{picture}(0,23)

      \put(-30,23){\line(1,0){60}}
      \put(-30,13){\line(1,0){60}}
      \put(-30,3){\line(1,0){60}}
      \put(-12,3){\line(-1,1){10}}
      \put(-12,3){\line(-1,4){2.5}}
      \put(-12,3){\line(1,4){2.5}}
      \put(-12,3){\line(1,2){5}}
      \put(-12,3){\line(3,4){7.5}}
      \put(-12,3){\line(1,1){10}}
      \put(0,13){\line(-1,1){10}}
      \put(0,13){\line(-1,4){2.5}}
      \put(0,13){\line(1,4){2.5}}
      \put(0,13){\line(1,2){5}}
      \put(0,13){\line(3,4){7.5}}
      \put(0,13){\line(1,1){10}}
      \put(12,3){\line(-1,1){10}}
      \put(12,3){\line(-1,4){2.5}}
      \put(12,3){\line(1,1){10}}

      \put(-24,13){\line(1,3){3.33}}
      \put(-24,13){\line(-1,4){2.5}}
      \put(-24,13){\line(1,1){10}}
      \put(24,13){\line(-1,1){10}}
      \put(24,13){\line(-1,4){2.5}}
      \put(24,13){\line(1,2){5}}

      \put(-15,1){$(\!k\!-\!1\!)\!\Delta\! x$}
      \put(9,1){$(\!k\!+\!1\!)\!\Delta\! x$}
      \put(-27,1){$(\!k\!-\!2\!)\!\Delta\! x$}
      \put(21,1){$(\!k\!+\!2\!)\!\Delta\! x$}
      \put(-1.5,1){$k\!\Delta\! x$}

      \multiput(-24,2.5)(0,10){3}{\line(0,1){1}}
      \multiput(0,2.5)(0,10){3}{\line(0,1){1}}
      \multiput(24,2.5)(0,10){3}{\line(0,1){1}}

      \put(-17,12.5){$\times$}
      \put(12,12.5){$\times$}
      \put(1,22.5){$\times$}
      \put(-2,2.5){$\times$}
      \put(-28,2.5){$\times$}
      \put(28,22.5){$\times$}

      \thicklines
      \qbezier(-27.5,3)(-22.5,7.75)(-16.5,13)
      \qbezier(-16.5,13)(-8.25,7.75)(-1.5,3)
      \qbezier(-1.5,3)(6,8)(12.5,13)
      \qbezier(12.5,13)(19.5,17.5)(28.5,23)
      \qbezier(-27.5,3)(-28.5,3.75)(-30,5)
      \qbezier(28.5,23)(29.5,22.32)(30,22)

      \put(1,6.5){$J$}

    \end{picture}
    \caption{I-Curve $J$.\label{Figure-I-Curve-J}}
  \end{center}
\end{figure}
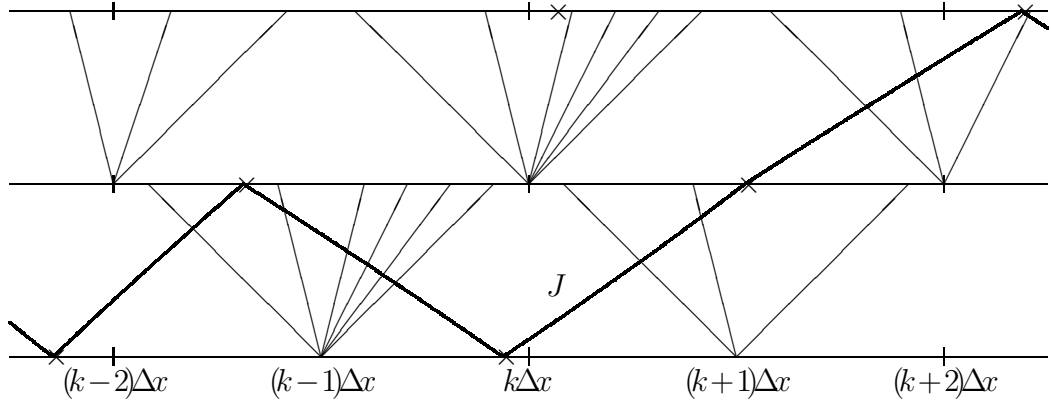

In order to estimate the change in the variation of our approximate
solutions, we will define piecewise linear, space-like curves,
called I-curves, which connect sample points at different time
levels.  If an I-curve $J$ passes through the sampling point
$((k+\theta_{n})\Delta x,n \Delta t)$, then on the right $J$ is only
allowed to connect to $((k+1+\theta_{n\pm 1})\Delta x,(n \pm 1)
\Delta t)$ and on the left to $((k-1+\theta_{n\pm 1})\Delta x,(n\pm
1) \Delta t)$.

We consider two functionals defined on $I-$curves and will analyze
how the functionals change as we change from one $I-$curve to
another. This will allow us to estimate the change in variation of
the approximate solution as it is evolved using the Glimm scheme. We
define for an $I-$curve $J$:
\begin{equation}\label{I-Curve-Fntl-F}
  F(J)=\sum_{J}\alpha_{i}+\sum_{J}\beta_{i}+V
\end{equation}
and
\begin{equation}\label{I-Curve-Fntl-L}
  L(J)=\sum_{J}\left(\alpha_{i}-M_{0}\delta_{\alpha_{i}}\right)+\sum_{J}\left(\beta_{i}-M_{0}\delta_{\beta_{i}}\right)-M_{0}\sum_{j}|\delta|+V,
\end{equation}
where the sums are taken over all waves, or fractions of them in the
case of rarefaction waves, that cross $J$. The constant $M_{0}$ will
be chosen later and $V=Var\left\{U_{0}(\cdot)\right\}$ is the
variation of the initial data.

The main problem in our analysis is to show that the variation in
the entropy waves stays bounded for all time.  To do this we need to
bound the possible change in $\Sigma$ across shock waves.  This is
accomplished by first showing that the variation in $r$ and $s$
stays finite for all time.  This implies that all the interactions,
as projected onto the $rs-$plane, occur in a compact set.  Thus,
there is a largest possible shock strength in this compact set, and
using the fact that the derivative of the entropy change as a
function of wave strength is monotone increasing, there is a
constant such that the entropy change is bounded by a constant times
the wave strength.  Moreover, we can then use Proposition
\ref{Entropy-Interaction} to estimate the increase in the variation
in entropy in our approximate solutions.

\bigskip

\section{Estimates on Approximate Solutions}

For initial data $U_{0}(x)$ and corresponding approximate solution
$U_{\theta,\Delta x}(x,t)$, define $\overline{U}_{0}(x)$ and
$\overline{U}_{\theta,\Delta x}(x,t)$ as the initial data and
approximate solutions viewed as functions of $r$ and $s$ only.  The
first estimate will show that the variation in the Riemann
invariants across an I-curve $J$ is bounded above by the functional
$F(\cdot)$ on $J$.

    \begin{proposition}\label{Var-rs-J-Bndd-FJ}
Let $\overline{U}_{0}(\cdot)$ be of finite variation. If the
approximate solution $\overline{U}_{\theta,\Delta x}(x,t)$ is
defined on an I-curve $J$, then,
        \begin{equation}\label{Rarefaction-Bndd-by-Shock}
          Var_{rs}(J)\leq 4F(J).
        \end{equation}
    \end{proposition}
    \begin{proof}
Let $Var^{-}_{r}(J)$ denote the variation across $J$ given by a
decrease in $r$.  The only waves that contribute to the decrease in
$r$ are $1$ and $3-$shocks.  Furthermore, we have
      \begin{equation}\label{eqn-VarrMinus-al-be}
        Var_{r}^{-}(J)\leq \sum_{J}\alpha_{i}+\sum_{J}\beta_{i},
      \end{equation}
where the sum is over all waves of the particular type crossing $J$.
We can similarly define $Var^{+}_{r}(J)$ as the variation given by
increases of $r$ across elementary waves. The only increase is given
by $1-$rarefaction waves,
      \begin{equation}\label{eqn-VarrPlus-mu}
        Var_{r}^{+}(J)=\sum_{J}\mu_{i}.
      \end{equation}
Following this line of reasoning for $s$, we also have
      \begin{equation}\label{eqn-VarsMinus-al-be}
        Var_{s}^{-}(J)\leq \sum_{J}\alpha_{i}+\sum_{J}\beta_{i}
      \end{equation}
and
      \begin{equation}\label{eqn-VarsPlus-eta}
        Var_{s}^{+}(J)=\sum_{J}\eta_{i}.
      \end{equation}
The initial data $\overline{U}_{0}$ may be written as a function of
the Riemann invariants $r$ and $s$,
$\overline{U}_{0}(x)=(r_{0}(x),s_{0}(x))$. Since
$\overline{U}_{0}(\cdot)$ is of finite variation, the following
limits must exist:
      \begin{displaymath}
        \lim_{x\rightarrow \pm\infty}r_{0}(x)=r^{\pm}, \phantom{444} \lim_{x\rightarrow
        \pm\infty}s_{0}(x)=s^{\pm}.
      \end{displaymath}
Indeed, let $\{x_{n}\}_{n=0}^{\infty}$ be an increasing sequence of
real numbers such that $x_{n}\rightarrow \infty$ as $n \rightarrow
\infty$. Then,
\begin{displaymath}
  \sum_{n=1}^{\infty}|r_{0}(x_{n})-r_{0}(x_{n-1})|\leq
  Var\{\overline{U}_{0}(\cdot)\}<\infty.
\end{displaymath}
Hence the sequence $r_{0}(x_{n})$ is Cauchy, which converges to a
finite limit $r^{+}$.  The other cases are entirely similar.

For any I-curve $J$, the end states at $\pm\infty$ are given by
$(r^{\pm},s^{\pm})$. From this we obtain
    \begin{displaymath}
      |Var^{+}_{r}(J)-Var^{-}_{r}(J)|=|r^{+}-r^{-}|\leq V,
    \end{displaymath}
and hence,
    \begin{displaymath}
      Var^{+}_{r}(J)\leq Var^{-}_{r}(J)+V.
    \end{displaymath}
Using $(\ref{eqn-VarrMinus-al-be})$ and $(\ref{eqn-VarrPlus-mu})$,
    \begin{displaymath}
      \sum_{J}\mu_{i}\leq \sum_{J}\alpha_{i}+\sum_{J}\beta_{i}+V.
    \end{displaymath}
Similarly from $(\ref{eqn-VarsMinus-al-be})$ and
$(\ref{eqn-VarsPlus-eta})$,
    \begin{displaymath}
      \sum_{J}\eta_{i}\leq \sum_{J}\alpha_{i}+\sum_{J}\beta_{i}+V.
    \end{displaymath}
Combining these together we have,
    \begin{displaymath}
      \sum_{J}\mu_{i}+\sum_{J}\eta_{i} \leq 2\left(\sum_{J}\alpha_{i}+\sum_{J}\beta_{i}+V\right).
    \end{displaymath}
Thus,
    \begin{eqnarray}
      Var_{rs}(J) & \leq &
      2\left(\sum_{J}\alpha_{i}+\sum_{J}\beta_{i}\right)+\sum_{J}\mu_{i}+\sum_{J}\eta_{i},\nonumber\\
                  & \leq & 4\left(\sum_{J}\alpha_{i}+\sum_{J}\beta_{i}\right)+2V,\nonumber\\
                  & \leq & 4\left(\sum_{J}\alpha_{i}+\sum_{J}\beta_{i}+V\right),\nonumber\\
                  & \leq & 4F(J).\phantom{\left(\sum_{J}\alpha_{i}\right)}\nonumber
    \end{eqnarray}
    \end{proof}

    We will now show that the functional $F(\cdot)$ on the I-curves
    is non-increasing.  We define a partial ordering on the I-curves
    by saying that $J\prec J'$ if the curve $J'$ never lies below
    the curve $J$.  Furthermore, we say that $J'$ is an immediate
    successor to $J$ if $J\prec J'$ and $J$ and $J'$ share all the same sample points
    except for one.  It is clear that for any pair of I-curves such
    that $J\prec J'$, there is a sequence of immediate successors
    that begins at $J$ and ends at $J'$.  The next proposition shows
    that if our approximate solution is defined on an I-curve, it
    can be defined for all following I-curves.

\begin{proposition}\label{I-Curves-Induct-rs}
    Let $J$ and $J'$ be $I-$curves, $J\prec J'$, and suppose that $J$ is in the
    domain of definition of $\overline{U}_{\theta,\Delta x}$. If
    $F(J)<\infty$, then $J'$ is in the domain of definition of $\overline{U}_{\Delta
    x,\theta}$, and $F(J')\leq F(J)$. Moreover, if
    $Var_{rs}\left\{U_{0}(\cdot)\right\}<\infty$
    then $\overline{U}_{\theta,\Delta x}$ can be defined for $t\geq 0$.
\end{proposition}

\begin{figure}
  \begin{center}
    \setlength{\unitlength}{2.3mm}
    \begin{picture}(0,23)

      \put(-30,23){\line(1,0){60}}
      \put(-30,13){\line(1,0){60}}
      \put(-30,3){\line(1,0){60}}
      \put(-12,3){\line(-1,1){10}}
      \put(-12,3){\line(-1,4){2.5}}
      \put(-12,3){\line(1,4){2.5}}
      \put(-12,3){\line(1,2){5}}
      \put(-12,3){\line(3,4){7.5}}
      \put(-12,3){\line(1,1){10}}
      \put(0,13){\line(-1,1){10}}
      \put(0,13){\line(-1,4){2.5}}
      \put(0,13){\line(1,4){2.5}}
      \put(0,13){\line(1,2){5}}
      \put(0,13){\line(3,4){7.5}}
      \put(0,13){\line(1,1){10}}
      \put(12,3){\line(-1,1){10}}
      \put(12,3){\line(-1,4){2.5}}
      \put(12,3){\line(1,1){10}}

      \put(-24,13){\line(1,3){3.33}}
      \put(-24,13){\line(-1,4){2.5}}
      \put(-24,13){\line(1,1){10}}
      \put(24,13){\line(-1,1){10}}
      \put(24,13){\line(-1,4){2.5}}
      \put(24,13){\line(1,2){5}}

      \put(-15,1){$(\!k\!-\!1\!)\!\Delta\! x$}
      \put(9,1){$(\!k\!+\!1\!)\!\Delta\! x$}
      \put(-27,1){$(\!k\!-\!2\!)\!\Delta\! x$}
      \put(21,1){$(\!k\!+\!2\!)\!\Delta\! x$}
      \put(-1.5,1){$k\!\Delta\! x$}

      \multiput(-24,2.5)(0,10){3}{\line(0,1){1}}
      \multiput(0,2.5)(0,10){3}{\line(0,1){1}}
      \multiput(24,2.5)(0,10){3}{\line(0,1){1}}

      \put(-17,12.5){$\times$}
      \put(12,12.5){$\times$}
      \put(1,22.5){$\times$}
      \put(-2,2.5){$\times$}
      \put(-28,2.5){$\times$}
      \put(28,22.5){$\times$}

      \thicklines
      \qbezier(-27.5,3)(-22.5,7.75)(-16.5,13)
      \qbezier(-16.5,13)(-7.5,18)(1.5,23)
      \qbezier(-16.5,13)(-8.25,7.75)(-1.5,3)
      \qbezier(1.5,23)(7,18)(12.5,13)
      \qbezier(-1.5,3)(6,8)(12.5,13)
      \qbezier(12.5,13)(19.5,17.5)(28.5,23)
      \qbezier(-27.5,3)(-28.5,3.75)(-30,5)
      \qbezier(28.5,23)(29.5,22.32)(30,22)

      \put(-12,17){$J'$}
      \put(1,6.5){$J$}

    \end{picture}
    \caption{Immediate Successor I-Curves $J'$ and $J$.\label{Figure-Immediate-Successor}}
  \end{center}
\end{figure}
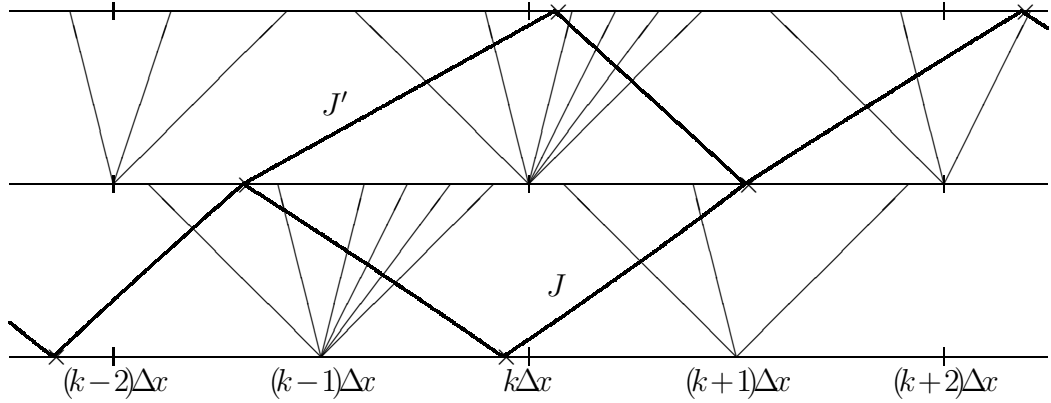

    \begin{proof}
      We proceed by induction.  Suppose first that $J'$ is an
      immediate successor to $J$.  Then the difference $F(J')-F(J)$,
      is given by the change in shock wave strengths across the
      diamond enclosed by $J'$ and $J$. See Figure
      \ref{Figure-Immediate-Successor}. This is a consequence of
      the fact that the waves the head into the diamond from the
      left and right solve the same Riemann problem as the outgoing
      waves in the new single Riemann problem.  If we denote
      $J'_{0}$ and $J_{0}$ as the diamond portion of $J'$ and $J$,
      we have,
      \begin{eqnarray}
        F(J')-F(J) & = &
        \sum_{J'}\alpha_{i}+\sum_{J'}\beta_{i}+V-\left(\sum_{J}\alpha_{i}+\sum_{J}\beta_{i}+V\right),\nonumber\\
                   & = & \sum_{J_{0}'}\alpha_{i}+\sum_{J_{0}'}\beta_{i}-\sum_{J_{0}}\alpha_{i}-\sum_{J_{0}}\beta_{i},\nonumber\\
                   & = &
                   \left(\alpha'-\alpha_{1}-\alpha_{2}\right)+\left(\beta'-\beta_{1}-\beta_{2}\right),\nonumber\\
                   & = & A+B\leq 0.\nonumber
      \end{eqnarray}
      The last line follows from Proposition \ref{Interaction-Estimates}.
      Thus,
      $F(J')\leq F(J)$ for immediate successors.  For any a general
      $J$ and $J'$ such that $J\prec J'$, we produce a sequence of immediate successors
      that take $J$ to $J'$.  At each step the functional $F$ is non-increasing,
      thus $F(J')\leq F(J)$ continues to hold.

      By Proposition \ref{Var-rs-J-Bndd-FJ},
      $Var_{rs}(J')\leq 4F(J')\leq 4F(J)$, so, $J'$ is in the domain
      of definition of $\overline{U}_{\theta,\Delta x}$.  Moreover,
      if $Var_{rs}\left\{\overline{U}_{0}(\cdot)\right\}<\infty$,
      then $Var_{rs}(\mathbf{0})<\infty$ for the unique $I-$curve $\mathbf{0}$ that
      lies along the line $t=0$.  In order to show that
      $\overline{U}_{\Delta x,\theta}$ can be defined for $t\geq 0$,
      we must show that
      $Var_{rs}\left\{\overline{U}_{\theta,\Delta x}(\cdot,t)\right\}<\infty$ for
      all time.  But, this condition is equivalent to showing the
      variation across any $I-$curve $J$ is always finite.  Since
      for any $I-$curve $J$,
\begin{displaymath}
Var_{rs}(J)\leq 4F(J)\leq 4F(\mathbf{0})\leq 8
Var_{rs}\left\{U_{0}(\cdot)\right\},
\end{displaymath}
the result follows.
    \end{proof}

Again, Proposition \ref{I-Curves-Induct-rs} shows that the variation
of our approximate solution in the variables $r$ and $s$ is finite.
Thus, there exists a compact set in the $rs-$plane that contains all
the interactions in our approximate solution.

\begin{corollary}\label{Corollary-Compact-Set}
  Suppose that $Var_{rs}\left\{U_{0}(\cdot)\right\}<\infty$.  Then
  there exists a simply connected compact set $\Omega$ in the $rs-$plane such that
  all possible interactions are contained in $\Omega$.
\end{corollary}
\begin{proof}
From Proposition \ref{Var-rs-J-Bndd-FJ} and Proposition
\ref{I-Curves-Induct-rs} we know that for any I-curve $J$,
  \begin{displaymath}
    Var_{rs}(J)<4F(J)<4F(\mathbf{0})<8Var_{rs}\left\{U_{0}(\cdot)\right\}=N<\infty.
  \end{displaymath}
  Thus, the distance between any to states occurring anywhere in our
  approximate solution is bounded by $N$.  Consider the left limit
  state of $\overline{U}_{0}(\cdot)$, $(r^{-},s^{-})$.  Therefore, all
  states must be contained within, $B_{2N}(r^{-},s^{-})$, the ball of radius $2N$ centered
  around $(r^{-},s^{-})$.
\end{proof}

 Now, we show that the variation of our
approximate solution, including the variation in $\Sigma$, is
bounded above by the functional $L(\cdot)$.

    \begin{proposition}\label{Var-Bndd-LJ}
Suppose $Var\left\{U_{0}(\cdot)\right\}<\infty$ and $J$ is an
$I-$curve that is in the domain of definition of $U_{\theta,\Delta
x}$. Then there exists constants $M_{0}>0$ and $K>0$, independent of
$\Delta x$ and $\theta$, such that,
      \begin{equation}
        Var(J)\leq K\cdot L(J).
      \end{equation}
    \end{proposition}
    \begin{proof}
      The variation across the $I-$curve $J$ is bounded by
      \begin{eqnarray}
        Var(J) & \leq & Var(\textrm{Shock Waves})+Var(\textrm{Rarefaction
        Waves})\nonumber\\
               &      &+Var(\Sigma \textrm{-Waves})+Var(\Sigma \textrm{
        across Shocks}).\nonumber
      \end{eqnarray}
Since $Var_{rs}\left\{\overline{U}_{0}(\cdot)\right\}\leq
Var\left\{U_{0}(\cdot)\right\}<\infty$, we have from Corollary
\ref{Corollary-Compact-Set} that all the interactions projected into
the $rs-$plane occur in a compact set $\Omega$. Therefore there
exists a constant $\overline{M}>0$ such that for a shock wave of
strength $\omega$, $\delta_{\omega}\leq\overline{M}\omega$.  Let
$M=(1-C_{0})^{-1}\overline{M}$ as in Proposition
\ref{Entropy-Interaction}.  Since, $\overline{M}<M$ we have for a
shock wave of strength $\omega$, $\delta_{\omega}<M\omega$.

From the proof of Proposition \ref{Var-rs-J-Bndd-FJ}, we can bound
the variation from the shock waves and rarefaction waves by the
shock waves crossing $J$ and the initial variation $V$. Thus,
\begin{eqnarray}
  Var(J) & \leq &
  2\left(\sum_{J}\alpha_{i}+\sum_{J}\beta_{i}\right)+\sum_{J}\mu_{i}+\sum_{J}\eta_{i}\phantom{\left(\sum\right)}\nonumber\\
         &      &
         \phantom{44444444}+\sum_{J}|\delta|+\sum_{J}\delta_{\alpha_{i}}+\sum_{J}\delta_{\beta_{i}},\phantom{\sum_{J}|\delta|}\nonumber\\
         & \leq &
         4\left(\sum_{J}\alpha_{i}+\sum_{J}\beta_{i}+V\right)+\sum_{J}|\delta|+M\left(\sum_{J}\alpha_{i}+\sum_{J}\beta_{i}\right),\nonumber\\
         & \leq &
         (4+M)\left(\sum_{J}\alpha_{i}+\sum_{J}\beta_{i}+V\right)+\sum_{J}|\delta|.\nonumber
\end{eqnarray}
Let $M_{0}\leq 1/2M$.  Then,
\begin{displaymath}
  M_{0}\delta_{\omega}\leq
  \frac{1}{2M}\delta_{\omega}\leq\frac{1}{2M}\left(M\omega\right)\leq\frac{1}{2}\omega.
\end{displaymath}
Thus, for a shock wave of strength $\omega$,
\begin{displaymath}
  \omega\leq 2(\omega-M_{0}\delta_{\omega}).
\end{displaymath}
Using this, we find,
\begin{displaymath}
  Var(J)\leq
  2(4+M)\left(\sum_{J}\left(\alpha_{i}-M_{0}\delta_{\beta_{i}}\right)+\sum_{J}\left(\beta_{i}-M_{0}\delta_{\beta_{i}}\right)+V\right)+\sum_{J}|\delta|.
\end{displaymath}
Finally, we put the sum of the strengths of the entropy waves
inside,
\begin{displaymath}
  Var(J)\leq
  2(4+M)\left(\sum_{J}\left(\alpha_{i}-M_{0}\delta_{\beta_{i}}\right)+\sum_{J}\left(\beta_{i}-M_{0}\delta_{\beta_{i}}\right)+M_{0}\sum_{J}|\delta|+V\right).
\end{displaymath}
We can do this because,
\begin{displaymath}
  M_{0}\cdot 2(4+M)\geq 2MM_{0}\geq 1.
\end{displaymath}
Therefore,
\begin{displaymath}
  Var(J)\leq K\cdot L(J),
\end{displaymath}
with $K=2(4+M)$.
\end{proof}

    \begin{proposition}\label{Induct-J-I-Curves}
Suppose that $Var\left\{U_{0}(\cdot)\right\}<\infty$ and $J$, $J'$
are I-curves such that $J\prec J'$ and $L(J)<\infty$.  Then $J'$ is
in the domain of definition of $U_{\theta,\Delta x}(x,t)$,
$L(J')\leq L(J)$ and $U_{\Delta x,\theta}(x,t)$ is defined for
$t\geq 0$.
    \end{proposition}

    \begin{proof}

Since
$Var_{rs}\left\{\overline{U}_{0}(\cdot)\right\}<Var\left\{U_{0}(\cdot)\right\}<\infty$
there exists a compact set $\Omega$ that contains all possible
interactions.  Define $M$ as in Proposition
\ref{Entropy-Interaction} and take $M_{0}\leq 1/2M$.  As with
Proposition \ref{I-Curves-Induct-rs}, we prove the result by
induction on the $I$ curves.  First let $J'$ be an immediate
successor to $J$. Let $J'_{0}$ and $J_{0}$ be the parts of $J'$ and
$J$ that bound the diamond formed by $J$ and $J'$. Using this and
the definition of $L(J)$,
 \begin{eqnarray}
      L(J')-L(J) & \leq &
      \left[\sum_{J'_{0}}\left(\alpha_{i}-M_{0}\delta_{\alpha_{i}}\right)+\sum_{J'_{0}}\left(\beta_{i}-M_{0}\delta_{\beta_{i}}\right)+M_{0}\sum_{J'_{0}}|\delta|\right]\nonumber\\
                 &      & \phantom{44}-\left[\sum_{J_{0}}\left(\alpha_{i}-M_{0}\delta_{\alpha_{i}}\right)+\sum_{J_{0}}\left(\beta_{i}-M_{0}\delta_{\beta_{i}}\right)+M_{0}\sum_{J_{0}}|\delta|\right],\nonumber\\
                 &=&\left(\alpha'-\alpha_{1}-\alpha_{2}\right)+\left(\beta'-\beta_{1}-\beta_{2}\right)+M_{0}\left(\delta_{\alpha_{1}}+\delta_{\alpha_{2}}-\delta_{\alpha'}\right)\nonumber\\
                 &&+M_{0}\left(\delta_{\beta_{1}}+\delta_{\beta_{2}}-\delta_{\beta'}\right)+M_{0}\left(|\delta'|-|\delta_{1}|-|\delta_{2}|\right).\nonumber
    \end{eqnarray}
    Now we refer to Proposition \ref{Interaction-Estimates} and
    \ref{Entropy-Interaction}.  We see that the first two terms are
    equal to $(A+B)$ and the others are bounded above by $-M(A+B)$.
    Putting this together,
    \begin{displaymath}
      L(J')-L(J)\leq (A+B)-MM_{0}(A+B)\leq\frac{1}{2}(A+B)\leq 0.
    \end{displaymath}
For immediate successors, we have $L(J')\leq L(J)$.  Moreover, by
Proposition \ref{Var-Bndd-LJ} we have that the variation along $J'$
is bounded by $L(J')$ and hence $L(J)$.  Thus, $J'$ is in the domain
of definition of $U_{\Delta x,\theta}$.

For general $J$ and $J'$ such that $J\prec J'$, the same conclusion
holds by constructing a sequence of immediate successors to move
from $J$ to $J'$.  Along each step, the results above continue to
hold.

Finally, if $Var\left\{U_{0}(\cdot)\right\}<\infty$, we have
$L(0)<\infty$ and for any $I-$curve $J$, $L(J)\leq L(0)$.  Which we
can conclude that
\begin{displaymath}
 Var(J)\leq 2(4+M)L(J)\leq 2(4+M)L(0)<\infty,
\end{displaymath}
so our approximate solution can be defined for $t\geq 0$.
    \end{proof}

    \chapter{Existence Theorem for Two Gasses}\label{Existence}
    \thispagestyle{firstpage}
    In this chapter, we use Glimm's Theorem \cite{Glimm} to prove
existence of solutions to \eqref{System-Conservation} in the
ultra-relativistic limit with an equation of state of the form
\eqref{EOS-Family}. It should be noted that for $\theta$ fixed and
$x_{n}=1/2^{n}$, the set of approximate solutions
$\left\{U_{\theta,\Delta x_{n}}(x,t)\right\}^{\infty}_{n=1}$ has
uniformly bounded variation by Proposition \ref{Var-Bndd-LJ}.
Furthermore, since the variation is bounded and each approximate
solution has the same limits at infinity, the sup norm is also
uniformly bounded. The approximate solutions are $\mathbf{L}^{1}$
Lipschitz in time too since
\begin{eqnarray}
  \|U_{\theta,\Delta x}(\cdot,t)-U_{\theta,\Delta x}(\cdot,s)\|_{\mathbf{L}^{1}}&\leq
  &  C \left\{\textrm{Sum of all wave
  strengths}\right\}\cdot\nonumber\\
  & & \phantom{4444} \left\{\textrm{Maximum Speed of
  Wavefronts}\right\}<C'|t-s|.\nonumber
\end{eqnarray}
At this point Helly's Theorem \cite{Bressan-Book} provides a
convergent subsequence, $U_{\theta,\Delta x_{n_{i}}}(x,t)$, that
converges to a function $U(x,t)$ with finite variation for each
fixed time. However at this time, there is no justification that
this limit function is actually a weak solution. Glimm's Theorem
guarantees that there exists a subsequence that converges to a weak
solution.

\bigskip

\section{Existence of Weak Solutions}

        \begin{theorem}[Glimm, $1965$]\label{Glimm}
      Assume that the approximate solution $U_{\theta,\Delta x_{i}}$
      satisfies,
      \begin{equation}
        Var\left\{U_{\theta,\Delta x_{i}}(\cdot,t)\right\}<N<\infty
      \end{equation}
      for $x_{i}=1/2^{i}$, $\theta\in A=\prod_{i=0}^{\infty}$, and all $t\geq 0$.  Then there exits a subsequence of mesh
      lengths $\Delta x_{i_{k}}$ such that $U_{\theta,\Delta
      x_{i_{k}}}\rightarrow U$ in $\mathbf{L}^{1}_{Loc}$ where $U(x,t)$ satisfies,
      \begin{displaymath}
        Var\left\{U(\cdot,t)\right\}<N.
      \end{displaymath}
      Furthermore, there exits a set of measure zero
      $\overline{A}\subset A$ such that if $\theta\in
      A-\overline{A}$ then $U(x,t)$ is a weak solution to
      \eqref{System-Conservation}.
    \end{theorem}

We now prove Theorem \ref{Main-Theorem} by showing that our
approximate solutions meet the assumptions of Glimm's Theorem.

    \begin{proof}
Assume the initial data satisfies, \eqref{Initial-Bound-Entropy},
\eqref{Initial-Bound-Density}, and \eqref{Initial-Bound-Velocity} We
will show that for all $\Delta x_{i}$ and sample points $\theta$,
    \begin{equation}\label{Bounded-Conserved-Variables}
      Var\left\{U_{\Delta x,\theta}(\cdot,t)\right\}<N<\infty,
    \end{equation}
where $U_{\theta,\Delta
x}(\rho(x,t),v(x,t),S(x,t))=(U_{1},U_{2},U_{3})_{\theta,\Delta x}$.
First we show that the variation in $\rho$, $v$, and $S$ is bounded
for all time in the approximate solutions.

    From Proposition \ref{Var-rs-J-Bndd-FJ} and Proposition
    \ref{I-Curves-Induct-rs} we have that the variation
of our approximate solution in $r$ and $s$ is uniformly bounded for
all time.  More specifically,

    \begin{eqnarray}
      Var_{rs}\left(\overline{U}_{\theta,\Delta x}(\cdot,t)\right) & < &
      4F(\mathbf{0}),\nonumber\\
      & < &
      4\left[\sum_{\mathbf{0}}\alpha_{i}+\sum_{\mathbf{0}}\beta_{i}+Var_{rs}(U_{0})\right],\nonumber\\
      & < & 8\cdot Var_{rs}(U_{0}(\cdot)).\nonumber
    \end{eqnarray}

    From this estimate, we show that the variation of
    \begin{displaymath}
      \ln\left(\rho\right) \phantom{4444} \textrm{and}
      \phantom{4444} \ln\left(\frac{1+v}{1-v}\right)
    \end{displaymath}
    are also bounded for all time.  Indeed by the definition of $r$
    and $s$,
    \begin{displaymath}
      \ln\left(\frac{1+v}{1-v}\right)=\frac{1}{2}(r+s),
    \end{displaymath}
    we have
    \begin{eqnarray}
      Var\left\{\ln\left(\frac{1+v(\cdot,t)}{1-v(\cdot,t)}\right)\right\} &
      = &
      \frac{1}{2}\sup_{N}\sum_{i=1}^{N}|(r(x_{i+1},t)+s(x_{i+1},t))-(r(x_{i},t)+s(x_{i},t))|,\nonumber\\
      & \leq &
      \frac{1}{2}\sup_{N}\sum_{i=1}^{N}|r(x_{i+1},t)-r(x_{i},t)|\nonumber\\
      & & \phantom{4444444} +\frac{1}{2}\sup_{N}\sum_{i=1}^{N}|s(x_{i+1},t)-s(x_{i},t)|,\nonumber\\
      & \leq & \frac{1}{2}Var_{rs}\left\{\overline{U}_{\Delta
      x,\theta}(\cdot,t)\right\}+\frac{1}{2}Var_{rs}\left\{\overline{U}_{\Delta
      x,\theta}(\cdot,t)\right\},\nonumber\\
      & \leq & 8\cdot Var\left\{U_{0}(\cdot)\right\}.\nonumber
    \end{eqnarray}
    Similarly, using
    \begin{displaymath}
      \ln(\rho)=\frac{1+a^{2}}{a}(s-r),
    \end{displaymath}
    we find,
    \begin{displaymath}
      Var\left\{\ln(\rho(\cdot,t))\right\}\leq
      16\left(\frac{1+a^{2}}{a}\right)Var\left\{U_{0}(\cdot)\right\}.
    \end{displaymath}

    Now, we show the variation in $\Sigma$ is bounded for all time
    in approximate solutions.  This is clear from Proposition
    \ref{Var-Bndd-LJ} and Proposition \ref{Induct-J-I-Curves}
    because there exists a constant $M$ so that
    \begin{displaymath}
      Var\left\{\Sigma_{\theta,\Delta x}(\cdot,t)\right\}\leq 2(4+M)L(\mathbf{0}).
    \end{displaymath}

    We can now show that the variation in $\rho$, $v$ and $S$ is
    bounded for all time.  Since
    $Var\left\{\ln(\rho(\cdot,t))\right\}<\infty$ for all $t>0$
    there exists a constant $b>0$ such that $\rho(x,t)<b$.  Let
    $c=max\left\{1,b\right\}$, then
    \begin{eqnarray}
      Var\left\{\rho(\cdot,t)\right\} & = &
      \sup_{N}\sum_{i=1}^{N}|\rho(x_{i+1},t)-\rho(x_{i},t)|,\nonumber\\
      & \leq & c\cdot\sup_{N}\sum_{i=1}^{N}|\ln(\rho(x_{i+1},t))-\ln(\rho(x_{i},t))|,\nonumber\\
      & \leq & c\cdot Var\left\{\ln(\rho(\cdot,t))\right\}.\nonumber
    \end{eqnarray}

    For $v$ we have,
    \begin{eqnarray}
      Var\left\{v(\cdot,t)\right\} & = &
      \sup_{N}\sum_{i=1}^{N}\left|v(x_{i+1},t)-v(x_{i},t)\right|,\nonumber\\
      & \leq & \frac{1}{2}\sup_{N}\sum_{i=1}^{N}\left|\ln\left(\frac{1+v(x_{i+1},t)}{1-v(x_{i+1},t)}\right)-\ln\left(\frac{1+v(x_{i},t)}{1-v(x_{i},t)}\right)\right|,\nonumber\\
      & \leq & \frac{1}{2}Var\left\{\ln\left(\frac{1+v(\cdot,t)}{1-v(\cdot,t)}\right)\right\}.\nonumber
    \end{eqnarray}
    The factor $1/2$ comes from the fact that the slope of the chord
    connecting the points,
\begin{displaymath}
  \left(v(x_{i+1},t),\ln\left(\frac{1+v(x_{i+1},t)}{1-v(x_{i+1},t)}\right)\right)
  \phantom{333} \textrm{and} \phantom{333}
  \left(v(x_{i},t),\ln\left(\frac{1+v(x_{i},t)}{1-v(x_{i},t)}\right)\right)
\end{displaymath}
is bounded below by $2$.

    For $S$ we need to find a constant $C$ such that
\begin{displaymath}
    \left|S(x,t)-S(y,t)\right|\leq C\left|\Sigma(x,t)-\Sigma(y,t)\right|.
\end{displaymath} Since $\Sigma$ is of finite variation for all
time, there exists a largest and smallest value of $S$, say
$S_{max}$ and $S_{min}$ with $0<S_{min}\leq S_{max}$. Define $C$ by
\begin{displaymath}
  C=max_{S\in\left[S_{min},S_{max}\right]}\left(\frac{d\Sigma}{d
  S}\right)^{-1}=max_{S\in\left[S_{min},S_{max}\right]}\frac{A(S)}{A'(S)}.
\end{displaymath}
    It follows that,
    \begin{displaymath}
      Var\left\{S(\cdot,t)\right\}\leq C\cdot Var\left\{\Sigma(\cdot,t)\right\}.
    \end{displaymath}

    Finally, from Proposition \ref{Change-Of-Variables} the
    determinant of the Jacobian is bounded away from zero for all
    approximate solutions.  Therefore, the variation in conserved variables, $(U_{1},U_{2},U_{3})$,
    are bounded for all $t \geq 0$, $\theta$ and $\Delta x_{i}$.

    Therefore, Theorem \ref{Glimm} provides existence of a set measure zero
    $\overline{A}\subset A$ such that if we choose $\theta\in
    A-\overline{A}$ there exists a subsequence of mesh refinements,
    $\Delta x_{i_{k}}\to 0$ such that $U_{\theta,\Delta_{x_{i_{k}}}}$ converges pointwise almost everywhere in
    $L^{1}_{loc}$ to a weak solution, $U(x,t)$ of
    \eqref{System-Conservation}.  Moreover, this solution satisfies
    \begin{displaymath}
      Var\left\{\ln(\rho(\cdot,t))\right\}<N,
    \end{displaymath}
    \begin{displaymath}
      Var\left\{\ln\left(\frac{1+v(\cdot,t)}{1-v(\cdot,t)}\right)\right\}<N,
    \end{displaymath}
    and
    \begin{displaymath}
      Var\left\{S(\cdot,t)\right\}<N,
    \end{displaymath}
    for some $N>0$, all $t>0$ and is $L^{1}$ Lipschitz in time.

    \end{proof}

    \chapter{Interaction
    Estimates}\label{Interaction-Estimates-Chapter}
    \thispagestyle{firstpage}
    We give a systematic approach to the wave interaction estimates
needed to prove Proposition \ref{Interaction-Estimates}.  From the
special geometry of the shock-rarefaction curves in the space of
Riemann invariants we can analyze the interactions as done for the
classical Euler equations in \cite{Nishida-Smoller} and
\cite{Temple-Large}. There are sixteen possible incoming wave
profiles, and among these one to four different outgoing wave
profiles, each of which will be covered on a case by case basis.  We
assume that all the interactions occur in a simply connected compact
set $\Omega\subset\mathbb{R}^{2}$. Recall that for a compact set
$\Omega$ we have defined the constant $C_{0}$,
\eqref{Definition-C_0}, as the max of $1/2$ or the largest slope
possible of a shock curve contained in $\Omega$.  Since the shock
wave slopes are strictly bounded above by $1$, we have, $0<1/2\leq
C_{0}<1$.

During these estimates we repeatedly use the fact that the shock
curves in the space of Riemann invariants are translationally
invariant, convex and whose derivatives are bounded above by a
constant.  We reference Lemma \ref{Lemma-Smoller-Temple} for these
results.  Since our definition of wave strength is determined by the
change in $r$ for $1-$waves and $s$ for $3-$waves, we use the
following two facts:
\begin{enumerate}

\item The change in $s$ along a $1-$shock is uniformly bounded by
the change in $r$ and vice versa for $3-$shocks.  Indeed, since we
have
\begin{displaymath}
  \frac{ds}{dr}\leq \frac{\sqrt{2K}-1}{-\sqrt{2K}-1}
\end{displaymath}
for $1-$shocks and
\begin{displaymath}
  \frac{dr}{ds}\leq \frac{\sqrt{2K}-1}{-\sqrt{2K}-1}
\end{displaymath}
for $3-$shocks, we have for our constant $C_{0}$,
\begin{displaymath}
  \frac{y_{1}}{z_{1}}<C_{0} \phantom{333} \textrm{and} \phantom{333}
  \frac{y_{3}}{z_{3}}<C_{0}.
\end{displaymath}
See Figure \ref{Shock-Change}.

\begin{figure}
\begin{center}
  \includegraphics[width=4in]{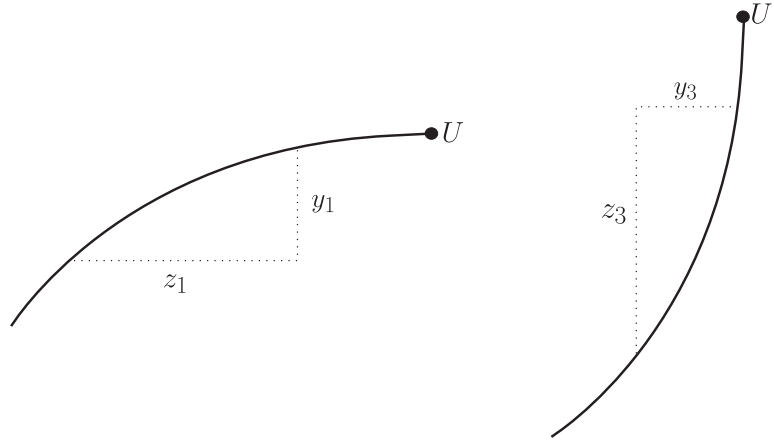}
  \caption{Shock Curve Slopes are bounded by
  $C_{0}$.}\label{Shock-Change}
\end{center}
\end{figure}

\item  Suppose two shock curves of the same family that begin at two
distinct states $U_{1}$ and $U_{2}$ and meet at a common third state
$U_{3}$.  Then the ratio of the distances along the $r$ and $s$ axes
from $U_{1}$ and $U_{2}$ are bounded by $C_{0}$.  Again, we have,
\begin{displaymath}
  \frac{y_{1}}{z_{1}}<C_{0} \phantom{333} \textrm{and} \phantom{333}
  \frac{y_{3}}{z_{3}}<C_{0}.
\end{displaymath}
See Figure \ref{Shock-Meet}.

\begin{figure}
\begin{center}
  \includegraphics[width=4in]{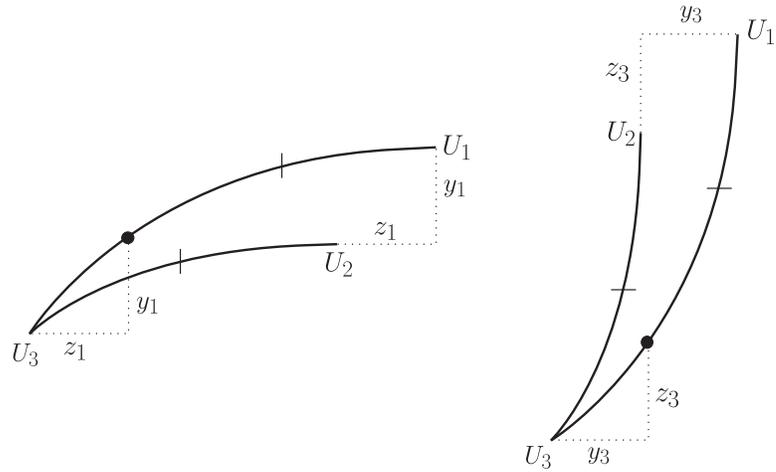}
  \caption{Shock curves intersecting at $U_{3}$ satisfy, $y/z<C_{0}$.}\label{Shock-Meet}
\end{center}
\end{figure}

\end{enumerate}


We now begin our interaction analysis.

\begin{enumerate}

\item
$(\alpha_{1},\beta_{1})+(\alpha_{2},\beta_{2})$
\begin{itemize}
  \item If $A\leq 0$ and $B\leq 0$ we are done.

\begin{figure}
\begin{center}
  \includegraphics[width=150pt]{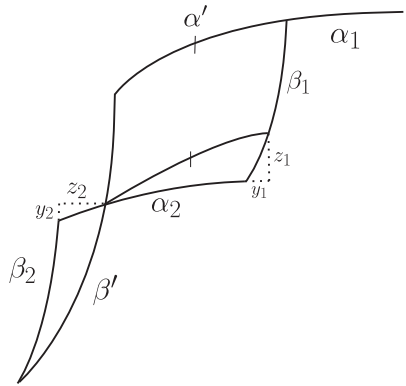}
  \caption{$(\alpha_{1},\beta_{1})+(\alpha_{2},\beta_{2})\rightarrow(\alpha',\beta')$, $A\leq 0$, $B\geq 0$}\label{SSSS-SS-1}
\end{center}

\end{figure}

\item
$(\alpha_{1},\beta_{1})+(\alpha_{2},\beta_{2})\rightarrow(\alpha',\beta')$: Figure \ref{SSSS-SS-1}\\
$\phantom{33333}$ Suppose that $A=-\xi\leq 0$ and $B\geq 0$.  We
have $A=y_{1}-z_{2}$ and $B=y_{2}-z_{1}$. From $A\leq 0$,
$y_{2}>z_{1}$, and hence,
\begin{displaymath}
  y_{1}<z_{1}<y_{2}<z_{2}.
\end{displaymath}
Therefore,
\begin{eqnarray}
  A+B & = & (y_{1}-z_{2})+(y_{2}-z_{1}),\nonumber\\
      & = & (y_{1}+y_{2})-(z_{1}+z_{2}),\nonumber\\
      & \leq & (C_{0}-1)(z_{1}+z_{2}),\nonumber\\
      & \leq & (C_{0}-1)(z_{2}-y_{1}),\nonumber\\
      & \leq & C_{0}\xi+A.\nonumber
\end{eqnarray}
Hence, $B\leq C_{0}\xi$.  Note, the inequalities hold since
$(C_{0}-1)<0$.


\begin{figure}
\begin{center}
  \includegraphics[width=150pt]{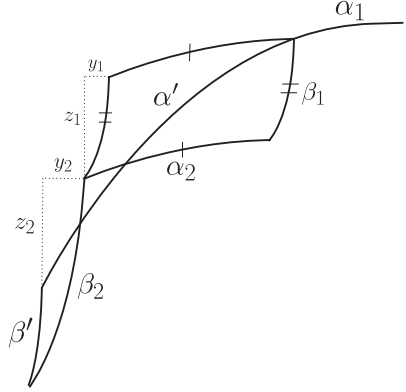}
  \caption{$(\alpha_{1},\beta_{1})+(\alpha_{2},\beta_{2})\rightarrow(\alpha',\beta')$, $A\geq 0$, $B\leq 0$}\label{SSSS-SS-2}
\end{center}
\end{figure}

\item $(\alpha_{1},\beta_{1})+(\alpha_{2},\beta_{2})\rightarrow(\alpha',\beta')$: Figure \ref{SSSS-SS-2}\\
$\phantom{33333}$ Suppose that $A\geq 0$ and $B=-\xi\leq 0$. We have
$\alpha'=\alpha_{1}+y_{1}+\alpha_{2}+y_{2}$ and hence,
$A=\alpha'-\alpha_{1}-\alpha_{2}=y_{1}+y_{2}$.  Furthermore,
$\beta'+z_{2}+z_{1}=\beta_{1}+\beta_{2}$, which gives,
$B=\beta'-\beta_{1}-\beta_{2}=-z_{1}-z_{2}=-\xi$.  Thus, $A\leq
C_{0}\xi$.


\begin{figure}
\begin{center}
  \includegraphics[height=150pt]{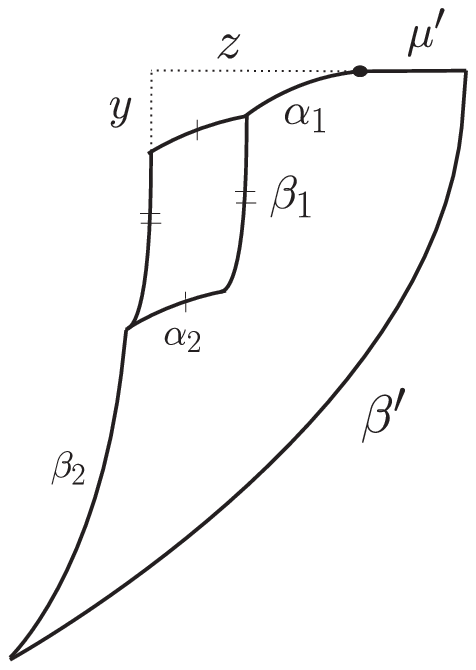}
  \caption{$(\alpha_{1},\beta_{1})+(\alpha_{2},\beta_{2})\rightarrow(\mu',\beta')$}\label{SSSS-RS}
\end{center}
\end{figure}

\item $(\alpha_{1},\beta_{1})+(\alpha_{2},\beta_{2})\rightarrow(\mu',\beta')$: Figure \ref{SSSS-RS}\\
$A=-z$, $B=y\leq C_{0}z$.


\begin{figure}
\begin{center}
  \includegraphics[height=150pt]{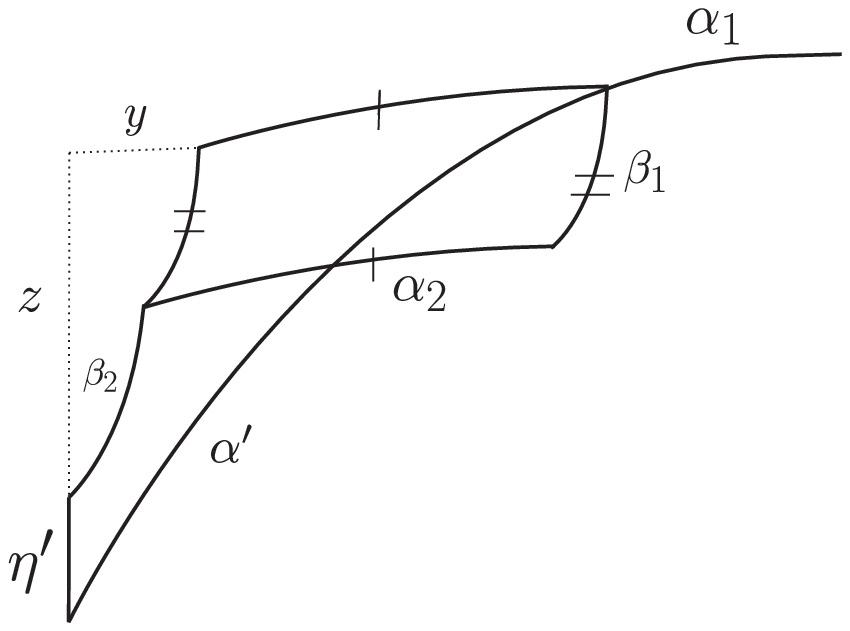}
  \caption{$(\alpha_{1},\beta_{1})+(\alpha_{2},\beta_{2})\rightarrow(\alpha',\eta')$}\label{SSSS-SR}
\end{center}
\end{figure}

\item $(\alpha_{1},\beta_{1})+(\alpha_{2},\beta_{2})\rightarrow(\alpha',\eta')$: Figure \ref{SSSS-SR}\\
$B=-\beta_{1}-\beta_{2})=-z$, and
$A=\alpha'-\alpha_{1}-\alpha_{2}=y\leq C_{0}z$.

\end{itemize}


\item $(\alpha_{1},\beta_{1})+(\alpha_{2},\eta_{2})$

\begin{figure}
\begin{center}
  \includegraphics[height=130pt]{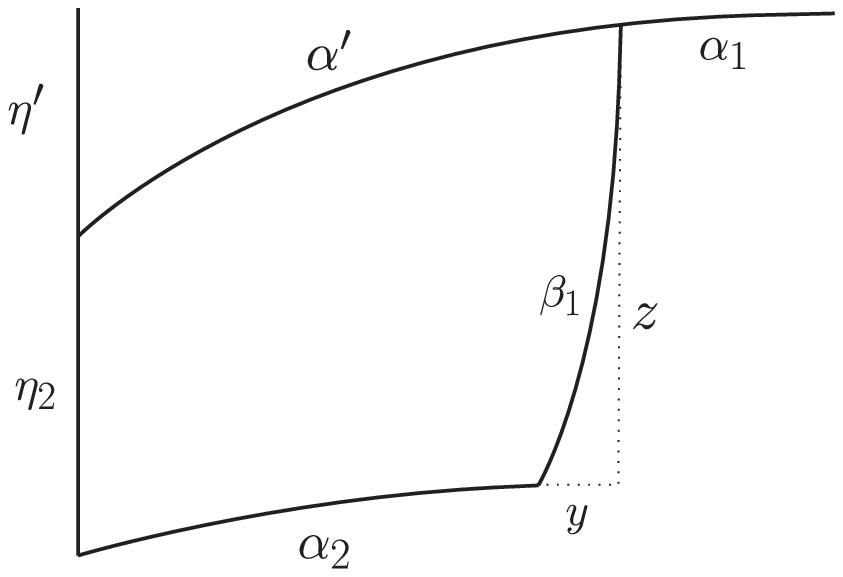}
  \caption{$(\alpha_{1},\beta_{1})+(\alpha_{2},\eta_{2})\rightarrow(\alpha',\eta')$}\label{SSSR-SR}
\end{center}
\end{figure}

\begin{itemize}
\item $(\alpha_{1},\beta_{1})+(\alpha_{2},\eta_{2})\rightarrow(\alpha',\eta')$: Figure \ref{SSSR-SR}\\
$\phantom{333333}$ $B=-z$, $A=\alpha'-\alpha_{1}-\alpha_{2}=y\leq
C_{0}z$.

\begin{figure}
\begin{center}
  \includegraphics[height=150pt]{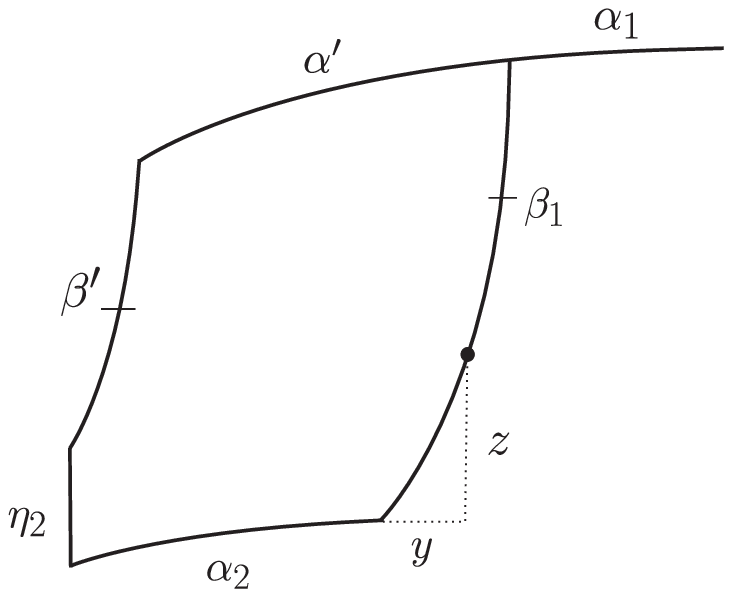}
  \caption{$(\alpha_{1},\beta_{1})+(\alpha_{2},\eta_{2})\rightarrow(\alpha',\beta')$}\label{SSSR-SS}
\end{center}
\end{figure}

\item $(\alpha_{1},\beta_{1})+(\alpha_{2},\eta_{2})\rightarrow(\alpha',\beta')$: Figure \ref{SSSR-SS}\\
$\phantom{333333}$ $B=\beta'-\beta_{1}-\beta_{2}=-z$.
 Since, $\alpha_{1}+\alpha'=\alpha_{2}+y+\alpha_{1}\Rightarrow
\alpha'-\alpha_{2}=y$  then, $A=\alpha'-\alpha_{1}-\alpha_{2}\leq y
\leq C_{0}z$.

\end{itemize}


\begin{figure}
\begin{center}
  \includegraphics[height=150pt]{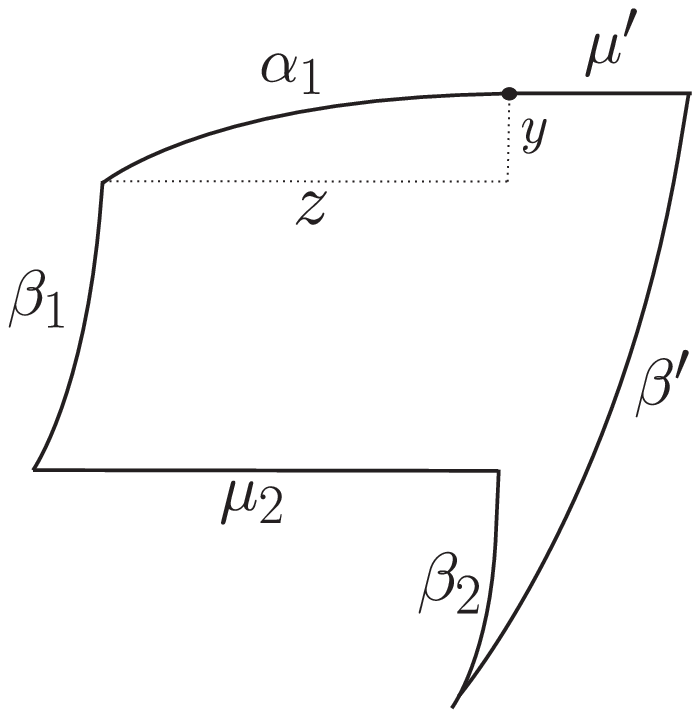}
  \caption{$(\alpha_{1},\beta_{1})+(\mu_{2},\beta_{2})\rightarrow(\mu',\beta')$}\label{SSRS-RS}
\end{center}
\end{figure}

\item $(\alpha_{1},\beta_{1})+(\mu_{2},\beta_{2})$
\begin{itemize}
\item $(\alpha_{1},\beta_{1})+(\mu_{2},\beta_{2})\rightarrow(\mu',\beta')$: Figure \ref{SSRS-RS}\\
$\phantom{333333}$ $A=-\alpha_{1}=-z$,
$B=\beta'-\beta_{1}-\beta_{2}=y\leq C_{0}z$.

\begin{figure}
\begin{center}
  \includegraphics[height=150pt]{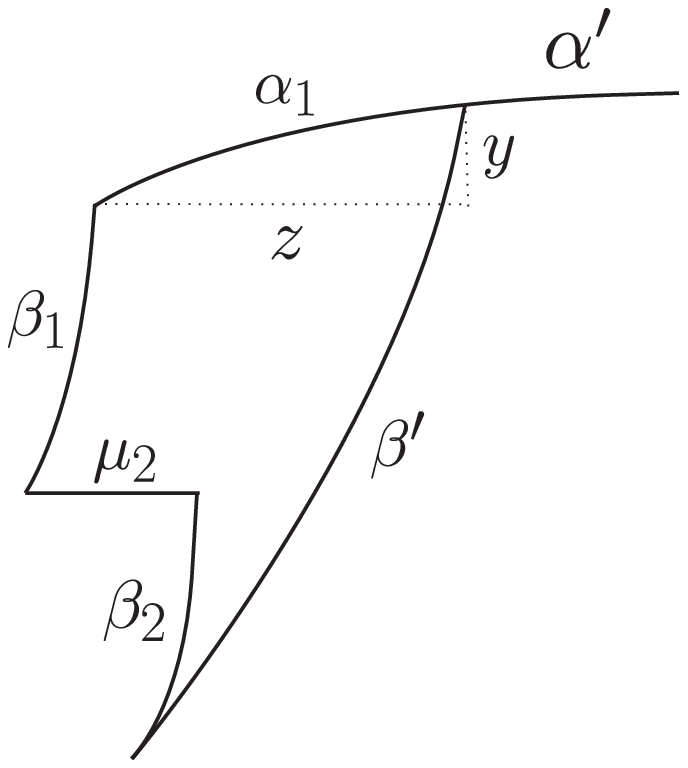}
  \caption{$(\alpha_{1},\beta_{1})+(\mu_{2},\beta_{2})\rightarrow(\alpha',\beta')$}\label{SSRS-SS}
\end{center}
\end{figure}

\item $(\alpha_{1},\beta_{1})+(\mu_{2},\beta_{2})\rightarrow(\alpha',\beta')$: Figure \ref{SSRS-SS}\\
$\phantom{333333}$ $A=\alpha'-\alpha_{1}=-z$,
$B=\beta'-\beta_{1}-\beta_{2}=y\leq C_{0}z$.

\end{itemize}


\item $(\alpha_{1},\beta_{1})+(\mu_{2},\eta_{2})$

\begin{itemize}

\item $(\alpha_{1},\beta_{1})+(\mu_{2},\eta_{2})\rightarrow(\mu',\eta')$\\
$\phantom{333333}$ $A\leq0$, $B\leq 0$.

\begin{figure}
\begin{center}
  \includegraphics[height=150pt]{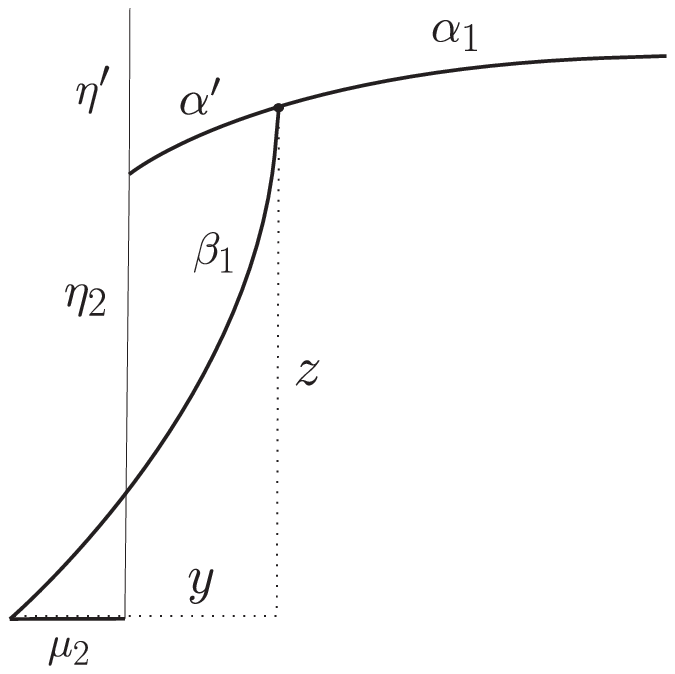}
  \caption{$(\alpha_{1},\beta_{1})+(\mu_{2},\eta_{2})\rightarrow(\alpha',\eta')$}\label{SSRR-SR}
\end{center}
\end{figure}

\item $(\alpha_{1},\beta_{1})+(\mu_{2},\eta_{2})\rightarrow(\alpha',\eta')$: Figure \ref{SSRR-SR}\\
$\phantom{333333}$ $A,B\leq 0$ or, $B=-\beta_{1}=-z$ and
$A=\alpha'-\alpha_{1}\leq y\leq C_{0}z$.

\begin{figure}
\begin{center}
  \includegraphics[width=150pt]{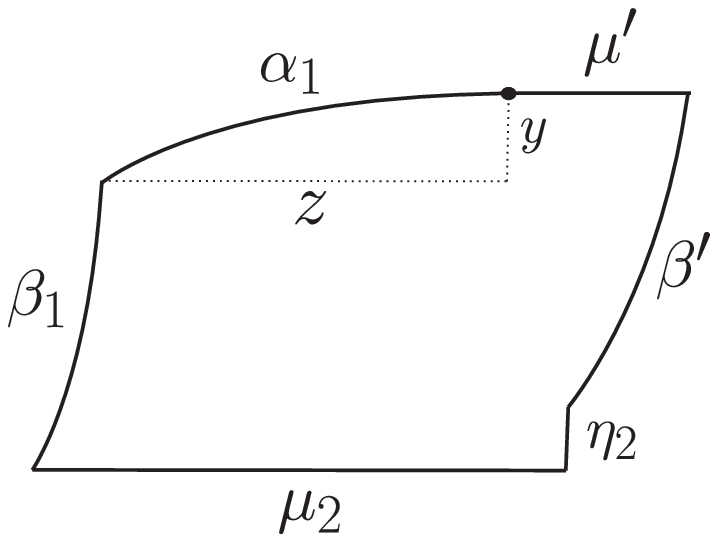}
  \caption{$(\alpha_{1},\beta_{1})+(\mu_{2},\eta_{2})\rightarrow(\mu',\beta')$}\label{SSRR-RS}
\end{center}
\end{figure}

\item $(\alpha_{1},\beta_{1})+(\mu_{2},\eta_{2})\rightarrow(\mu',\beta')$: Figure \ref{SSRR-RS}\\
$\phantom{333333}$ $A,B\leq 0$ or; $A=-z$ and
$y+\beta_{1}=\beta'+\eta_{2}$ $\Rightarrow$
$\beta'-\beta_{1}=y-\eta_{2}$ $\Rightarrow$ $B=\beta'-\beta_{1}\leq
y\leq C_{0}z$.

\begin{figure}
\begin{center}
  \includegraphics[height=150pt]{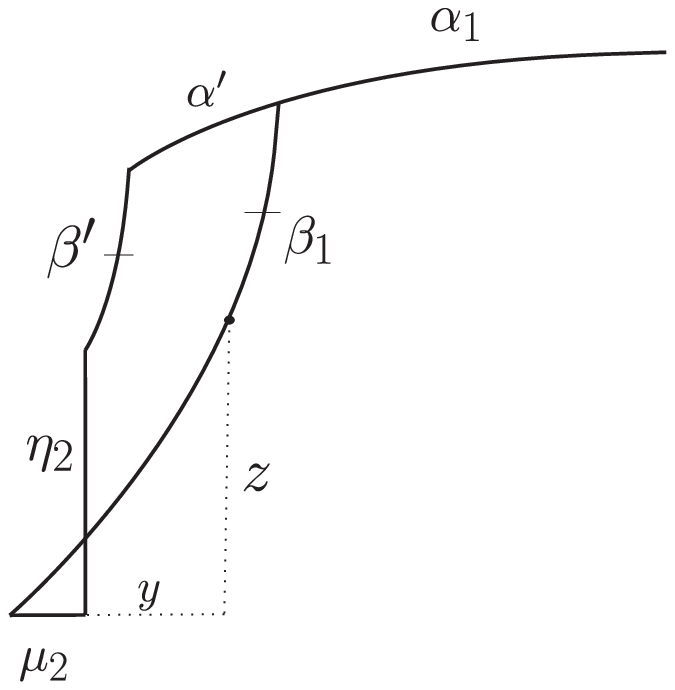}
  \includegraphics[height=150pt]{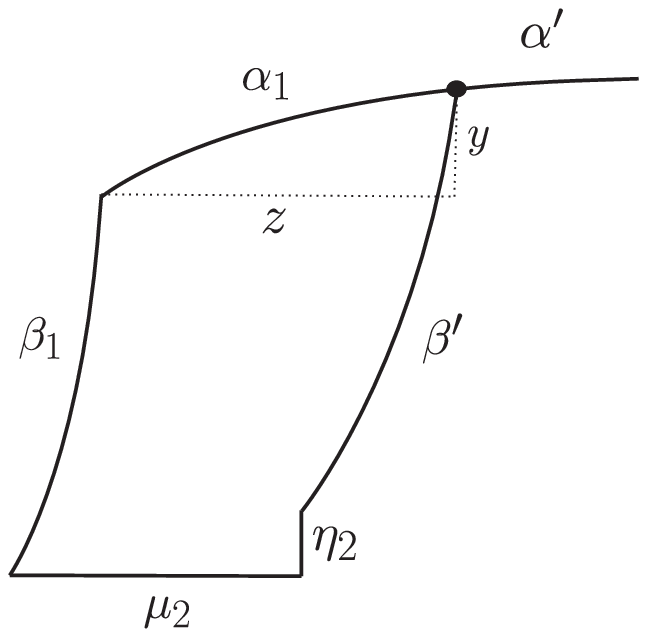}
  \caption{$(\alpha_{1},\beta_{1})+(\mu_{2},\eta_{2})\rightarrow(\alpha',\beta')$, $0\leq A$ and $0\leq B$.}\label{SSRR-SS}
\end{center}
\end{figure}

\item $(\alpha_{1},\beta_{1})+(\mu_{2},\eta_{2})\rightarrow(\alpha',\beta')$: Figure \ref{SSRR-SS}\\
$\phantom{333333}$ $B=\beta'-\beta_{1}=-z$ and
$\alpha_{1}+y=\mu_{2}+\alpha'$ $\Rightarrow$
$\alpha'-\alpha_{1}=y-\mu_{2}$ $\Rightarrow$
$A=\alpha'-\alpha_{1}\leq y \leq C_{0}z$.


\item $(\alpha_{1},\beta_{1})+(\mu_{2},\eta_{2})\rightarrow(\alpha',\beta')$: Figure \ref{SSRR-SS}\\
$\phantom{333333}$ $A=\alpha'-\alpha_{1}=-z$ and
$y+\beta_{1}=\beta'+\eta_{2}$ $\Rightarrow$
$\beta'-\beta_{1}=y-\eta_{2}$ $\Rightarrow$ $B=\beta'-\beta_{1}\leq
y \leq C_{0}z$.

\end{itemize}


\item $(\alpha_{1},\eta_{1})+(\alpha_{2},\beta_{2})$

\begin{figure}
\begin{center}
  \includegraphics[height=150pt]{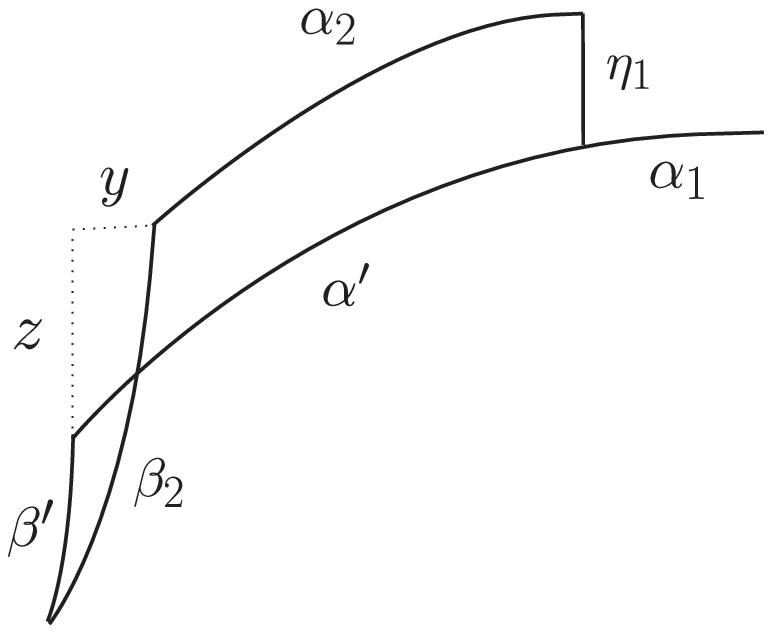}
  \caption{$(\alpha_{1},\eta_{1})+(\alpha_{2},\beta_{2})\rightarrow(\alpha',\beta')$}\label{SRSS-SS}
\end{center}
\end{figure}

\begin{itemize}

\item $(\alpha_{1},\eta_{1})+(\alpha_{2},\beta_{2})\rightarrow(\alpha',\beta')$: Figure \ref{SRSS-SS}\\
$\phantom{333333}$ $A=\alpha'-\alpha_{1}-\alpha_{2}=y$ and
$B=\beta'-\beta_{2}=-z$ $\Rightarrow$ $A\leq C_{0}z$.

\begin{figure}
\begin{center}
  \includegraphics[height=150pt]{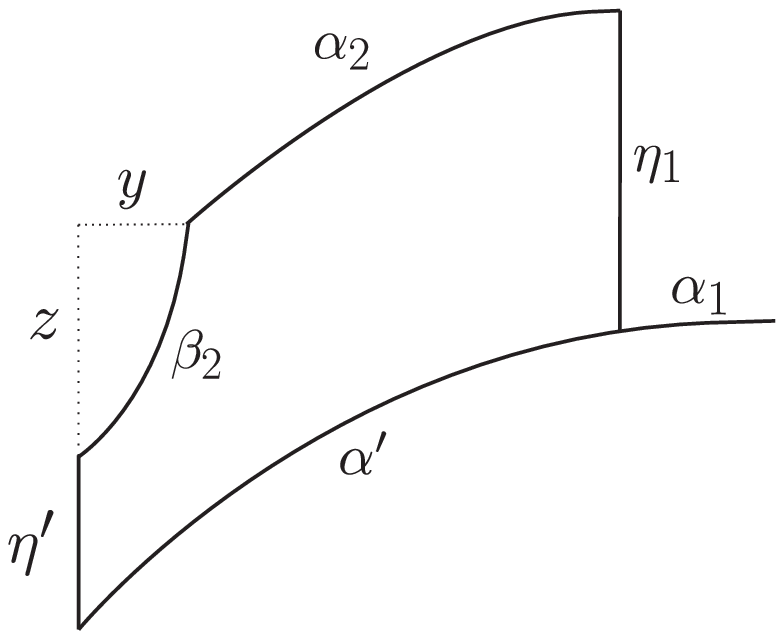}
  \caption{$(\alpha_{1},\eta_{1})+(\alpha_{2},\beta_{2})\rightarrow(\alpha',\eta')$}\label{SRSS-SR}
\end{center}
\end{figure}

\item $(\alpha_{1},\eta_{1})+(\alpha_{2},\beta_{2})\rightarrow(\alpha',\eta')$: Figure \ref{SRSS-SR}\\
$\phantom{333333}$ $A=\alpha'-\alpha_{1}-\alpha_{2}=y$ and
$B=-\beta_{2}=-z$ $\Rightarrow$ $A\leq C_{0}z$.

\end{itemize}


\begin{figure}
\begin{center}
  \includegraphics[height=130pt]{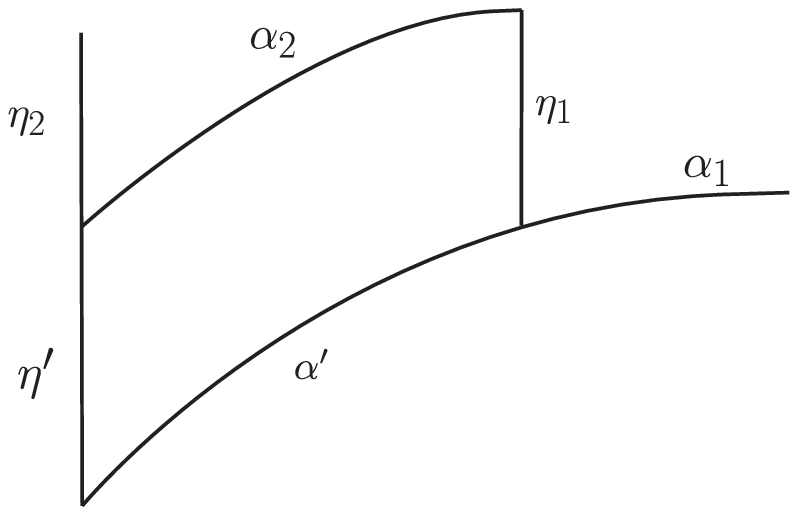}
  \caption{$(\alpha_{1},\eta_{1})+(\alpha_{2},\eta_{2})\rightarrow(\alpha',\eta')$}\label{SRSR-SR}
\end{center}
\end{figure}

\item $(\alpha_{1},\eta_{1})+(\alpha_{2},\eta_{2})\rightarrow(\alpha',\eta')$: Figure \ref{SRSR-SR}\\
$\phantom{333333}$ $A=\alpha'-\alpha_{1}-\alpha_{2}=0\leq 0$ and
$B=0\leq 0$


\item $(\alpha_{1},\eta_{1})+(\mu_{2},\beta_{2})$

\begin{figure}
\begin{center}
  \includegraphics[height=150pt]{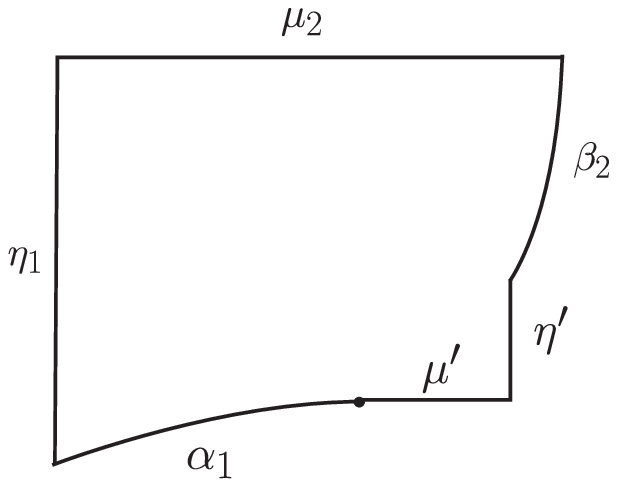}
  \caption{$(\alpha_{1},\eta_{1})+(\mu_{2},\beta_{2})\rightarrow(\mu',\eta')$}\label{SRRS-RR}
\end{center}
\end{figure}

\begin{itemize}

\item $(\alpha_{1},\eta_{1})+(\mu_{2},\beta_{2})\rightarrow(\mu',\eta')$: Figure \ref{SRRS-RR}\\
$\phantom{333333}$ $A=-\alpha_{1}\leq 0$ and $B=-\beta_{2}\leq 0.$

\begin{figure}
\begin{center}
  \includegraphics[height=150pt]{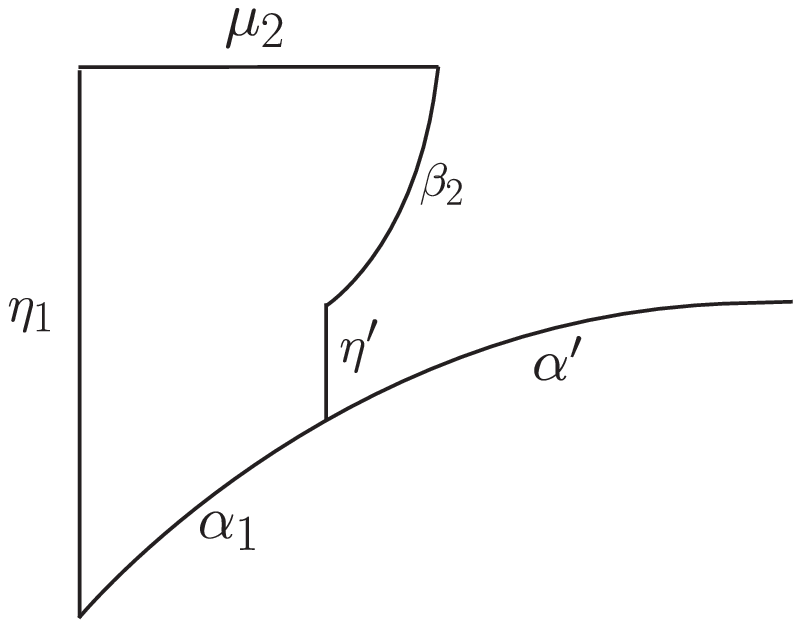}
  \includegraphics[height=150pt]{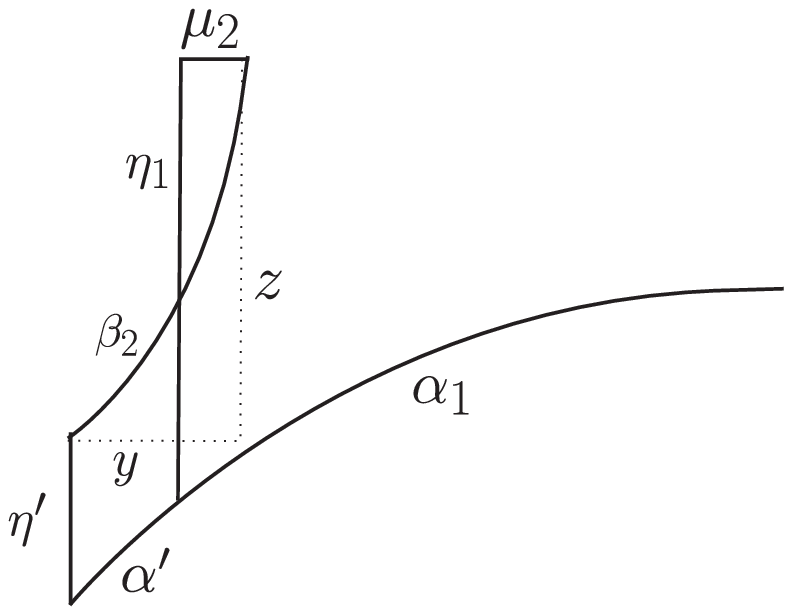}
  \caption{$(\alpha_{1},\eta_{1})+(\mu_{2},\beta_{2})\rightarrow(\alpha',\eta')$, $A,B\leq 0$ and $0\leq A$.}\label{SRRS-SR}
\end{center}
\end{figure}

\item $(\alpha_{1},\eta_{1})+(\mu_{2},\beta_{2})\rightarrow(\alpha',\eta')$: Figure \ref{SRRS-SR}\\
$\phantom{333333}$ Either $A\leq 0$ and $B\leq 0$ or $A\leq y$ and
$B=-\beta_{2}=z$, therefore, $A\leq -C_{0}B$.

\begin{figure}
\begin{center}
  \includegraphics[height=150pt]{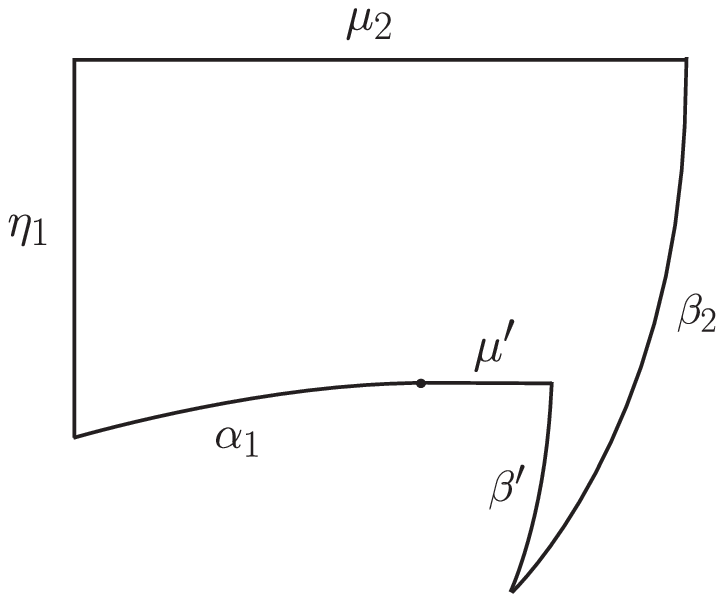}
  \includegraphics[height=100pt]{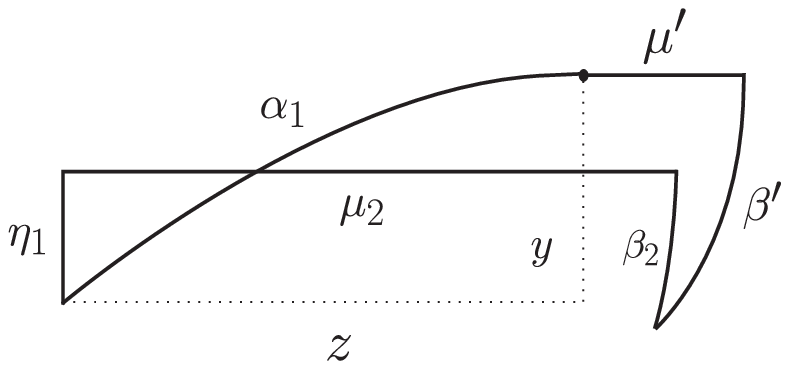}
  \caption{$(\alpha_{1},\eta_{1})+(\mu_{2},\beta_{2})\rightarrow(\mu',\beta')$, $A,B\leq 0$, and $0\leq B$}\label{SRRS-RS}
\end{center}
\end{figure}

\item $(\alpha_{1},\eta_{1})+(\mu_{2},\beta_{2})\rightarrow(\mu',\beta')$: Figure \ref{SRRS-RS}\\
$\phantom{333333}$ Either $A\leq 0$ and $B\leq 0$ or
$B=\beta'-\beta_{2}\leq y$ and $A=-\alpha_{1}=-z$, therefore, $B\leq
C_{0}z$.

\begin{figure}
\begin{center}
  \includegraphics[width=170pt]{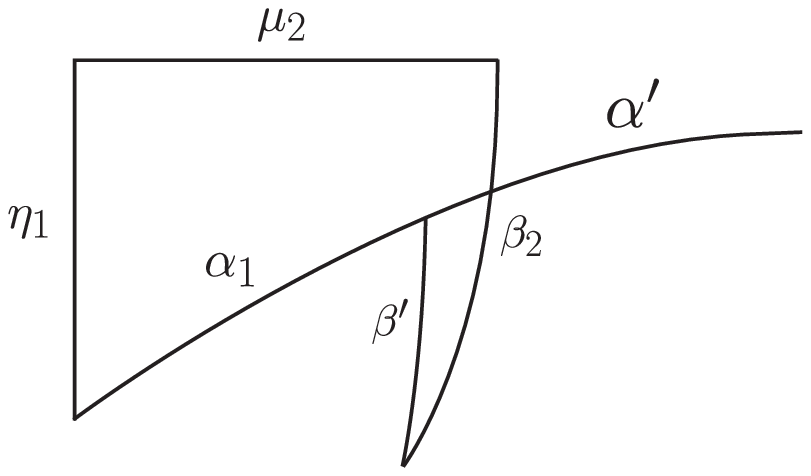}
  \caption{$(\alpha_{1},\eta_{1})+(\mu_{2},\beta_{2})\rightarrow(\alpha',\beta')$, Case $1$.}\label{SRRS-SS2}
\end{center}
\end{figure}

\begin{figure}
\begin{center}
  \includegraphics[width=170pt]{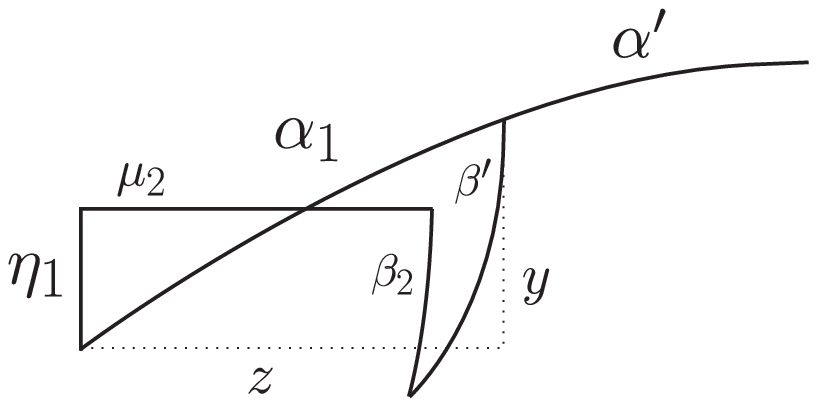}
  \includegraphics[width=170pt]{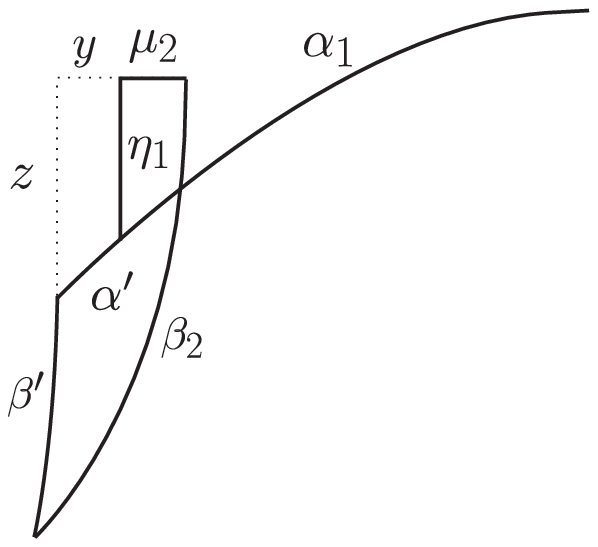}
  \caption{$(\alpha_{1},\eta_{1})+(\mu_{2},\beta_{2})\rightarrow(\alpha',\beta')$, Case $2$ and $3$.}\label{SRRS-SS1}
\end{center}
\end{figure}

\item $(\alpha_{1},\eta_{1})+(\mu_{2},\beta_{2})\rightarrow(\alpha',\beta')$: Figure \ref{SRRS-SS1} and \ref{SRRS-SS2} \\
$\phantom{333333}$\\
Case $1.)$ $A\leq 0$ and $B\leq 0$.\\
Case $2.)$ $A=\alpha'-\alpha_{1}=-z$ and
$B=\beta'-\beta_{1}\leq y\leq C_{0}z$.\\
Case $3.)$ $B=\beta'-\beta_{2}=-z$ and $A=\alpha'-\alpha_{1}=y\leq
C_{0}z$.

\end{itemize}

\newpage


\item $(\alpha_{1},\eta_{1})+(\mu_{2},\eta_{2})$

\begin{itemize}
  \item
  $(\alpha_{1},\eta_{1})+(\mu_{2},\eta_{2})\rightarrow(\mu',\eta')$\\
  $\phantom{333333}$ $A=-\alpha_{1}\leq 0$ and $B=0\leq 0.$

  \item
  $(\alpha_{1},\eta_{1})+(\mu_{2},\eta_{2})\rightarrow(\alpha',\eta')$\\
  $\phantom{333333}$ $A=-\mu_{2}\leq 0$ and $B=0\leq 0.$

\begin{figure}
\begin{center}
  \includegraphics[width=170pt]{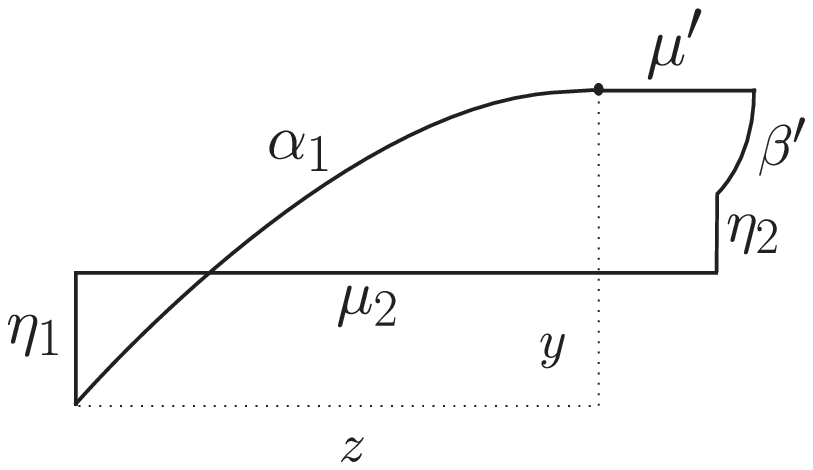}
  \caption{$(\alpha_{1},\eta_{1})+(\mu_{2},\eta_{2})\rightarrow(\mu',\beta')$}\label{SRRR-RS}
\end{center}
\end{figure}

  \item $(\alpha_{1},\eta_{1})+(\mu_{2},\eta_{2})\rightarrow(\mu',\beta')$: Figure \ref{SRRR-RS}\\
  $\phantom{333333}$ $A=-\alpha_{1}=-z$ and $B=\beta'\leq y \leq C_{0}z$.

\begin{figure}
\begin{center}
  \includegraphics[height=150pt]{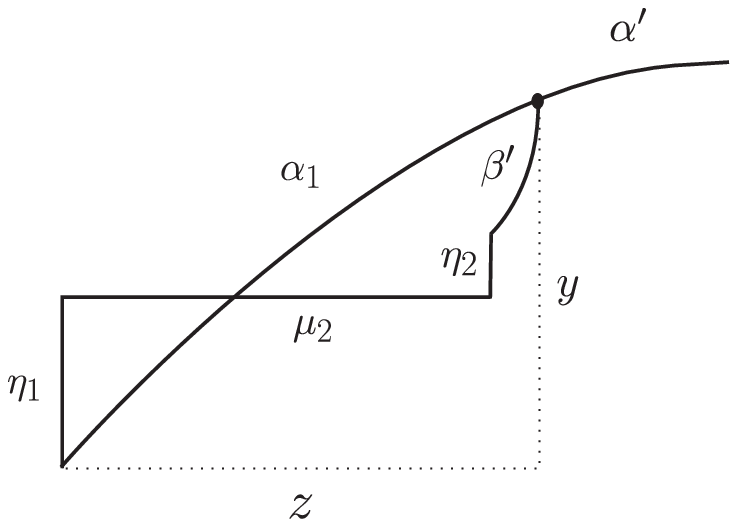}
  \caption{$(\alpha_{1},\eta_{1})+(\mu_{2},\eta_{2})\rightarrow(\alpha',\beta')$}\label{SRRR-SS}
\end{center}
\end{figure}

  \item $(\alpha_{1},\eta_{1})+(\mu_{2},\eta_{2})\rightarrow(\alpha',\beta')$: Figure \ref{SRRR-SS}\\
  $\phantom{333333}$ $A=\alpha'-\alpha_{1}=-z$ and $B=\beta'\leq y \leq C_{0}z$.

\end{itemize}


\item $(\mu_{1},\beta_{1})+(\alpha_{2},\beta_{2})$

\begin{figure}
\begin{center}
  \includegraphics[height=150pt]{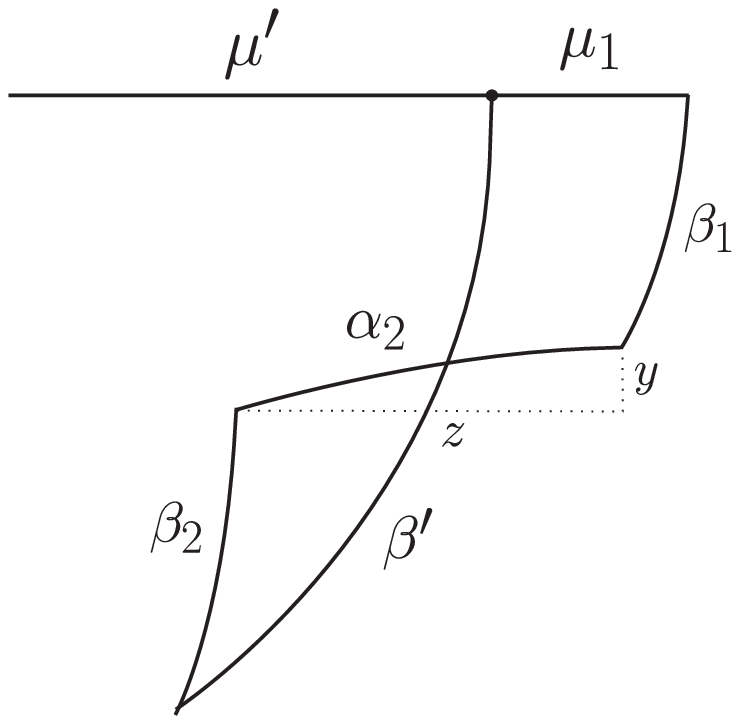}
  \caption{$(\mu_{1},\beta_{1})+(\alpha_{2},\beta_{2})\rightarrow(\mu',\beta')$}\label{RSSS-RS}
\end{center}
\end{figure}

\begin{itemize}

\item $(\mu_{1},\beta_{1})+(\alpha_{2},\beta_{2})\rightarrow(\mu',\beta')$: Figure \ref{RSSS-RS}\\
  $\phantom{333333}$ $A=-\alpha_{2}=-z$ and $B=y \leq C_{0}z$.

\begin{figure}
\begin{center}
  \includegraphics[height=150pt]{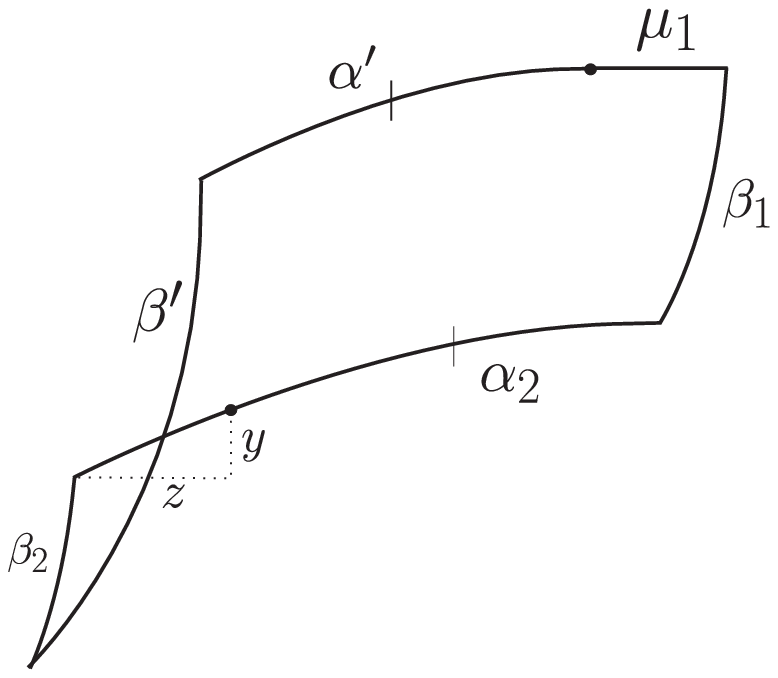}
  \caption{$(\mu_{1},\beta_{1})+(\alpha_{2},\beta_{2})\rightarrow(\alpha',\beta')$}\label{RSSS-SS}
\end{center}
\end{figure}

\item $(\mu_{1},\beta_{1})+(\alpha_{2},\beta_{2})\rightarrow(\alpha',\beta')$: Figure \ref{RSSS-SS}\\
  $\phantom{333333}$ $A=\alpha'-\alpha_{2}=-z$ and $B=\beta'-\beta_{1}-\beta_{2}=y \leq C_{0}z$.

\end{itemize}


\begin{figure}
\begin{center}
  \includegraphics[height=150pt]{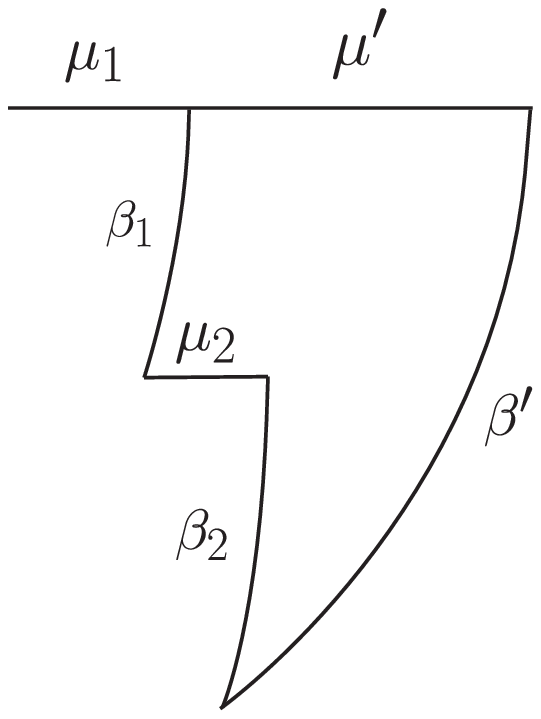}
  \caption{$(\mu_{1},\beta_{1})+(\mu_{2},\beta_{2})\rightarrow(\mu',\beta')$}\label{RSRS-RS}
\end{center}
\end{figure}

\item $(\mu_{1},\beta_{1})+(\mu_{2},\beta_{2})\rightarrow(\mu',\beta')$: Figure \ref{RSRS-RS}\\
$\phantom{333333}$ $A=0\leq 0$ and
$B=\beta'-\beta_{1}-\beta_{2}=0\leq 0$.


\item $(\mu_{1},\beta_{1})+(\alpha_{2},\eta_{2})$

\begin{itemize}

\item $(\mu_{1},\beta_{1})+(\alpha_{2},\eta_{2})\rightarrow(\mu',\eta')$\\
  $\phantom{333333}$ $A=-\alpha_{2}\leq 0$ and $B=-\beta_{1} \leq 0$.

\begin{figure}
\begin{center}
  \includegraphics[height=150pt]{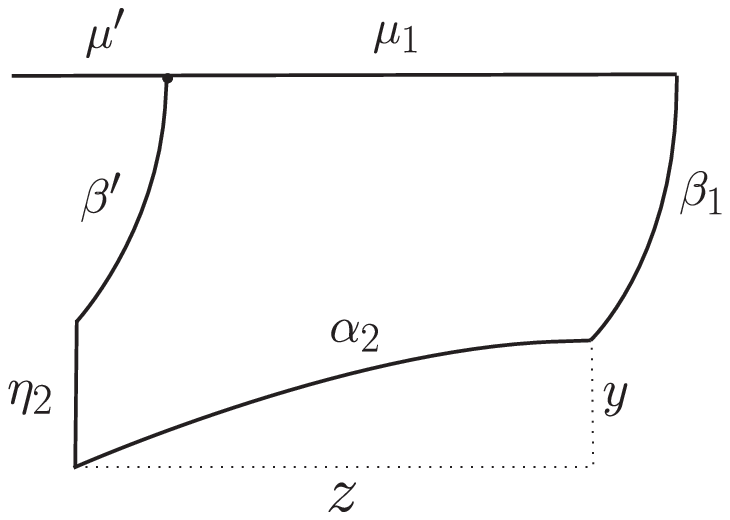}
  \caption{$(\mu_{1},\beta_{1})+(\alpha_{2},\eta_{2})\rightarrow(\mu',\beta')$}\label{RSSR-RS}
\end{center}
\end{figure}

\item $(\mu_{1},\beta_{1})+(\alpha_{2},\eta_{2})\rightarrow(\mu',\beta')$: Figure \ref{RSSR-RS}\\
  $\phantom{333333}$ $A,B\leq 0$ or, $A=-\alpha_{2}=-z$ and $B=\beta'-\beta_{1}\leq y \leq C_{0}z$.

\begin{figure}
\begin{center}
  \includegraphics[height=150pt]{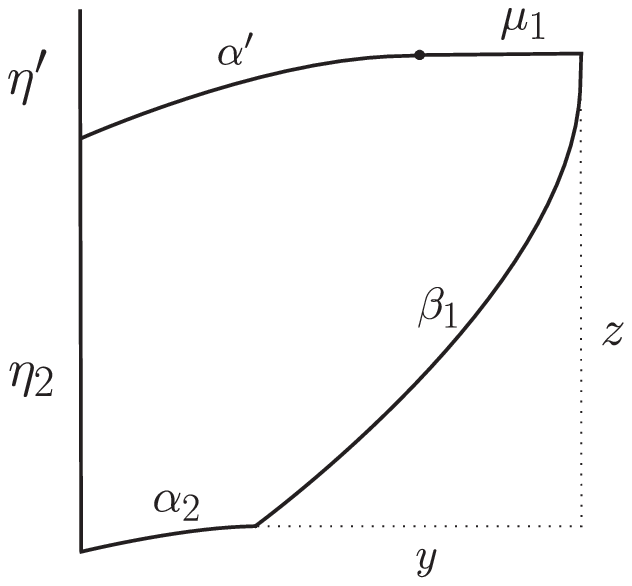}
  \caption{$(\mu_{1},\beta_{1})+(\alpha_{2},\eta_{2})\rightarrow(\alpha',\eta')$}\label{RSSR-SR}
\end{center}
\end{figure}

\item $(\mu_{1},\beta_{1})+(\alpha_{2},\eta_{2})\rightarrow(\alpha',\eta')$: Figure \ref{RSSR-SR}\\
  $\phantom{333333}$ $A,B\leq 0$ or, we have $y+\alpha_{2}=\mu_{1}+\alpha'$. Therefore, with $B=-\beta_{1}=-z$, $A=\alpha'-\alpha_{2}=y-\mu_{1}\leq y \leq
  C_{0}z$.

\begin{figure}
\begin{center}
  \includegraphics[width=170pt]{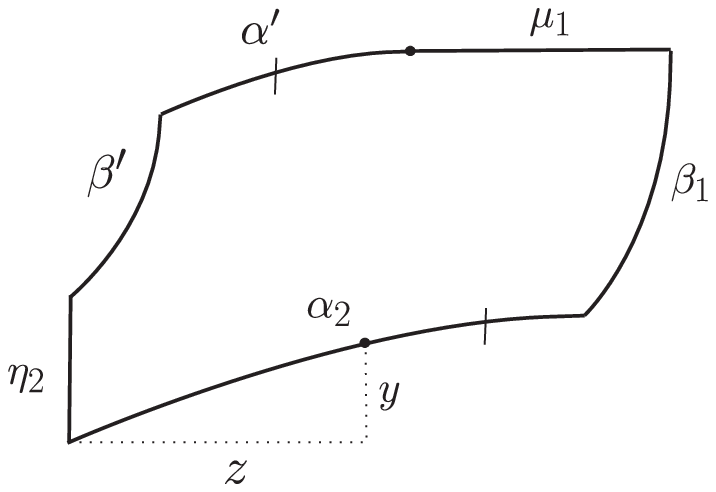}
  \includegraphics[width=140pt]{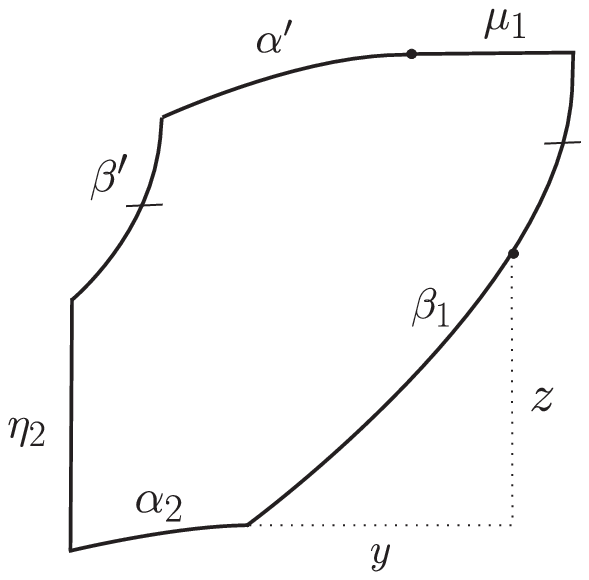}
  \caption{$(\mu_{1},\beta_{1})+(\alpha_{2},\eta_{2})\rightarrow(\alpha',\beta')$ Case $2$ and $3$.}\label{RSSR-SS}
\end{center}
\end{figure}

\item $(\mu_{1},\beta_{1})+(\alpha_{2},\eta_{2})\rightarrow(\alpha',\beta')$: Figure \ref{RSSR-SS}\\
  $\phantom{333333}$ Case $1.)$ $A\leq 0$ and $B\leq 0$.

  $\phantom{333333}$ Case $2.)$ $A=\alpha'-\alpha_{2}=-z$.  Then we know $\beta'+\eta_{2}=\beta'+y$ and therefore, $B=\beta'-\beta_{1}\leq y \leq C_{0}z$.

  $\phantom{333333}$ Case $3.)$ $B=\beta'-\beta_{1}=-z$.  Then we know $\alpha'+\mu_{1}=\alpha_{2}'+y$ and therefore, $A=\alpha'-\alpha_{2}\leq y \leq C_{0}z$.

\end{itemize}


\item $(\mu_{1},\beta_{1})+(\mu_{2},\eta_{2})$

\begin{itemize}

\item $(\mu_{1},\beta_{1})+(\mu_{2},\eta_{2})\rightarrow(\mu',\eta')$\\
  $\phantom{333333}$ $A=0\leq 0$ and $B=-\beta_{1} \leq 0$.

\begin{figure}
\begin{center}
  \includegraphics[height=130pt]{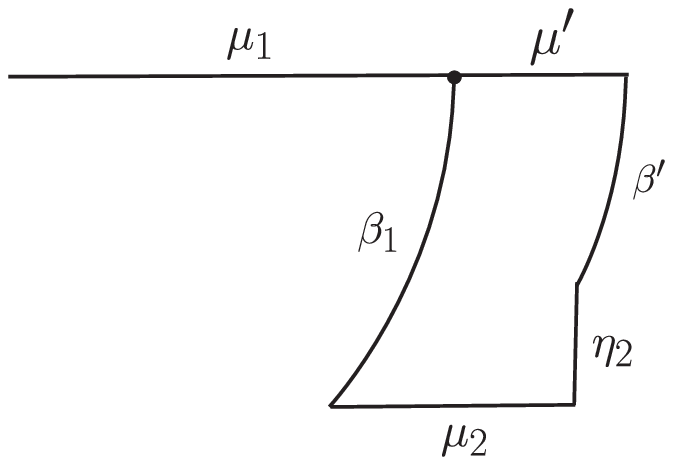}
  \caption{$(\mu_{1},\beta_{1})+(\mu_{2},\eta_{2})\rightarrow(\mu',\beta')$}\label{RSRR-RS}
\end{center}
\end{figure}

 \item $(\mu_{1},\beta_{1})+(\mu_{2},\eta_{2})\rightarrow(\mu',\beta')$: Figure \ref{RSRR-RS}\\
  $\phantom{333333}$ $A=0$ and $B=\beta'-\beta_{2}\leq 0$.

\begin{figure}
\begin{center}
  \includegraphics[height=150pt]{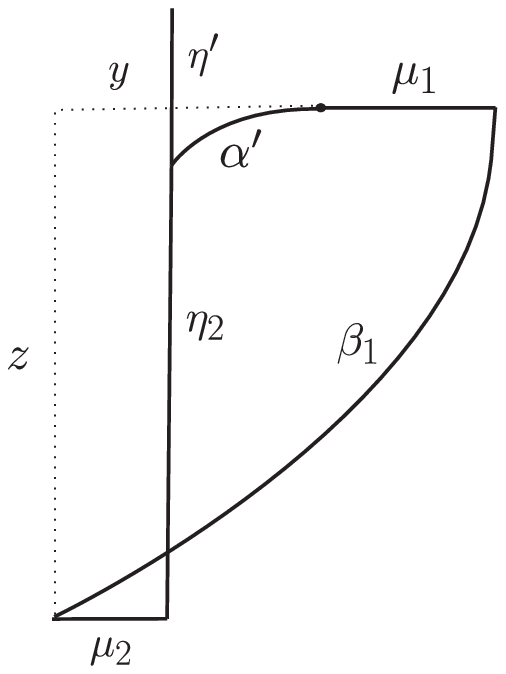}
  \caption{$(\mu_{1},\beta_{1})+(\mu_{2},\eta_{2})\rightarrow(\alpha',\eta')$}\label{RSRR-SR}
\end{center}
\end{figure}

 \item $(\mu_{1},\beta_{1})+(\mu_{2},\eta_{2})\rightarrow(\alpha',\eta')$: Figure \ref{RSRR-SR}\\
  $\phantom{333333}$ $B=-\beta_{1}=-z$ and $A=\alpha'\leq y\leq C_{0}z$.

\begin{figure}
\begin{center}
  \includegraphics[height=150pt]{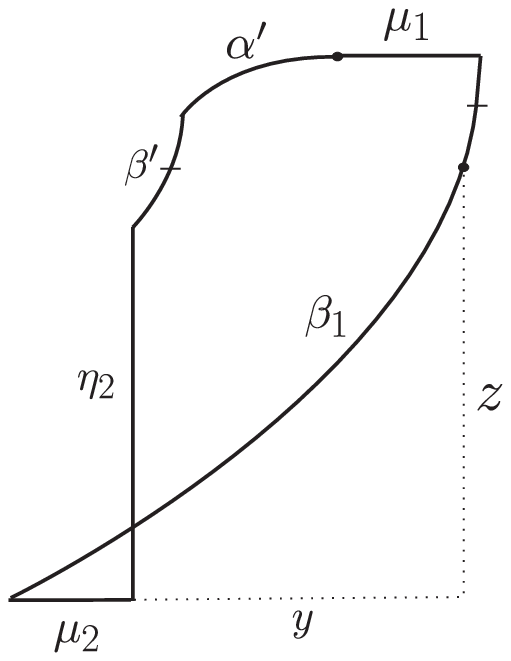}
  \caption{$(\mu_{1},\beta_{1})+(\mu_{2},\eta_{2})\rightarrow(\alpha',\beta')$}\label{RSRR-SS}
\end{center}
\end{figure}

 \item $(\mu_{1},\beta_{1})+(\mu_{2},\eta_{2})\rightarrow(\alpha',\beta')$: Figure \ref{RSRR-SS}\\
  $\phantom{333333}$ We have $B=\beta'-\beta_{1}=-z$ and $y=\alpha'+\mu_{1}+\mu_{2}$. Thus, $A=\alpha' \leq y \leq C_{0}z$.

\end{itemize}

\vspace{1in}


\item $(\mu_{1},\eta_{1})+(\alpha_{2},\beta_{2})$

\begin{figure}
\begin{center}
  \includegraphics[height=150pt]{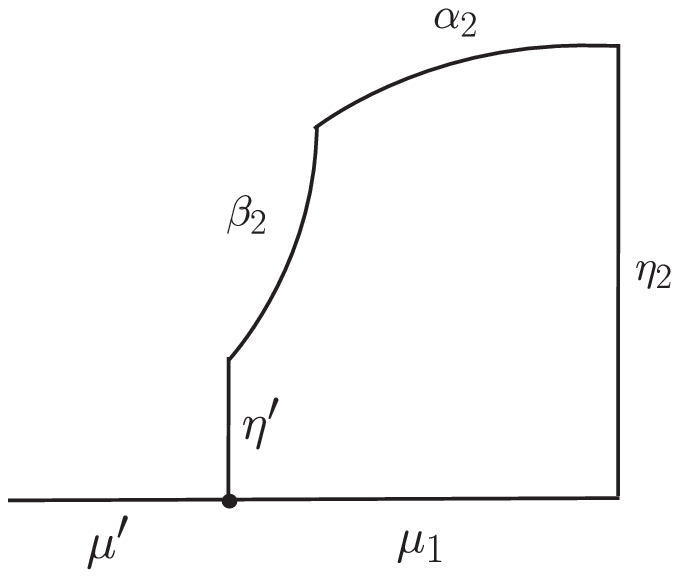}
  \caption{$(\mu_{1},\eta_{1})+(\alpha_{2},\beta_{2})\rightarrow(\mu',\eta')$}\label{RRSS-RR}
\end{center}
\end{figure}

\begin{itemize}

\item $(\mu_{1},\eta_{1})+(\alpha_{2},\beta_{2})\rightarrow(\mu',\eta')$: Figure \ref{RRSS-RR}\\
  $\phantom{333333}$ We have $A=-\alpha_{2}\leq 0$ and $B=-\beta_{2}\leq 0$.

\begin{figure}
\begin{center}
  \includegraphics[height=150pt]{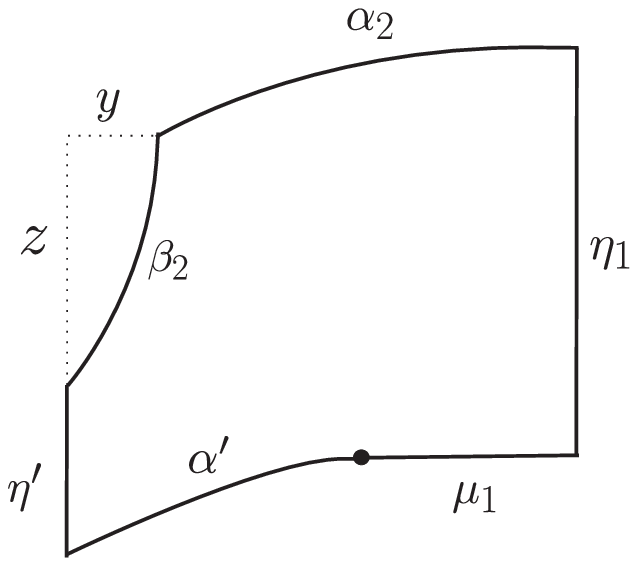}
  \caption{$(\mu_{1},\eta_{1})+(\alpha_{2},\beta_{2})\rightarrow(\alpha',\eta')$}\label{RRSS-SR}
\end{center}
\end{figure}

\item $(\mu_{1},\eta_{1})+(\alpha_{2},\beta_{2})\rightarrow(\alpha',\eta')$: Figure \ref{RRSS-SR}\\
  $\phantom{333333}$ $B=-\beta_{2}=-z\leq 0$ and $A=\alpha'-\alpha_{2}=y-\mu_{1}\leq y\leq C_{0}z$.

\begin{figure}
\begin{center}
  \includegraphics[height=150pt]{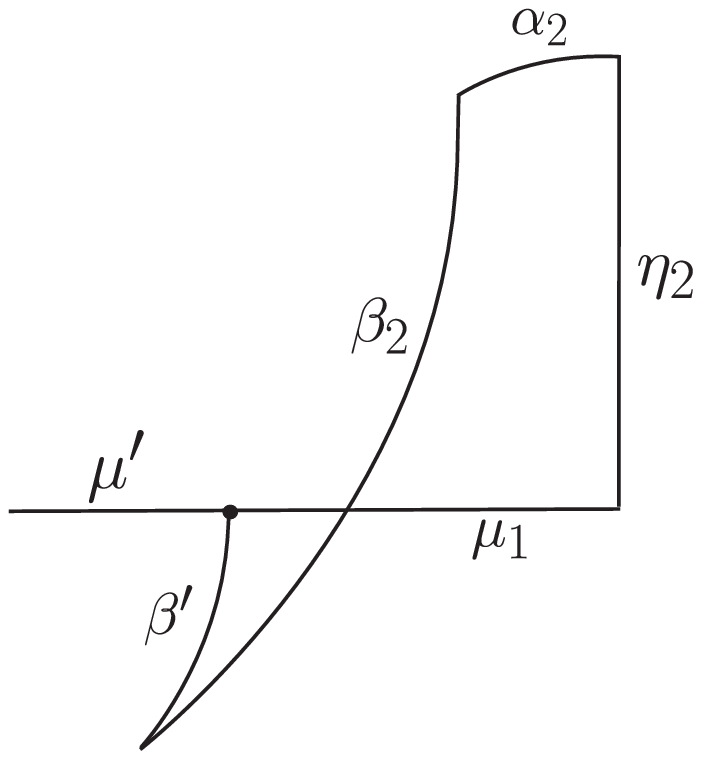}
  \caption{$(\mu_{1},\eta_{1})+(\alpha_{2},\beta_{2})\rightarrow(\mu',\beta')$}\label{RRSS-RS}
\end{center}
\end{figure}

\item $(\mu_{1},\eta_{1})+(\alpha_{2},\beta_{2})\rightarrow(\mu',\beta')$: Figure \ref{RRSS-RS}\\
  $\phantom{333333}$ $A=-\alpha_{2} \leq 0$, $B=\beta'-\beta_{2}\leq -\eta_{2}\leq 0$.

\begin{figure}
\begin{center}
  \includegraphics[height=150pt]{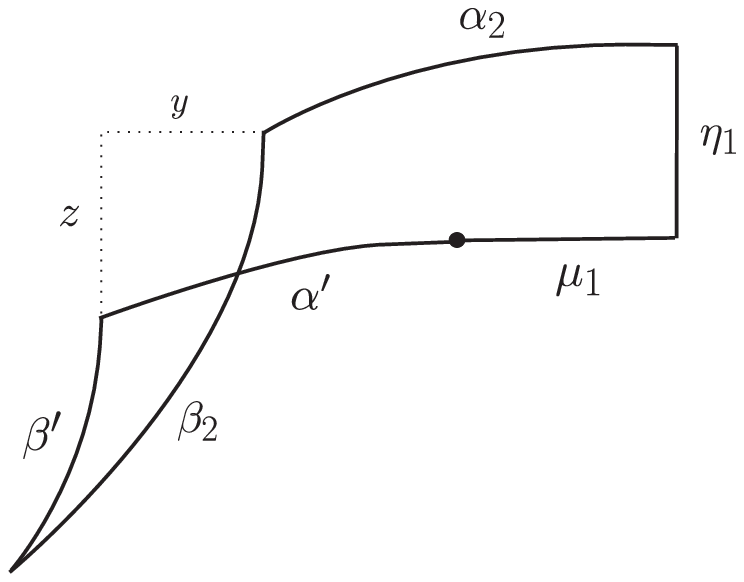}
  \includegraphics[height=150pt]{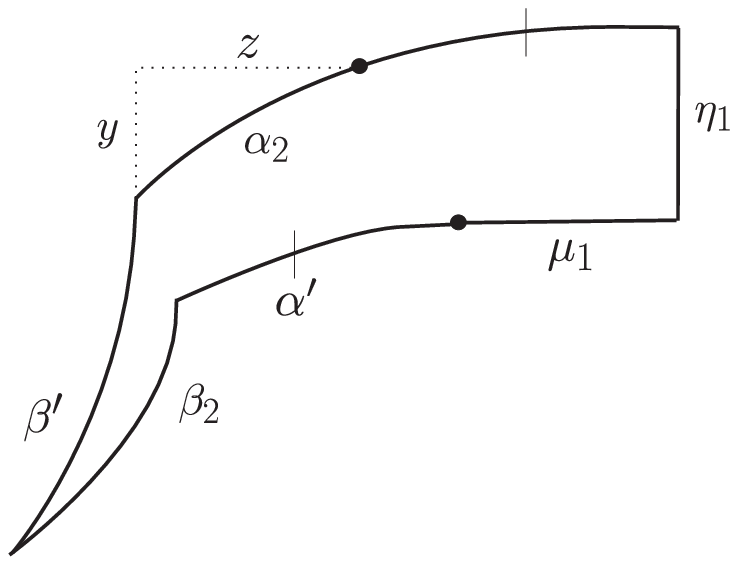}
  \caption{$(\mu_{1},\eta_{1})+(\alpha_{2},\beta_{2})\rightarrow(\alpha',\beta')$ Case $1$ and $2$.}\label{RRSS-SS}
\end{center}
\end{figure}

\item $(\mu_{1},\eta_{1})+(\alpha_{2},\beta_{2})\rightarrow(\alpha',\beta')$: Figure \ref{RRSS-SS}\\
  $\phantom{333333}$ Case $1.)$ $B=\beta'-\beta_{2}=-z$.  Since $\alpha_{2}+y=\mu_{1}+\alpha'$, we have $A=\alpha'-\alpha_{2}=y-\mu_{1}$.  Therefore, $A=\alpha'-\alpha_{2}\leq y \leq C_{0}z$.\\

  $\phantom{333333}$ Case $2.)$ $A=\alpha'-\alpha_{2}=-z$.  Since $\beta_{2}+y=\beta'+\eta_{1}$, we have $B=\beta'-\beta_{2} \leq y \leq C_{0}z$.

\end{itemize}

\vspace{1in}


\item $(\mu_{1},\eta_{1})+(\mu_{2},\beta_{2})$

\begin{figure}
\begin{center}
  \includegraphics[width=150pt]{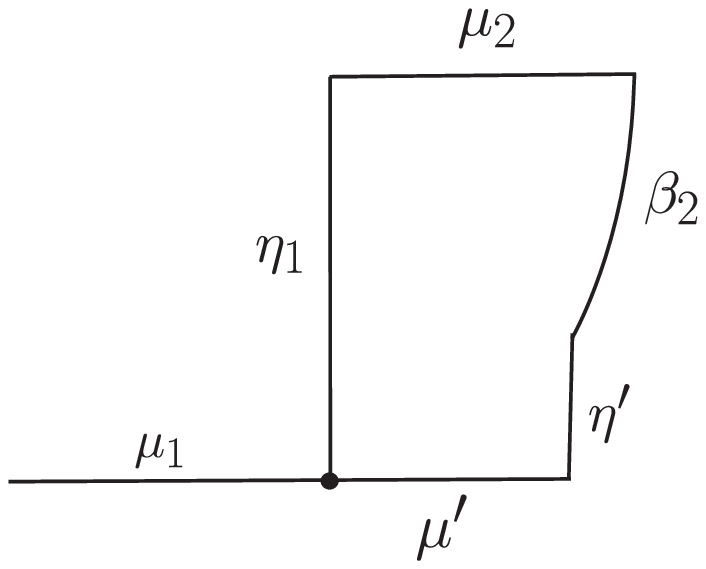}
  \caption{$(\mu_{1},\eta_{1})+(\mu_{2},\beta_{2})\rightarrow(\mu',\eta')$}\label{RRRS-RR}
\end{center}
\end{figure}

\begin{itemize}

\item $(\mu_{1},\eta_{1})+(\mu_{2},\beta_{2})\rightarrow(\mu',\eta')$: Figure \ref{RRRS-RR}\\
  $\phantom{333333}$ We have $A=0 \leq 0$ and $B=-\beta_{2} \leq 0$.

\begin{figure}
\begin{center}
  \includegraphics[width=150pt]{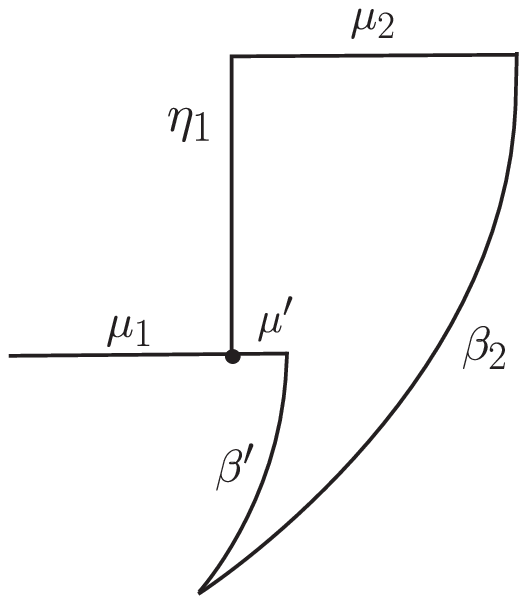}
  \caption{$(\mu_{1},\eta_{1})+(\mu_{2},\beta_{2})\rightarrow(\mu',\beta')$}\label{RRRS-RS}
\end{center}
\end{figure}

\item $(\mu_{1},\eta_{1})+(\mu_{2},\beta_{2})\rightarrow(\mu',\beta')$: Figure \ref{RRRS-RS}\\
  $\phantom{333333}$ $A=0 \leq 0$ and $B=\beta'-\beta_{2}=-\eta_{1} \leq 0$.

\begin{figure}
\begin{center}
  \includegraphics[height=150pt]{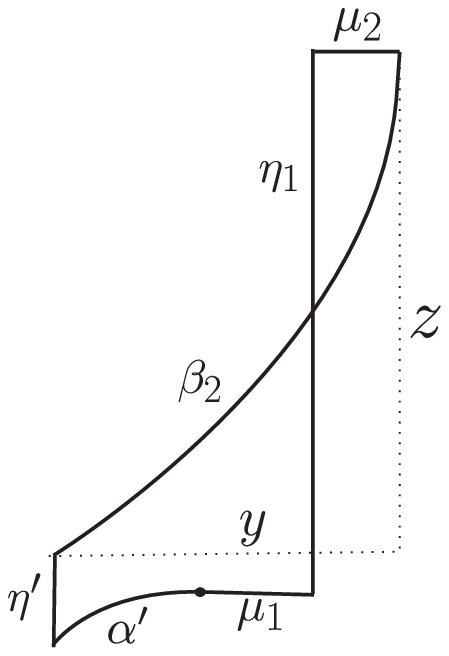}
  \caption{$(\mu_{1},\eta_{1})+(\mu_{2},\beta_{2})\rightarrow(\alpha',\eta')$}\label{RRRS-SR}
\end{center}
\end{figure}

\item $(\mu_{1},\eta_{1})+(\mu_{2},\beta_{2})\rightarrow(\alpha',\eta')$: Figure \ref{RRRS-SR}\\
  $\phantom{333333}$  $B=-\beta_{2}=-z$ and $A=\alpha' \leq y \leq C_{0}z$.

\begin{figure}
\begin{center}
  \includegraphics[height=150pt]{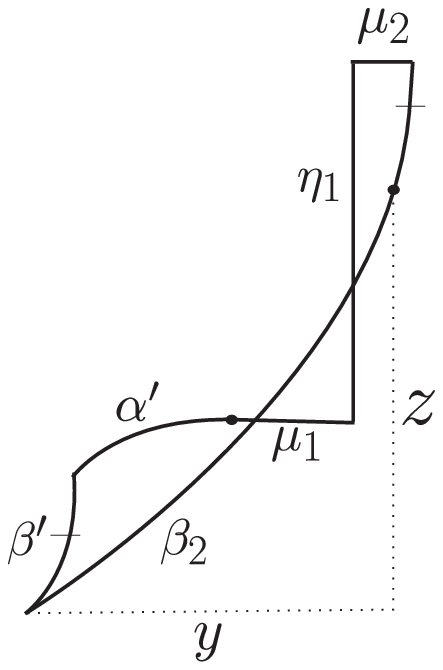}
  \caption{$(\mu_{1},\eta_{1})+(\mu_{2},\beta_{2})\rightarrow(\alpha',\beta')$}\label{RRRS-SS}
\end{center}
\end{figure}

\item $(\mu_{1},\eta_{1})+(\mu_{2},\beta_{2})\rightarrow(\alpha',\beta')$: Figure \ref{RRRS-SS}\\
  $\phantom{333333}$ $B=\beta'-\beta_{2}=-z$ and $A=\alpha'\leq y \leq C_{0}z$.

\end{itemize}


\item $(\mu_{1},\eta_{1})+(\mu_{2},\eta_{2})$

\begin{figure}
\begin{center}
  \includegraphics[width=150pt]{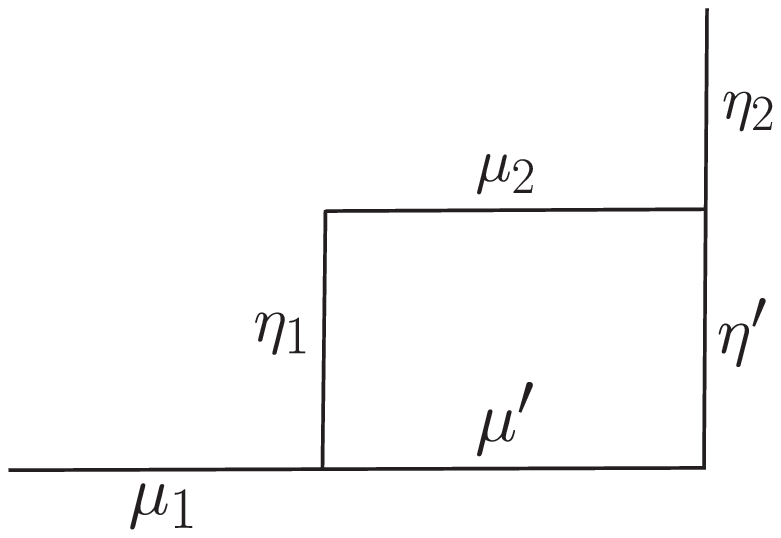}
  \caption{$(\mu_{1},\eta_{1})+(\mu_{2},\eta_{2})\rightarrow(\mu',\eta')$}\label{RRRR-RR}
\end{center}
\end{figure}

\begin{itemize}

\item $(\mu_{1},\eta_{1})+(\mu_{2},\eta_{2})\rightarrow(\mu',\eta')$: Figure \ref{RRRR-RR}\\
  $\phantom{333333}$ We have $A=0 \leq 0$ and $B=0 \leq 0$.

\end{itemize}

\vspace{1in}


\item $(\mu_{1},\eta_{1})+(\alpha_{2},\eta_{2})$

\begin{figure}
\begin{center}
  \includegraphics[width=150pt]{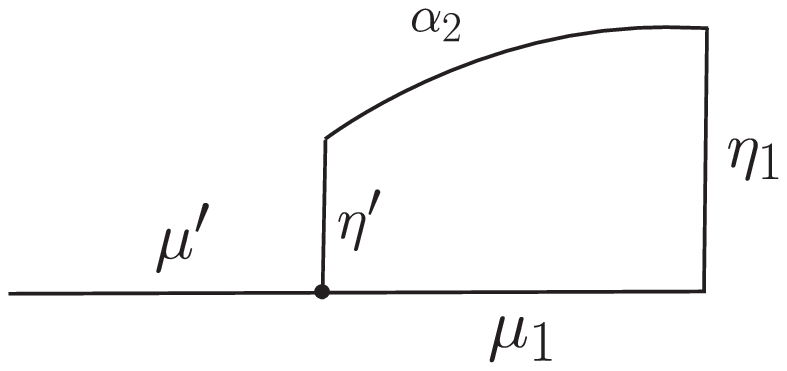}
  \caption{$(\mu_{1},\eta_{1})+(\alpha_{2},\eta_{2})\rightarrow(\mu',\eta')$}\label{RRSR-RR}
\end{center}
\end{figure}

\begin{itemize}

\item $(\mu_{1},\eta_{1})+(\alpha_{2},\eta_{2})\rightarrow(\mu',\eta')$: Figure \ref{RRSR-RR}\\
  $\phantom{333333}$ We have $A=-\alpha_{1}\leq 0$ and $B=0 \leq 0$.

\begin{figure}
\begin{center}
  \includegraphics[width=150pt]{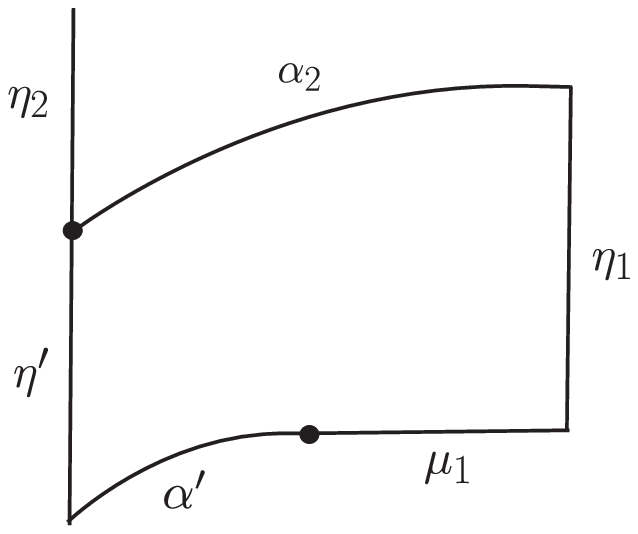}
  \caption{$(\mu_{1},\eta_{1})+(\alpha_{2},\eta_{2})\rightarrow(\alpha',\eta')$}\label{RRSR-SR}
\end{center}
\end{figure}

\item $(\mu_{1},\eta_{1})+(\alpha_{2},\eta_{2})\rightarrow(\alpha',\eta')$: Figure \ref{RRSR-SR}\\
  $\phantom{333333}$ We have $A=\alpha'-\alpha_{2}=-\mu_{1}\leq 0$ and $B=0 \leq 0$.

\begin{figure}
\begin{center}
  \includegraphics[width=170pt]{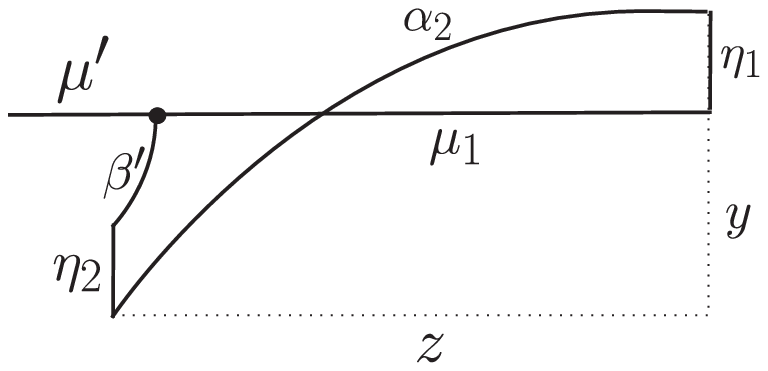}
  \caption{$(\mu_{1},\eta_{1})+(\alpha_{2},\eta_{2})\rightarrow(\mu',\beta')$}\label{RRSR-RS}
\end{center}
\end{figure}

\item $(\mu_{1},\eta_{1})+(\alpha_{2},\eta_{2})\rightarrow(\mu',\beta')$: Figure \ref{RRSR-RS}\\
  $\phantom{333333}$ We have $A=-\alpha_{2}=-z$ and $B\leq y \leq C_{0}z$.

\begin{figure}
\begin{center}
  \includegraphics[width=190pt]{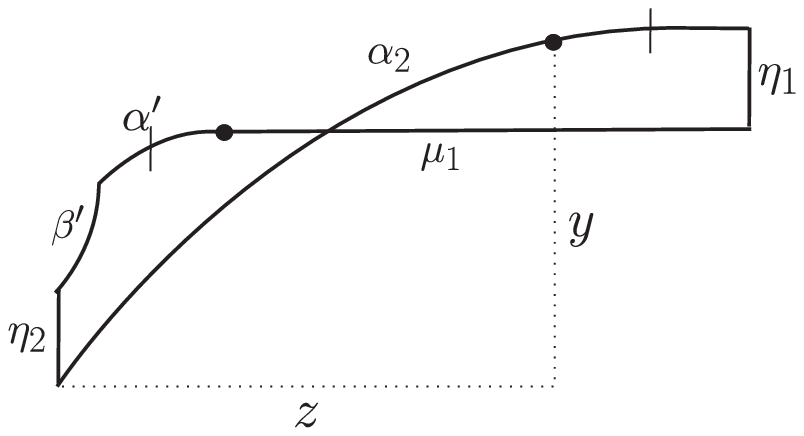}
  \caption{$(\mu_{1},\eta_{1})+(\alpha_{2},\eta_{2})\rightarrow(\alpha',\beta')$}\label{RRSR-SS}
\end{center}
\end{figure}

\item $(\mu_{1},\eta_{1})+(\alpha_{2},\eta_{2})\rightarrow(\alpha',\beta')$: Figure \ref{RRSR-SS}\\
  $\phantom{333333}$ We have $A=\alpha'-\alpha_{2}=-z$ and $B=\beta'\leq y \leq C_{0}z$.

\end{itemize}

\end{enumerate}

    \fancypagestyle{headings}{
  \lhead{}
  \fancyhfoffset[r]{.5in}
  \rhead{\thepage \skipline \skipline \skipline}
  \cfoot{}
  \renewcommand{\headrulewidth}{0in}
} \pagestyle{headings}

\end{document}